\documentclass[final,3p,times]{elsarticle}
\usepackage{amssymb}
\usepackage{amsthm,bm}
\usepackage{multirow}
\usepackage{amsmath}
\usepackage{color}
\usepackage{hyperref}
\usepackage{lineno}
\usepackage{soul}

\newcommand{\Vper}{V_\mathrm{per}(\Omega^\varepsilon)}

\newcommand{\app}{{\mathrm{app}}}

\newcommand{\pd}[2]{\frac{\partial{#1}}{\partial{#2}}}

\newcommand{\bl}{{\mathrm{bl}}}
\newcommand{\FF}{{\mathrm{ff}}}
\newcommand{\PM}{{\mathrm{pm}}}
\newcommand{\lrp}[1]{\left( #1 \right)}
\newcommand{\vecee}[1]{{\ensuremath{\mathbf{ #1}}}}
\renewcommand{\vecee}[1]{{\ensuremath{\boldsymbol{\mathrm #1}}}}

\newcommand{\ten}[1]{\ensuremath{\mathbf{#1}}}

\newcommand{\vdot}{\boldsymbol{\mathsf{\ensuremath\cdot}}}
\newcommand{\del}{\ensuremath{\nabla}}
\newcommand{\deld}{\ensuremath{\del\vdot}}
\newcommand{\RR}{\mathbb{R}}
\renewcommand{\div}{\ensuremath{\del\vdot}}
\newcommand{\divy}{\ensuremath{\del_{\vecee y}\vdot}}

\newcommand{\dvdto}{\frac{\partial}{\partial x_1} \frac{\partial v_1^\FF}{\partial x_2}\bigg|_\Sigma}
\newcommand{\dvdtoo}{\frac{\partial^2}{\partial x_1^2} \frac{\partial v_1^\FF}{\partial x_2}\bigg|_\Sigma}

\newcommand{\Omepspm}{\Omega^\varepsilon_\PM}
\newcommand{\Omeps}{\Omega^\varepsilon}

\usepackage{mathtools}
\DeclarePairedDelimiter{\norm}{\lVert}{\rVert}

\newcommand{\normphipm}{\| \vecee \varphi \|_{L^2(\Omega_\PM^\varepsilon)^2}}
\newcommand{\normgradphi}{\| \nabla \vecee \varphi \|_{L^2(\Omeps)^{2 \times 2}}}
\newcommand{\normgradphipm}{\| \nabla \vecee \varphi \|_{L^2(\Omepspm)^{2 \times 2}}}

\newcommand{\dpxij}{\frac{\partial^2 p^\PM}{\partial x_i x_j}\bigg|_\Sigma}
\newcommand{\dpxijdom}{\frac{\partial^2 p^\PM}{\partial x_i x_j}}

\newcommand{\dpxj}{\frac{\partial p^\PM}{\partial x_j}\bigg|_\Sigma}
\newcommand{\dpxjH}{\frac{\partial p^\PM}{\partial x_j}(x_1, -H)}

\newcommand{\intvel}[3]{ 
 \int_{\Omeps}  \nabla \left( {#1} {#2} {#3}  \right) \colon \nabla  \vecee \varphi}
\newcommand{\intpmvel}[3]{ 
 \int_{\Omepspm}  \nabla \left( {#1} {#2} {#3}  \right) \colon \nabla  \vecee \varphi}

\newcommand{\intpmp}[3]{ 
 \int_{\Omepspm}  \left( {#1} {#2} {#3}  \right) \div  \vecee \varphi}

\newcommand{\intp}[3]{ 
 \int_{\Omeps}  \left( {#1} {#2} {#3}  \right) \div  \vecee \varphi 
}

\newcommand{\intotim}[3]{2 \int_{\Omeps} {#1} \left( \left( {#2} \right) \otimes \nabla {#3} \right) \colon \nabla \vecee \varphi}

\newcommand{\intdel}[3]{
\int_{\Omeps} {#1}\left( \left( {#2} \right) \Delta {#3}  \right)  \vdot  \vecee \varphi
}

\newcommand{\intpff}[3]{
 \int_{\Omega_\FF} {#1}\left( \left( {#2}  \right) \pd{}{x_1} {#3} \right) \varphi_1
}

\newcommand{\intppm}[3]{
\int_{\Omepspm} {#1}\left(  {#2}  \pd{}{x_1} {#3}\right)   \varphi_1
}

\newcommand{\intpeps}[3]{
\int_{\Omeps} {#1}\left( \left( {#2}  \right) \pd{}{x_1} {#3} \right) \varphi_1
}

\usepackage{stmaryrd}
\usepackage{amsthm}

\newtheorem{theorem}{Theorem}
\newtheorem{remark}{Remark}
\newtheorem{corollary}{Corollary}

\begin{document}

\begin{frontmatter}

\title{Higher-order coupling conditions for arbitrary flows \\ in Stokes--Darcy systems}

\author[1]{Elissa Eggenweiler}\ead{elissa.eggenweiler@ians.uni-stuttgart.de}

\author[1]{Iryna Rybak\corref{cor2}}\ead{iryna.rybak@ians.uni-stuttagrt.de}

%\affiliation[1]{organization={Institute of Applied Analysis and Numerical Simulation, University of Stuttgart},
%            addressline={Pfaffenwaldring 57}, 
%            city={Stuttgart},
%            postcode={70569}, 
%            country={Germany}}
\address[1]{Institute of Applied Analysis and Numerical Simulation, 
University of Stuttgart \\ Pfaffenwaldring 57, 70569 Stuttgart, Germany}

\cortext[cor2]{Corresponding author}

\begin{abstract}
The choice of interface conditions for coupling free-flow and porous-medium flow systems is crucial in order to obtain accurate coupled flow models and precise numerical simulation results. 
Typically, the Stokes equations are considered in the free-flow region, Darcy's law is applied in the porous medium, and traditional coupling conditions (conservation of mass, balance of normal forces, the Beavers--Joseph condition on tangential velocity) are set on the interface.
However, these traditional conditions are applicable to flows parallel to the fluid--porous interface only. 
Recently, we derived generalized interface conditions accounting for arbitrary flow directions to the porous layer using homogenization and boundary layer theory.
We validated these conditions numerically and demonstrated that they are more accurate than the traditional coupling conditions. However, error estimates have not been derived yet.
In this paper, we extend the generalized coupling conditions and prove rigorous error estimates for the homogenization result.
All effective parameters appearing in the developed higher-order interface conditions are computed numerically based on the pore geometry. 
We validate the derived conditions by comparing numerical simulation results for the coupled Stokes--Darcy model and the pore-scale resolved model. 
Moreover, we compare the new coupling conditions to the traditional as well as generalized interface conditions and highlight the importance of the additional higher-order terms appearing in the derived coupling concept. 
\end{abstract}


\end{frontmatter}

%% \linenumbers

\biboptions{sort&compress}

%% main text
\section{Introduction}
\label{sec:intro}

%%%%%% Motivation %%%%%%%%%%%%
Coupled systems containing a free-flow region and a porous medium are prevalent, both in industrial applications as well as in biological and environmental settings. Examples include the gas/water management in PEM fuel cells~\cite{Gurau_Mann_2009}, %industrial filtration~\citep{Hanspal_etal_2006},
blood flows through human tissues and vessels~\cite{Smith_Humphrey_2007} and surface water/groundwater flows~\cite{Weishaupt_etal_2019}.
Pore-scale resolved simulations of such coupled flow systems are computationally demanding and often not feasible for practical applications. Therefore, macroscale models treating the coupled fluid--porous system as two different continua divided by a sharp interface are commonly used. 
Selecting appropriate coupling conditions at the interface is essential for macroscopic modeling of coupled fluid flow systems to ensure the attainment of physically consistent and reliable numerical simulation results. 
%This paper addresses this issue by rigorously deriving a novel set of coupling conditions that is suitable for arbitrary flow directions at the fluid-porous interface.

%%%%%% State of the art %%%%%%%%%%%%
In the literature, diverse model formulations are available to describe fluid flows in coupled systems involving both free-flow and porous-medium flow, with the choice depending on the particular flow regime and the specific application under consideration. In the broadest scenario, the free-flow domain employs the Navier--Stokes equations, while the subsurface flow is described using multiphase Darcy’s law~\cite{Helmig_97}. For coupled problems, particularly when viscous forces play a predominant role in fluid flow, the Navier--Stokes equations simplify to the Stokes equations. 
When the vertical length scale significantly differs from the horizontal one, the shallow water equations are utilized in the free-flow region~\cite{Vreugdenhil_1994}. 
In many applications, it is common to consider the porous medium as being completely saturated with the fluid present in the free-flow region, and the single-phase Darcy law~\cite{Darcy_1856} is applied in the porous medium. 
Alternative models describing the fluid flow in porous media are the Brinkman equation~\cite{Brinkman_47} that is used in case of high porosity ($\phi > 0.95$), the Forchheimer law~\cite{Forchheimer_1901} that describes porous-medium flow in case of high velocities, or the Richards equation~\cite{Richards_31} that is applied in case of unsaturated porous media.

%%%%%% In this work, we do %%%%%%%%%%%%
In this work, we study the Stokes--Darcy problem which is typically used for mathematical modeling, numerical analysis, as well as the development of efficient solution strategies, e.g.,~\cite{Angot_etal_17,Beavers_Joseph_67,Discacciati_GerardoGiorda_18,Eggenweiler_Rybak_MMS20,Jaeger_Mikelic_96,Lacis_Bagheri_17}.
%
%%%%%%%%%%%% What is the problem with existing IC %%%%%%%%%%%%
%%%%%%%%%%%% Classical IC %%%%%%%%%%%%
The classical conditions for coupling the Stokes and Darcy equations across the interface
are the continuity of normal velocity across the interface~\eqref{eq:IC-classical-mass}, the balance of normal forces~\eqref{eq:IC-classical-momentum} and the Beavers--Joseph condition~\eqref{eq:IC-classical-BJJ} on the tangential component of velocity
\begin{alignat}{2}
    \label{eq:IC-classical-mass}
    \vecee v^\FF \vdot \vecee n &= \vecee v^\PM \vdot \vecee n \qquad &&\text{on} 
    \;\Sigma \, ,
    \\
    \label{eq:IC-classical-momentum}
    p^\PM &=-\vecee n \vdot %\ten T (\vecee v^\FF, p^\FF) %\vdot 
    ( \nabla \vecee v^\FF - p^\FF \ten I ) 
    \vecee n 
    \qquad &&\text{on} \; \Sigma\, ,
    \\
    \label{eq:IC-classical-BJJ}(\vecee v^\FF - \vecee v^\PM ) \vdot \vecee \tau
    &= \alpha^{-1}\,\sqrt{K}
    \vecee \tau \vdot \nabla \vecee v^\FF %\vdot 
    \vecee n \qquad &&\text{on} \; \Sigma \, .
\end{alignat}
Here, $\{\vecee v^i, p^i\}$, $i \in \{\FF, \PM\}$ is the velocity and pressure describing the free flow ($\FF$) %respective
and porous-medium flow ($\PM$), $\ten I$ is the identity tensor, $\alpha>0$ is the Beavers--Joseph parameter, $K$~is a characteristic permeability~\citep{Eggenweiler_Rybak_20}
and $\vecee n$, $\vecee \tau$ are the unit normal 
and tangential vectors at the fluid--porous interface $\Sigma$ (figure~\ref{fig:setting}, left).
Note that conditions~\eqref{eq:IC-classical-mass}--\eqref{eq:IC-classical-BJJ} are provided in their non-dimensional form.
The classical set of interface conditions~\eqref{eq:IC-classical-mass}--\eqref{eq:IC-classical-BJJ} is valid for parallel flows to the porous medium~\citep{Beavers_Joseph_67,Jones_73,Nield_09,Saffman} but unsuitable for arbitrary flow directions at the fluid--porous interface~\citep{Eggenweiler_Rybak_20,Strohbeck_etal_2021,Sudhakar_etal_2021}. 
Besides the classical interface conditions, several other coupling concepts have been proposed in the literature.  
However, most of these coupling conditions are limited in their applicability, i.e., some require the determination of unknown model parameters~\citep{Angot_2010}, while others are still restricted to specific flow directions~\citep{Carraro_etal_15,Jaeger_etal_01,Jaeger_Mikelic_00}.

%%%%%%%%%%%% Alternative IC %%%%%%%%%%%%
Alternative interface conditions recently developed in~\citep{Eggenweiler_Rybak_MMS20,Sudhakar_etal_2021} are not restricted to unidirectional flows (either parallel or perpendicular) to the interface and do not contain undetermined parameters.
The conditions proposed by~\cite{Sudhakar_etal_2021} are obtained using formal multiscale homogenization and are shown to be valid for the lid-driven cavity flow over a porous bed. 
%%%%%%%%%%%% IC by us %%%%%%%%%%%%
The generalized interface conditions, derived using homogenization and boundary layer theory in our previous work~\citep{Eggenweiler_Rybak_MMS20}, are numerically validated for various flow problems with arbitrary flow direction at the fluid--porous interface~\citep{Eggenweiler_Rybak_MMS20,Strohbeck_etal_2021}. These coupling conditions in the non-dimensional form read
\begin{alignat}{2}
    \vecee v^\FF \vdot \vecee n &= \vecee v^\PM \vdot \vecee n \quad &&\text{on } \Sigma \, , \label{eq:ER-mass}
    \\
    p^\PM &= 
    -\vecee n \vdot %\ten T (\vecee v^\FF, p^\FF) %\vdot 
    ( \nabla \vecee v^\FF - p^\FF \ten I ) 
    \vecee n
    + N_s^{\bl}  \vecee \tau \vdot \nabla \vecee v^\FF 
    \vecee n \hspace{3.5mm}  &&\text{on } \Sigma \, , \label{eq:ER-momentum}
    \\
    \vecee v^\FF \vdot \vecee \tau &= 
    - \varepsilon N_1^{\bl}  \vecee  \tau \vdot 
    \nabla \vecee v^\FF
    \vecee n 
    + \varepsilon^2 \sum_{j=1}^2 \vecee M^{j,\bl} \frac{\partial p^\PM}{\partial x_j}
    \vdot \vecee \tau  \quad &&\text{on } \Sigma \, , \label{eq:ER-tangential}
\end{alignat}
where $\varepsilon$ is the characteristic pore size (figure~\ref{fig:setting}, left), $N_s^\bl$, $N_1^\bl$ and $\vecee M^{j,\bl}$, $j=1,2$ are parameters that are computed based on the geometrical configuration of the porous medium (see~\ref{appendix:cell-and-boundary-layer-problems}).
However, the justification of the homogenization result~\eqref{eq:ER-mass}--\eqref{eq:ER-tangential} through rigorous error estimates has remained an unresolved issue. We address this issue in our current work. 
%%%%%% What we do here %%%%%%%%%%%%
In this paper, we develop a novel set of coupling conditions that extends interface conditions~\eqref{eq:ER-mass}--\eqref{eq:ER-tangential} from~\citep{Eggenweiler_Rybak_MMS20} by additional higher-order terms w.r.t. the scale separation parameter. 
The newly developed coupling conditions are derived for arbitrary flow directions to the interface using periodic homogenization and boundary layer theory.
%The main contribution of this paper lies in the rigorous error estimates we establish for the developed higher-order interface conditions that justify the obtained homogenization result.
Under an assumption on the regularity and uniform boundedness of the free-flow velocity and the porous-medium pressure, we prove rigorous error estimates that justify the obtained higher-order interface conditions.
Moreover, we demonstrate the suitability of the novel set of interface conditions for arbitrary flow directions at the fluid--porous interface and highlight the significance of the additional higher-order terms.
%Since our current focus is on the rigorous derivation of coupling conditions and error estimates, a comprehensive comparison study with other existing coupling concepts is beyond the scope of this paper and will be provided in a forthcoming publication. 
%The forthcoming publication will delve into a detailed comparison with other coupling approaches, aiming to provide a comprehensive evaluation of the proposed interface conditions and their advantages in various flow scenarios.

%%%%%% Structure of paper %%%%%%%%%%%%
The paper is organized as follows. The problem setting is described in section~\ref{sec:setting}, where we present the geometrical description of the coupled system, the assumptions regarding the flow system, and the microscopic as well as macroscopic flow models. 
The main results are provided in section~\ref{sec:main-results} which include the newly developed coupling conditions and the derived error estimates.
The rigorous derivation of coupling conditions and error estimates for the model approximation is presented in section~\ref{sec:derivation}.
The developed interface conditions are validated by comparison of pore-scale resolved and macroscale numerical simulation in section~\ref{sec:numerics}. The conclusions follow in section~\ref{sec:conclusion}.

\begin{figure}
    \centering
    \includegraphics[width=0.85\textwidth]{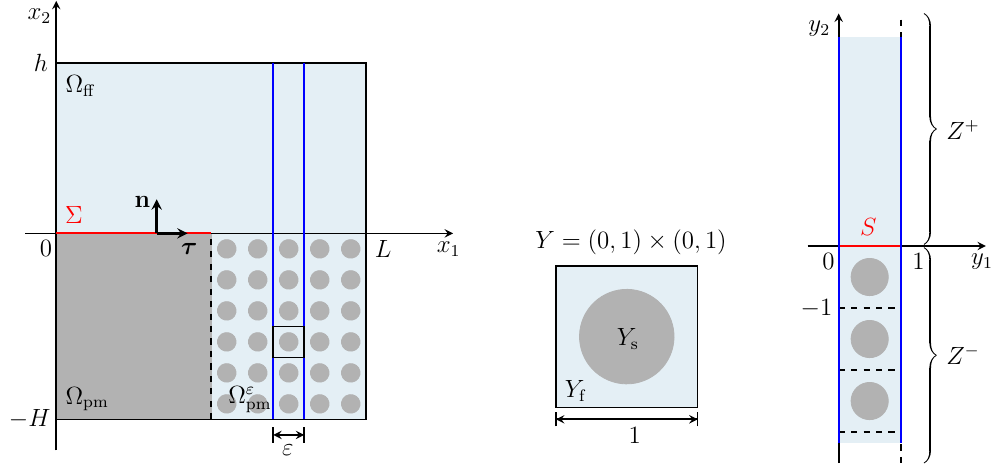}
    \caption{Coupled flow domain (left) at the macroscale (left part of domain) and at the pore scale (right part of domain), unit cell $Y=(0,1)\times (0,1)$ (middle), and the infinite long boundary layer stripe $Z^\bl = Z^+ \cup S \cup Z^-$ (right).}
    \label{fig:setting}
\end{figure}

\section{Problem setting}
\label{sec:setting}
In this section, we first introduce the geometrical setting and present the assumptions on the flow, the fluid, and the porous medium. Then, we provide the microscopic flow model %(section~\ref{subsec:microscopic}) 
and the macroscopic models in the free-flow and porous-medium domain%(section~\ref{subsec:macroscopic})
. 

\subsection{Definition of geometry and assumption on the flow system}
\label{sec:assumptions}
From the macroscale perspective, we consider the coupled domain $\RR^2\supseteq \Omega = \Omega_\FF \cup \Sigma \cup \Omega_\PM$ consisting of the free-flow region~$\Omega_\FF = (0,L) \times (0,h)$, the sharp interface~$\Sigma=(0,L) \times \{0\}$, and the porous medium~$\Omega_\PM = (0,L) \times (-H,0)$ (figure~\ref{fig:setting}, left).
The unit normal vector on the horizontal fluid--porous $\Sigma$ is $\vecee n = \vecee e_2$ and the unit tangential vector is~$\vecee \tau = \vecee e_1$.

From the pore-scale perspective, the entire flow domain~$\RR^2\supseteq \Omega^\varepsilon = \Omega_\FF \cup \Sigma \cup \Omega^\varepsilon_\PM$ comprises the free-flow region~$\Omega_\FF$, the interface $\Sigma$, and the pore space~$\Omega_\PM^\varepsilon \subset \RR^{2}$ of the porous medium. 
We consider the coupled flow region which satisfies $\frac{L}{\varepsilon}, \frac{H}{\varepsilon} \in \mathbb{N}$, where $\varepsilon$ is the characteristic pore size (figure~\ref{fig:setting}, left).
For the construction of the porous-medium domain, we follow~\cite{Hornung_97,Jaeger_Mikelic_96}.
We assume that the porous medium is built by the periodic repetition of the scaled unit cell $Y^\varepsilon = \varepsilon Y = (0,\varepsilon) \times (0,\varepsilon)$, which consists of the fluid part $Y_\mathrm{f}^\varepsilon= \varepsilon Y_\mathrm{f}$ and the solid part $Y_\mathrm{s}^\varepsilon=\varepsilon Y_\mathrm{s}$ (figure~\ref{fig:setting}, middle).

In this work, we study the steady-state, laminar ($Re \ll 1$), single-phase flow of an incompressible fluid. We assume that the fluid has constant viscosity, fully saturates the pore space $\Omepspm$ of the porous medium, and contains only one chemical species. 
Moreover, the coupled system is supposed to be isothermal, the solid obstacles are impermeable, and the porous material is non-deformable.

\subsection{Microscopic flow model}

Under the prescribed assumptions in section~\ref{sec:assumptions}, we consider the following non-dimensional Stokes equations to describe the fluid flow in the entire flow domain~$\Omega^\varepsilon$:
% \begin{eqnarray}
\begin{equation} \label{eq:pore-scale}
    \begin{aligned}
    - \Delta   \vecee{v}^\varepsilon+ \nabla p^\varepsilon &= \vecee 0 \, , \quad 
    \deld \vecee{v}^{\varepsilon} = 0 \quad \textnormal{ in } \Omega^{\varepsilon} \, ,
    \quad \ \int_{\Omega_\FF} p^\varepsilon \ \mathrm{d} \vecee x = 0 \, , 
     \\
    \vecee v^{\varepsilon} &= \vecee 0 \quad \textnormal{on } \partial \Omega^{\varepsilon} \setminus \partial \Omega \, ,    
    \quad 
    \{\vecee{v}^{\varepsilon}, p^{\varepsilon} \}  \textnormal{ is $L$-periodic in $x_1$} \, ,
    % \label{eq:pore-scale-2}
    \\[1ex]
    \vecee{v}^{\varepsilon} &= (v_1^\mathrm{in}(x_1),0)^\top \quad \textnormal{on } \Gamma_h:=(0,L) \times \{h\} %\{x_2 = h\}
    \, , 
    \\
    v_2^{\varepsilon} &= \frac{\partial v_1^\varepsilon}{\partial x_2} = 0 \hspace{5.9ex} \textnormal{on } \Gamma_H := (0,L) \times \{-H\} %\{x_2 = -H\} 
    \, ,
    %\label{eq:pore-scale-3}
    \end{aligned}
\end{equation}
where $\vecee{v}^\varepsilon=(v_1^\varepsilon, v_2^\varepsilon)^\top$ and $p^\varepsilon$ denote the fluid velocity and pressure, and $v_1^\mathrm{in} \neq 0 %\in C^\infty ((0,L))
$ is a prescribed inflow velocity. System~\eqref{eq:pore-scale} is originally formulated in~\citep{Eggenweiler_Rybak_MMS20}, on which the present work is based, and describes a coupled system where the flow direction is arbitrary at the fluid--porous interface~$\Sigma$.

\subsection{Macroscopic flow models}

At the macroscale, we model the fluid flow in the free-flow region~$\Omega_\FF$ by the Stokes equations in their non-dimensional form
% \begin{eqnarray*}
\begin{equation}
    \begin{array}{ccc}
    & - \Delta   \vecee{v}^\FF+ \nabla p^\FF = \vecee 0 \, , \quad 
    \deld \vecee{v}^{\FF} = 0 \quad \textnormal{ in } \Omega_\FF \, ,
    \quad \ \int_{\Omega_\FF} p^\FF \ \mathrm{d} \vecee x = 0 \, ,     
    \label{eq:macro-FF}
    \\[1ex]
    &
    \{\vecee{v}^{\FF}, p^{\FF} \}  \textnormal{ is $L$-periodic in $x_1$} \, ,
    \, \quad
      \vecee{v}^{\FF} = (v_1^\mathrm{in}(x_1),0)^\top \quad \textnormal{on } \Gamma_h %(0,L) \times \{h\}%\{x_2 = h\}
      \, ,  
      %\label{eq:macro-FF-BC}
    \end{array}
\end{equation}
% \end{eqnarray*}
and in the porous medium~$\Omega_\PM$ by the non-dimensional Darcy flow equations
% \begin{eqnarray*}
\begin{equation}
    \begin{array}{ccc}
    & \vecee{v}^\PM= - \ten K^\varepsilon \nabla p^\PM \, , \quad 
    \deld \vecee{v}^{\PM} = 0 \quad \textnormal{ in } \Omega_\PM \, ,
    \label{eq:macro-PM}
    \\[1ex]
    &
    p^{\PM} \textnormal{ is $L$-periodic in $x_1$} \, ,
    \, \quad
      v_2^{\PM} =0 \quad \textnormal{on } \Gamma_H %(0,L) \times \{-H\} %\{x_2 = -H\} 
      \, . 
      %\label{eq:macro-PM-BC}
        \end{array}
\end{equation}
% \end{eqnarray*}
Here, $\vecee v^i = (v_1^i, v_2^i)^\top$ and $p^i$ for $i \in \{\FF, \PM\}$ are the velocity and pressure in the free-flow region and the porous medium, $\ten K^\varepsilon = \varepsilon^2 \ten K$ is the scaled permeability tensor, and $v_1^\mathrm{in}$ is the same inflow velocity as in problem~\eqref{eq:pore-scale}. The entries of the permeability tensor~$\ten K$ are defined in the standard way, e.g.,~\cite[eq.~(1.5)]{Hornung_97}.

To obtain a complete macroscale model formulation, coupling conditions on the fluid--porous interface~$\Sigma$ need to be specified.
In this work, we derive higher-order interface conditions that are valid for arbitrary flows at the fluid--porous interface. These conditions are presented in the following section and their derivation is provided in section~\ref{sec:derivation}.

\section{Main results: Higher-order interface conditions and error estimates}
\label{sec:main-results}
In this section, we provide the main results which are the higher-order interface conditions and the rigorous error estimates for the model derived in section~\ref{sec:derivation}.

\subsection{Higher-order interface conditions}
We provide the generalized higher-order interface conditions for the Stokes--Darcy problem~\eqref{eq:macro-FF}--\eqref{eq:macro-PM} %derived in section~\ref{sec:derivation} 
considering the horizontal interface $\Sigma$ taking $\vecee n = \vecee e_2$ and $\vecee \tau =\vecee e_1$ (figure~\ref{fig:setting}, left). The conditions read
\begin{eqnarray}
    v_2^\FF     &=&  v_2^\PM  
    -\varepsilon^2 W^\bl \dvdto  %\quad \textnormal{on } \Sigma 
    \, ,
    \label{eq:IC-HO-normal}
    \\
    p^\PM &=& p^\FF - \frac{\partial v_2^\FF}{\partial x_2}\bigg|_\Sigma 
    + N_s^\bl \frac{\partial v_1^\FF}{\partial x_2}\bigg|_\Sigma
     - \varepsilon \sum_{j=1}^2 M_\omega^{j,\bl} \frac{\partial p^\PM }{\partial x_j}\bigg|_\Sigma
     %IRQ\nonumber
     %IRQ\\
     %IRQ&&
     + \varepsilon \left( L_\eta^\bl
     + E_b^\bl +  N_1^\bl \right)  \dvdto  %\quad \textnormal{on } \Sigma 
     \, ,
     \label{eq:IC-HO-p}
     \\
     v_1^\FF &=&  - \varepsilon N_1^\bl  \frac{\partial v_1^\FF}{\partial x_2}\bigg|_\Sigma 
    +\varepsilon^2 \sum_{j=1}^2 M_1^{j,\bl} \frac{\partial p^\PM }{\partial x_j}\bigg|_\Sigma 
    -\varepsilon^2 \left( E_1^\bl  +  L_1^\bl \right)  \dvdto %\quad \textnormal{on } \Sigma 
    \label{eq:IC-HO-tangential}
    \, ,
\end{eqnarray}
where~$N_1^\bl$,~$M_1^{j,\bl}$, $E_1^\bl$, $L_1^\bl$, $W^\bl$ $N_s^\bl$, $M_\omega^{j,\bl} $, $L_\eta^\bl$ and $E_b^\bl$ for $j=1,2$ are boundary layer constants defined in~\eqref{eq:BLP-constants-L},~\eqref{eq:BLP-constants-E},~\eqref{eq:BLP-constants-N},~\eqref{eq:BLP-constants-M} and~\eqref{eq:BLP-constant-W}.
All constants appearing in conditions~\eqref{eq:IC-HO-normal}--\eqref{eq:IC-HO-tangential} are computed numerically based on the pore geometry near the fluid--porous interface (see section~\ref{sec:numerics}). 
%In section~\ref{sec:numerics}, we provide the boundary layer constants and permeability for one specific pore geometry.

\begin{remark}
    Interface conditions~\eqref{eq:IC-HO-normal}--\eqref{eq:IC-HO-tangential} are of similar form as the conditions derived by~\cite{Sudhakar_etal_2021} although they used a different averaging technique to derive their coupling conditions.
\end{remark}

\subsection{Error estimates}
Coupling conditions~\eqref{eq:IC-HO-normal}--\eqref{eq:IC-HO-tangential} are derived from the approximation of the pore-scale velocity and pressure that we construct in section~\ref{sec:derivation} using homogenization and boundary layer theory.
The corresponding model approximation error is quantified by the functions $\{ \vecee U^{10,\varepsilon}, P^{10,\varepsilon} \}$ defined in~\eqref{eq:error-10-U} and~\eqref{eq:error-10-P}. The error functions describe the difference between the pore-scale solution~$\{ \vecee v^\varepsilon, p^\varepsilon\}$ and our constructed approximation.

To derive rigorous error estimates for the proposed pore-scale approximation in section~\ref{sec:derivation}, we need additional assumptions on the regularity and uniform boundedness of the Stokes velocity $\vecee v^\FF$ and the Darcy pressure $p^\PM$. These assumptions are summarized in Remark~\ref{rem}.
\begin{remark}\label{rem}
In this work we assume that the derivatives
$\partial^2 v_k^\FF/\partial x_1 \partial x_2$, $\partial^3 p^\PM/\partial x_i^2 \partial x_j$ exist %and are continuous 
for  $i,j,k=1,2$ and satisfy the following uniform bounds w.r.t. $\varepsilon>0$ for $C>0$: %and that for a fixed $\tilde \varepsilon > 0$ and for all  $0<\varepsilon < \tilde \varepsilon$  , i.e., there exists a constant $C>0$ such that it holds
\begin{eqnarray} \label{eq:assumption}
 \bigg\| \frac{\partial^2 v_k^\FF}{\partial x_1 \partial x_2} \bigg\|_{L^\infty(\Sigma)} \leq C \, ,
   \qquad \
   \bigg\| \frac{\partial^2 p^\PM}{\partial x_i \partial x_j} \bigg\|_{L^\infty(\Sigma)} \leq C \, ,
   \qquad \
   \bigg\| \frac{\partial^3 p^\PM}{\partial x_i^2 \partial x_j} \bigg\|_{L^\infty(\Omega_\PM^\varepsilon)} \leq C \, . 
\end{eqnarray}
\end{remark}
Note that similar regularity assumptions are also used for the rigorous derivation of interface conditions for the forced infiltration in~\citep{Carraro_etal_15}. % although not stated explicitly.
%Say that the assumptions are not too strong or at least that their validity will be tested etc...???

%Under the assumptions in remark~\ref{rem}, we obtain the following result.
\begin{theorem}
\label{theo}
%Let $\mathcal{B}$ be a neighborhood of the lower boundary $\{x_2 = -H\}$. 
Let us suppose the geometry as described in section~\ref{sec:setting} and let $\{\vecee U^{10,\varepsilon}, P^{10,\varepsilon}\}$ be defined in~\eqref{eq:error-10-U} and~\eqref{eq:error-10-P}. 
Moreover, we denote the extension of $\{  \vecee U^{10,\varepsilon},  P^{10,\varepsilon} \}$ defined in~$\Omeps$ to the macroscopic domain $\Omega$ by $\{ \widetilde{\vecee U}^{10,\varepsilon}, \widetilde P^{10,\varepsilon} \}$. Hereby, $\vecee U^{10,\varepsilon}$ is extended to zero in $\Omega \setminus \Omeps$ and $P^{10,\varepsilon}$ is extended in the standard way as presented in~\eqref{eq:pressure-extension}. 
Then, under the assumptions given in Remark~\ref{rem}, we have
\begin{align}
    \norm{\nabla \widetilde{\vecee U}^{10,\varepsilon}}_{L^2(\Omega)^{{2\times2}}} &\leq C \varepsilon^{5/2} \, , \label{eq:estimate-1}
    \\
    \norm{\widetilde{\vecee U}^{10,\varepsilon}}_{L^2(\Omega_\PM)^{{2}}} &\leq C \varepsilon^{7/2} \, ,
    \\
     \norm{\widetilde{\vecee U}^{10,\varepsilon}}_{L^2(\Omega_\FF)^{{2}}} 
    &\leq  C \varepsilon^3 \, , 
    \\
    \norm{\widetilde{\vecee U}^{10,\varepsilon}}_{L^2(\Sigma)^{{2}}} 
    &\leq  C \varepsilon^3 \, , 
    \\
    \norm{\widetilde P^{10,\varepsilon}}_{L^2(\Omega)} 
    &\leq  C \varepsilon^{3/2} \, .
    \label{eq:estimate-5}
 \end{align}
\end{theorem}
%Note: We do not need some exclusion on the lower boundary as in~\cite{Carraro_etal_15}, i.e., $\Omega_\PM \setminus \mathcal{O}$ since we consider the whole error function in our estimates.
\begin{proof}
    We prove estimates~\eqref{eq:estimate-1}--\eqref{eq:estimate-5} in section~\ref{sec:global-energy-estimates}.
\end{proof}

\section{Derivation of coupling conditions and error estimates}
\label{sec:derivation}
In this section, we present the derivation of the higher-order interface conditions~\eqref{eq:IC-HO-normal}--\eqref{eq:IC-HO-tangential} and the error estimates~\eqref{eq:estimate-1}--\eqref{eq:estimate-5} for the constructed model. 
The current work is an extension of~\citep{Eggenweiler_Rybak_MMS20}
% According to problem~(\ref{eq:pore-scale}) we introduce the test function space 
% \begin{eqnarray*}
% V_\mathrm{per}(\Omega^\varepsilon) = \{ &\vecee \phi \in H^1(\Omega^\varepsilon)^2: \vecee \phi = \vecee 0 \ \mathrm{on} \ \partial \Omega^\varepsilon \setminus \partial \Omega \cup \{x_2 = h\} \, , 
% \\
%  &\phi_2 = 0 \ \mathrm{on} \ \{ x_2 = -H\} \, , \vecee \phi \ \mathrm{is} \ L\mathrm{-periodic \ in } \ x_1 \, \} \, .
% \end{eqnarray*}
and uses the results presented therein.
Here, we follow a similar procedure for the derivation of coupling conditions. 
We introduce the space of test functions according to the pore-scale problem~\eqref{eq:pore-scale}:
\begin{align}\label{eq:Vperio}
\Vper = &\{ \vecee{\phi} \in H^1({\Omega^{\varepsilon}})^2: \vecee{\phi}= \vecee 0 \text{ on } \partial \Omega^{\varepsilon} \setminus \partial \Omega, \ \vecee{\phi}= \vecee 0 \text{ on } \Gamma_h, % \{ x_2 = h \}, 
%IR\notag
%IR\\
%IR& 
\ \ \phi_2=0 \text{ on } \Gamma_H, %\{ x_2 = -H \},
\ \vecee{\phi} \textnormal{ is $L$-periodic in $x_1$} \}\, .
\end{align}
We use the notation $H^1(\Omega^\varepsilon)^2$ for vector-valued functions in two space dimensions, where each component is an element of $H^1(\Omega^\varepsilon)$, and $H^1(\Omega^\varepsilon)^{{2\times2}}$ for the gradients of vector-valued functions. Similar notations are used for $L^2$-norms. 

The main goal of this section is to construct an accurate approximation of the pore-scale solution $\{\vecee v^\varepsilon, p^\varepsilon\}$ using homogenization and boundary layer theory and to derive the corresponding rigorous error estimates. 
% , i.e., such that the approximation error %$\{\vecee v^\varepsilon - \vecee v^\varepsilon_\app, p^\varepsilon - p^\varepsilon_\app\}$ 
% is sufficiently small, which is ensured by rigorous error estimates.
The latter are up to now not proven for the approximation proposed in~\citep{Eggenweiler_Rybak_MMS20}. 
To obtain such estimates additional higher-order terms w.r.t.~$\varepsilon$ are needed in the pore-scale approximation.
Thus, the aim of this work is to extend the approximation derived in~\citep{Eggenweiler_Rybak_MMS20} %by additional higher-order terms 
such that rigorous error estimates for the new model approximation can be obtained.

In the following, we first provide the pore-scale approximation from~\citep{Eggenweiler_Rybak_MMS20} %on which we base our work 
and explain why additional higher-order terms w.r.t. parameter $\varepsilon$ are needed %in this approximation 
(section~\ref{sec:approx-from-MMS}). Then, we construct additional boundary layer correctors that will appear in our new pore-scale approximation, and explain how interface conditions~\eqref{eq:IC-HO-normal}--\eqref{eq:IC-HO-tangential} are derived from the constructed approximation (sections~\ref{sec:additional-BLP}--\ref{sec:compressibility}). 
Finally, we derive estimates for the pressure error function (section~\ref{sec:pressure-estimate}), and prove Theorem~\ref{theo} %, and explain how interface conditions~\eqref{eq:IC-HO-normal}--\eqref{eq:IC-HO-tangential} are derived from the constructed approximation of the pore-scale solution 
(section~\ref{sec:global-energy-estimates}).

\subsection{Pore-scale approximation from~\cite{Eggenweiler_Rybak_MMS20}}
\label{sec:approx-from-MMS}
In this section, we present the pore-scale approximation proposed in~\citep{Eggenweiler_Rybak_MMS20} for which rigorous error estimates are not obtained.
% For this approximation, rigorous error estimates are not obtained up to now. Investigations in doing so suggest that additional higher-order terms in the pore-scale approximations are needed.
% Thus, the aim of this work is the extension of the pore-scale approximation derived in~\cite{Eggenweiler_Rybak_MMS20} by additional higher-order terms such that rigorous error estimates for the model approximation are obtained.
%
% Careful examination of our previous work and the results from~\cite{Carraro_etal_15} lead us to the conclusion that the solution $\{\vecee v^\mathrm{cf}, p^\mathrm{cf}\}$ to problem~(3.44),~(3.45) in~\citep{Eggenweiler_Rybak_MMS20} should better not be used in the pore-scale approximation. %~\cite[section 3.2.6]{Eggenweiler_Rybak_MMS20} of the pore-scale solution $\{\vecee v^\varepsilon, p^\varepsilon\}$. % since this leads to an undesirable contribution of the velocity error function on the top boundary. 
% 
After a thorough analysis of our previous work and the findings presented in~\citep{Carraro_etal_15} we came to the conclusion that employing the solution $\{\vecee v^\mathrm{cf}, p^\mathrm{cf}\}$ to~\citep[eqs.~(3.44),~(3.45)]{Eggenweiler_Rybak_MMS20} for the pore-scale approximation is not necessary.
It leads to an undesirable contribution of the velocity approximation %resulting velocity error function 
on the top boundary $\Gamma_h$ that we would need to correct. 
However, the required corrector would then lead to the same problems that were present before the functions $\vecee v^\mathrm{cf}$ and $ p^\mathrm{cf}$ were considered in the pore-scale approximation.
% This correction would require a correction that would lead to the same effects that were present before adding $\{\vecee v^\mathrm{cf}, p^\mathrm{cf}\}$. 
Thus, instead of the original approximations of pore-scale velocity and pressure from~\cite[section 3.2.6]{Eggenweiler_Rybak_MMS20}, we start our work from slightly modified versions, that do not include~$\{\vecee v^\mathrm{cf}, p^\mathrm{cf}\}$:
\begingroup
\allowdisplaybreaks
\begin{align}
    \vecee{v}^{6,\varepsilon}_{\app} = \, & \mathcal{H}(x_2)\vecee{v}^\FF - \mathcal{H}(-x_2) \varepsilon^2 \sum_{j=1}^2 \vecee{w}^{j, \varepsilon} \frac{\partial p^\PM }{\partial x_j}  \notag
    - \varepsilon \left(\vecee t^{\bl,\varepsilon} - \mathcal{H}(x_2) \vecee N^{\bl} \right) \frac{\partial v_1^\FF}{\partial x_2}\bigg|_\Sigma 
    %IR\\
    %IR& 
    + \varepsilon^2 \sum_{j=1}^2 \left( \vecee \beta^{j,\bl,\varepsilon} - \mathcal{H}(x_2) \vecee M^{j,\bl}\right) \frac{\partial p^\PM }{\partial x_j}\bigg|_\Sigma \notag
    \\
    &
    - \varepsilon^2 \sum_{j=1}^2  \frac{\partial p^\PM }{\partial x_j}(x_1, -H) 
    %\vecee q^{j,\bl}\left(\frac{x_1}{\varepsilon}, -\frac{x_2+H}{\varepsilon} \right) 
    \vecee q^{j,\bl,\varepsilon} 
    + \varepsilon^2 %\vecee \zeta^{\bl}(\vecee y)
    \vecee \zeta^{\bl,\varepsilon} \frac{\partial}{\partial x_1} \frac{\partial v_1^\FF}{\partial x_2}\bigg|_\Sigma 
    %IR\notag
    %IR\\
    %IR&
    + \mathcal{H}(-x_2)\varepsilon^2  \sum_{i,j=1}^2 \vecee \gamma^{j,i,\varepsilon} \frac{\partial^2 p^\PM}{\partial x_i x_j} \notag
    \\
    &+ \varepsilon^2  \sum_{i,j=1}^2   \left( \vecee \gamma^{j,i,\bl,\varepsilon}  -  \mathcal{H}(x_2)  \varepsilon \vecee  C^{j,i,\bl} \right)\frac{\partial^2 p^\PM}{\partial x_i x_j}\bigg|_\Sigma
     %\notag
     %\\
     %&
     +\varepsilon^2 \sum_{j=1}^2 \frac{\partial}{\partial x_1} \frac{\partial p^\PM}{\partial x_j}(x_1,-H) \left( 
     \vecee Z^{j,\bl,\varepsilon} + \varepsilon R^\varepsilon(\vecee e_2) \int_{Z^-} \!q_1^{j,\bl}  \text{d} \vecee y \right)
     \, ,
     \label{eq:approx-6-MMS-v}
    \\
    p^{6,\varepsilon}_{\app} = \, & \mathcal{H}(x_2)p^\FF 
    + \mathcal{H}(-x_2) \bigg( p^\PM - \varepsilon \sum_{j=1}^2 \pi^{j,\varepsilon} \frac{\partial p^\PM}{\partial x_j} \bigg)
    -  \left(s^{\bl,\varepsilon} - \mathcal{H}(x_2) N_s^{\bl} \right) \frac{\partial v_1^\FF}{\partial x_2}\bigg|_\Sigma 
    \notag
    \\
    &
    +\varepsilon \sum_{j=1}^2 \left( \omega^{j,\bl, \varepsilon} - \mathcal{H}(x_2)M_\omega^{j,\bl} \right) \frac{\partial p^\PM }{\partial x_j}\bigg|_\Sigma %IR\notag
    %IR\\
    %IR&
    -\varepsilon \sum_{j=1}^2 \frac{\partial p^\PM }{\partial x_j}(x_1, -H) % z^{j,\bl}\left(\frac{x_1}{\varepsilon}, -\frac{x_2+H}{\varepsilon} \right) 
    z^{j,\bl,\varepsilon}
    % \notag
    % \\
    % &
    +\varepsilon^2 \sum_{i,j=1}^2 \left( \pi^{j,i,\bl,\varepsilon} - C_\pi^{j,i,\bl} \right)\frac{\partial^2 p^\PM}{\partial x_i x_j}\bigg|_\Sigma \, .
    \label{eq:approx-6-MMS-p}
\end{align}
\endgroup
% Here, $\mathcal{H}$ denotes the Heaviside step function, $\vecee w^{j,\varepsilon}(\vecee x)= \vecee w^j(\frac{\vecee x}{\varepsilon})$ and $\pi^{j,\varepsilon}(\vecee x)= \pi^j(\frac{\vecee x}{\varepsilon})$ for $\vecee x \in \Omega_\PM^\varepsilon$, where $\{\vecee w^j, \pi^j\}$ are the solutions to the cell problems~\eqref{eq:cell-problems} for $j=1,2$.
Here, we denote the Heaviside step function by $\mathcal{H}$, and we set $\vecee w^{j,\varepsilon}(\vecee x)= \vecee w^j(\frac{\vecee x}{\varepsilon})$,~$\pi^{j,\varepsilon}(\vecee x)= \pi^j(\frac{\vecee x}{\varepsilon})$ for $\vecee x \in \Omega_\PM^\varepsilon$, where $\{\vecee w^j, \pi^j\}$ are the solutions to the cell problems~\eqref{eq:cell-problems} for $j=1,2$.
We define the functions $\vecee t^{\bl,\varepsilon} (\vecee x)= \vecee t^{\bl} (\frac{\vecee x}{\varepsilon})$ and $s^{\bl,\varepsilon} (\vecee x)= s^{\bl} (\frac{\vecee x}{\varepsilon})$ for $\vecee x \in \Omega^\varepsilon$ that are the  solutions to problem~\eqref{eq:BLP-t} extended to the pore-scale flow domain.
The constants $\vecee N^{\bl}$ and~$N_s^{\bl}$ are boundary layer constants defined by~\eqref{eq:BLP-constants-N}.
In an analogous way, $\{\vecee \beta^{j,\bl,\varepsilon}, \omega^{j,\bl,\varepsilon}\}$ is obtained from $\{\vecee \beta^{j,\bl}, \omega^{j,\bl}\}$ which is the solution to~\eqref{eq:BLP-beta}, and the boundary layer constants $\vecee M^{j,\bl}$ and $M_\omega^{j,\bl}$ are introduced in~\eqref{eq:BLP-constants-M} for $j=1,2$.
Moreover, for $\vecee x \in \Omega^\varepsilon$ we use following notation 
\[
\vecee q^{j,\bl,\varepsilon}(\vecee x) = \vecee q^{j,\bl}\left(\frac{x_1}{\varepsilon}, -\frac{x_2+H}{\varepsilon} \right)
\, , \quad
z^{j,\bl,\varepsilon}(\vecee x) = z^{j,\bl}\left(\frac{x_1}{\varepsilon}, -\frac{x_2+H}{\varepsilon} \right) \, , 
\]
where $\{\vecee q^{j,\bl}, z^{j,\bl}\}$ is the solution to~\eqref{eq:BLP-q}.
The boundary layer velocity $\vecee \zeta^{\bl}$ is given by~\eqref{eq:BLP-zeta} and we set $\vecee \zeta^{\bl,\varepsilon}(\vecee x) = \vecee \zeta^{\bl}\left(\frac{\vecee x}{\varepsilon}\right) $ in $\Omega^\varepsilon$.
The auxiliary function $\vecee \gamma^{j,i}$ is the solution to~\eqref{eq:BLP-gamma} and for $\vecee x \in \Omega_\PM^\varepsilon$ we define $\gamma^{j,i,\varepsilon} (\vecee x) = \varepsilon \vecee \gamma^{j,i} \left(\frac{\vecee x}{\varepsilon}\right)$.
The boundary layer correctors $\vecee \gamma^{j,i,\bl,\varepsilon}(\vecee x) = \varepsilon \vecee \gamma^{j,i,\bl} (\frac{\vecee x}{\varepsilon})$ and $\pi^{j,i,\bl,\varepsilon}(\vecee x) = \pi^{j,i,\bl}(\frac{\vecee x}{\varepsilon})$, $\vecee x \in \Omega^\varepsilon$, are defined by the solutions to boundary layer problem~\eqref{eq:BLP-gamma}, and the corresponding boundary layer constants $\vecee C^{j,i,\bl}$ and $C_\pi^{j,i,\bl}$ are given by~\eqref{eq:BLP-constants-gamma}.
Finally, $R^\varepsilon$ denotes the restriction operator defined in, e.g.~\cite[section~4.5]{Carraro_etal_15}, and we introduce the velocity function 
\[
\vecee Z^{j,\bl,\varepsilon}(\vecee x) = \varepsilon \vecee Z^{j,\bl}\left(\frac{x_1}{\varepsilon}, -\frac{x_2+H}{\varepsilon}\right)\, , \quad \text{for } \vecee x \in \Omega^\varepsilon \, ,
\]
where $\vecee Z^{j,\bl}$ is the solution to~\eqref{eq:BLP-Z}.

We introduce the velocity and pressure error functions in a standard way
\begin{align}
    \vecee U^{6,\varepsilon}= \vecee v^\varepsilon - \vecee v^{6,\varepsilon}_{\app} \, , \qquad 
    P^{6,\varepsilon} = p^\varepsilon - p^{6,\varepsilon}_{\app} \, .
\end{align}
%
% While working on the derivation of estimates for the model approximation error~$\{\vecee U^{6,\varepsilon}, P^{6,\varepsilon}\}$ proposed in~\citep{Eggenweiler_Rybak_MMS20} we found out that the estimates which can be obtained for~$\{\vecee U^{6,\varepsilon}, P^{6,\varepsilon}\}$ are not accurate enough. 
In the process of deriving estimates for the model approximation error~$\{\vecee U^{6,\varepsilon}, P^{6,\varepsilon}\}$ proposed in~\citep{Eggenweiler_Rybak_MMS20}, we found that the obtained estimates for~$\{\vecee U^{6,\varepsilon}, P^{6,\varepsilon}\}$ do not meet the desired level of accuracy.
% Since $\vecee v^{6,\varepsilon}_\app$ is of order $\varepsilon^2$, we expect that for the velocity error it holds  $\|\vecee U^{6,\varepsilon}\|_{L^2(\Omega^\varepsilon)^2} \leq C\varepsilon^i$ with $i>2$. However, we realized that such an estimate is not possible when the approximation~$\{\vecee v^{6,\varepsilon}_\app, p^{6,\varepsilon}_\app\}$ is used. Thus, we need to improve this approximation.
% Careful examination of~\cite{Carraro_etal_15,Eggenweiler_Rybak_MMS20} leads us to the following conclusion. For better error estimates we need to improve the result from~\cite[Corollary 3.5]{Eggenweiler_Rybak_MMS20}: 
% instead of the factors  $\varepsilon^{3/2}$ and $\varepsilon^{5/2}$ we need $\varepsilon^{5/2}$ and $\varepsilon^{7/2}$, respectively. 
% %In this way, we were able to obtain $\|\vecee U^{\varepsilon}\|_{L^2(\Omega^\varepsilon)^2} \leq C\varepsilon^3$. 
%
The reason for that is explained in the following.
If we consider the approximation~\eqref{eq:approx-6-MMS-v} of pore-scale velocity up to order $\varepsilon^2$  and the approximation~\eqref{eq:approx-6-MMS-p} of pore-scale pressure up to order $\varepsilon^1$, for the corresponding error functions $\{\vecee U^{6,\varepsilon}, P^{6,\varepsilon} \}$ it should hold
\begin{equation*}
\begin{aligned}
    \| \vecee U^{6,\varepsilon} \|_{L^2(\Omeps)^2} 
    %=\| \vecee v^\varepsilon - \vecee v^{6,\varepsilon}_\app \|_{L^2(\Omeps)^2} 
    \leq C\varepsilon^i \, , \quad i>2 \,  ,
    \qquad
    \| P^{6,\varepsilon} \|_{L^2(\Omeps)} 
    %=\| p^\varepsilon - p^{6,\varepsilon}_\app \|_{L^2(\Omeps)} 
    &\leq C\varepsilon^j \, , \quad j>1 \, .
\end{aligned}
\end{equation*}
However, we found that such estimates are only possible if in~\cite[Corollary 3.4]{Eggenweiler_Rybak_MMS20} the factor $\varepsilon^{3/2}$ is replaced by $\varepsilon^{5/2}$, i.e., if
for all  $\vecee \varphi \in \Vper$ the following inequality holds true
\begin{align}\label{eq:cor:4}
    \bigg| \int_{\Omega^\varepsilon} \nabla \vecee U^{6,\varepsilon} & \colon \nabla  \vecee  \varphi - \int_{\Omega^\varepsilon} P^{6,\varepsilon} \div \vecee \varphi 
    \bigg| 
    \leq C\varepsilon^{5/2} \norm{\nabla \vecee \varphi}_{L^2(\Omega^\varepsilon)^{{2\times2}}}
    \, . 
\end{align}
However, this inequality is not true for the approximations~\eqref{eq:approx-6-MMS-v},~\eqref{eq:approx-6-MMS-p}. Thus, we need to extend the pore-scale approximations~\eqref{eq:approx-6-MMS-v},~\eqref{eq:approx-6-MMS-p} by additional higher-order terms w.r.t. $\varepsilon$ in such a way that for the new error functions, inequality~\eqref{eq:cor:4} is fulfilled.
%The next step is then to estimate the pressure error function through the velocity error and finally to obtain the estimate for the velocity error.

In order to identify the terms of low order, i.e., terms on the right-hand side of~\eqref{eq:cor:4} that cannot be estimated at least by $C \varepsilon^{5/2}$, we write the weak formulation w.r.t.~$\{\vecee U^{6,\varepsilon}, P^{6,\varepsilon} \}$ as follows
\begingroup
\allowdisplaybreaks
\begin{align}
    \int_{\Omega^\varepsilon}&   \nabla \vecee U^{6,\varepsilon} \colon \nabla  \vecee  \varphi
    - \int_{\Omega^\varepsilon} P^{6,\varepsilon} \div \vecee \varphi
    - \int_{\Sigma} \underbrace{p^\PM}_{=\colon I^1}    \varphi_2 
    - \int_{\Sigma} \underbrace{\left(
    \frac{\partial v_2^\FF}{\partial x_2} 
     - p^\FF  \right)}_{=\colon I^2} \varphi_2 
    + \int_{\Sigma} \underbrace{ N_s^{\bl}  \frac{\partial v_1^\FF}{\partial x_2} \bigg|_\Sigma }_{=\colon -I^3} \varphi_2 
    - \int_\Sigma \underbrace{\varepsilon\sum_{j=1}^2 M_\omega^{j,\bl} \dpxj }_{=\colon I^4}  \varphi_2
    \notag 
    \\
    %%%%%% line 2...
     &
    -\int_{\Sigma} \underbrace{ \big\llbracket \varepsilon^2 \nabla \vecee \zeta^{\bl,\varepsilon} \vecee e_2 \rrbracket_\Sigma \vdot \vecee e_2 \dvdto }_{=\colon I^5} \varphi_2 
    + \int_{\Omega_\PM^\varepsilon}   \underbrace{ \sum_{j=1}^2 \bigg(  2\varepsilon^2 \nabla \vecee{w}^{j, \varepsilon}   - \varepsilon \pi^{j,\varepsilon}\ten I \bigg)\nabla \frac{\partial p^\PM}{\partial x_j}}_{=\colon \vecee C_\varepsilon}  \vdot \vecee  \varphi 
    \notag 
    \\
    %%%%%% line 3 and 4...
    &+ \int_{ \Omeps}  \underbrace{
    \varepsilon \left( \left(\vecee t^{\bl,\varepsilon} - \mathcal{H}(x_2) \vecee N^{\bl} \right) \frac{\partial^2 }{\partial x_1^2}\frac{\partial v_1^\FF}{\partial x_2}\bigg|_\Sigma
    \right)}_{=\colon \vecee D_\varepsilon} \vdot \vecee \varphi  
    %\notag
    %\\
    %&
    - \int_{\Omeps}  \underbrace{ \left(-  2 \varepsilon   \frac{\partial}{\partial x_1} \vecee t^{\bl,\varepsilon}
    +  s^{\bl,\varepsilon}\vecee e_1 - \mathcal{H}(x_2) N_s^{\bl} \vecee e_1 \right) \frac{\partial}{\partial x_1} \frac{\partial v_1^\FF}{\partial x_2}\bigg|_\Sigma }_{=\colon \vecee E_\varepsilon} \vdot \vecee \varphi 
    \notag
    \\
    %%%%%% line 5...
    & - \int_{ \Omepspm}  \underbrace{\varepsilon^2 \sum_{i,j=1}^2  \Delta \vecee \gamma^{j,i,\varepsilon} \dpxij}_{=\colon \vecee F_\varepsilon} \vdot \vecee \varphi 
    - \int_{\Omeps}   \underbrace{
    %\underbrace{
    \varepsilon^2 \Delta \vecee \zeta^{\bl,\varepsilon}%}_{=\Delta_{\vecee y} \vecee \zeta^{\bl}}  
    \dvdto }_{=\colon \vecee G_\varepsilon} \vdot \vecee \varphi 
    \notag 
    \\ 
    %%%%%% line 6...
    %%%%% here equality sign %%%%
    =& \int_{\Sigma}  \varepsilon^2 \sum_{j=1}^2 \left( \left(  \vecee{w}^{j, \varepsilon} \! \otimes \nabla \frac{\partial p^\PM }{\partial x_j} \right) \vecee e_2 \right) \vdot \vecee  \varphi
    -  \int\limits_{\{x_2 = -H\}}  \varepsilon^2 \sum_{j=1}^2  \left( \left(\vecee{w}^{j, \varepsilon} \! \otimes \nabla \frac{\partial p^\PM }{\partial x_j} \right) \vecee e_2 \right) \vdot \vecee  \varphi
    - \int_{\Omega_\PM^\varepsilon}   \sum_{j=1}^2  \varepsilon^2\vecee{w}^{j, \varepsilon} \Delta \frac{\partial p^\PM }{\partial x_j}  \vdot \vecee  \varphi
    \notag
    \\
    %%%line 7...
    & - \intdel{\varepsilon^2\sum_{j=1}^2}{\beta^{j,\bl,\varepsilon} - \mathcal{H} (x_2)\vecee M^{j,\bl} }{\dpxj} 
    -\intotim{\varepsilon^2\sum_{j=1}^2}{\beta^{j,\bl,\varepsilon} - \mathcal{H} (x_2)\vecee M^{j,\bl}}{\dpxj} \notag
    \\
    %%%%%% line 9...
    & 
    - \intpeps{\varepsilon\sum_{j=1}^2}{\omega^{j,\bl,\varepsilon} - \mathcal{H}(x_2)M_\omega^{j,\bl} }{\dpxj}
    +\int_{\Omepspm} \underbrace{ \varepsilon^2 \sum_{j=1}^2 \left( \vecee q^{j,\bl,\varepsilon} \Delta \dpxjH
    \right)}_{=\colon \vecee M_\varepsilon} \vdot \vecee \varphi 
     \notag
    \\
    %%%%%% line 10...
    & 
    +\int_{\Omepspm} \underbrace{ \varepsilon \sum_{j=1}^2 \left( \nabla  \dpxjH z^{j,\bl,\varepsilon}
    \right)}_{=\colon \vecee N_\varepsilon} \vdot \vecee \varphi
    + 2 \int_{\Omepspm} \underbrace{\varepsilon^2 \sum_{j=1}^2 \left( \vecee q^{j,\bl,\varepsilon} \otimes \nabla \dpxjH
    \right) }_{=\colon \ten L_\varepsilon } \colon \nabla \vecee \varphi
    \notag
    \\
    &
    +\int_{\Gamma_H}  \varepsilon^2 \sum_{i,j=1}^2  \pd{}{x_2} \gamma_2^{j,i,\varepsilon} \dpxij    \varphi_2 \notag 
    % \\
    % & \quad + 
    -\int_{ \Omepspm}  \varepsilon^2 \sum_{i,j=1}^2 %\left(  \Delta \vecee \gamma^{j,i,\varepsilon} \dpxij  - 
    \vecee \gamma^{j,i,\varepsilon} \Delta  \dpxij  %\right) 
    \vdot \vecee \varphi 
    - 2 \int_{\Omepspm} \varepsilon^2 \sum_{i,j=1}^2\left( \vecee \gamma^{j,i,\varepsilon} \otimes \nabla \dpxij \right) \colon \nabla \vecee \varphi
    \notag 
    \\
    & -\intdel{\varepsilon^2\sum_{i,j=1}^2}{\vecee \gamma^{j,i,\bl,\varepsilon} - \varepsilon\mathcal{H} (x_2)\vecee C^{j,i,\bl} }{\dpxij}  
    -\intotim{\varepsilon^2\sum_{i,j=1}^2}{\vecee \gamma^{j,i,\bl,\varepsilon} - \varepsilon \mathcal{H} (x_2)\vecee C^{j,i,\bl}}{\dpxij} \notag
    \\
    & 
    - \intpeps{\varepsilon^2\sum_{i,j=1}^2}{\pi^{j,i,\bl,\varepsilon} -  C_\pi^{j,i,\bl}}{\dpxij}
    - \int_{\Omeps}  \varepsilon^2 \vecee \zeta^{\bl,\varepsilon} \Delta \dvdto  \vdot \vecee \varphi 
    - 2 \int_{\Omeps}   \varepsilon^2  \left( \vecee \zeta^{\bl,\varepsilon} \otimes \nabla \dvdto  \right) \colon \nabla \vecee \varphi \notag 
    \\
    & - \int_{\Omega^\varepsilon}   \varepsilon^2  \sum_{j=1}^2 \nabla \left( 
    \frac{\partial}{\partial x_1} \frac{\partial p^\PM}{\partial x_j}(x_1,-H) \left( \vecee Z^{j,\bl,\varepsilon} + \varepsilon R^\varepsilon(\vecee e_2) \int_{Z^-} q_1^{j,\bl} \ \text{d} \vecee y \right) \right) \colon \nabla  \vecee  \varphi
    \notag 
    \\
    &+ \text{exponentially small terms} \, .
 \label{eq:weak-form-approx-6}
\end{align}
\endgroup
Note that equation~\eqref{eq:weak-form-approx-6} is already obtained in~\cite[eq. (3.53)]{Eggenweiler_Rybak_MMS20} with some additional terms on the right-hand side due to the corrector~$\{\vecee v^\mathrm{cf}, p^\mathrm{cf}\}$.
In this work, we put the integral terms including $I^4$, $I^5$, $\vecee C_\varepsilon$, $\vecee D_\varepsilon$,  $\vecee E_\varepsilon$,  $\vecee F_\varepsilon$ and  $\vecee G_\varepsilon$,  originally appearing on the right-hand side of the equation, to the left-hand side of~\eqref{eq:weak-form-approx-6}.
The terms $\int_\Sigma I^i \varphi_2$, $i=1,\ldots,5$, on the left-hand side of~\eqref{eq:weak-form-approx-6} will later vanish due to the new interface condition~\eqref{eq:IC-HO-p} for the pressure (see also equation~\eqref{eq:set-pressure-IC}). Exponentially small terms appearing on the right-hand side of~\eqref{eq:weak-form-approx-6} include integrals of boundary layer velocities and their gradients over the lower boundary $\Gamma_H$ as in~\cite[section 3.2.2]{Eggenweiler_Rybak_MMS20}. This notation will be used for such integrals throughout the manuscript.
The integral terms including $\vecee C_\varepsilon$, $\vecee D_\varepsilon$, $\vecee E_\varepsilon$, $\vecee F_\varepsilon$ and $\vecee G_\varepsilon$ are terms of low order w.r.t. $\varepsilon$ since it holds
\begin{align}
    \bigg| \int_{\Omepspm} \vecee C_\varepsilon \vdot \vecee \varphi \bigg|
    %= 
    % \bigg| \int_{\Omepspm} \sum_{j=1}^2 \bigg(  2\varepsilon^2 \nabla \vecee{w}^{j, \varepsilon}   - \varepsilon \pi^{j,\varepsilon}\ten I \bigg)\nabla \frac{\partial p^\PM}{\partial x_j} \vdot \vecee \varphi \bigg|
    % = 
    % \bigg| \int_{\Omepspm} \sum_{j=1}^2 \varepsilon  \bigg(  2\nabla_{\vecee y} \vecee{w}^{j}   -  \pi^{j}(\vecee y) \ten I \bigg)\nabla \frac{\partial p^\PM}{\partial x_j} \vdot \vecee \varphi \bigg|
    &
    \leq C \varepsilon^2 \normgradphipm \, ,
    \label{eq:low-order-estimate-1}
    \qquad 
    \bigg| \int_{ \Omeps} \vecee D_\varepsilon\vdot \vecee \varphi   \bigg|
    % = \bigg| \int_{ \Omeps}  
    % \varepsilon \left( \left(\vecee t^{\bl,\varepsilon} - \mathcal{H}(x_2) \vecee N^{\bl} \right) \frac{\partial^2 }{\partial x_1^2}\frac{\partial v_1^\FF}{\partial x_2}\bigg|_\Sigma
    % \right)  \vdot \vecee \varphi   \bigg|
    \leq C \varepsilon^{3/2} \normgradphi \, ,
    \\
    \bigg|   \int_{\Omeps}  \vecee E_\varepsilon \vdot \vecee \varphi  \bigg|
    % =  \bigg|   \int_{\Omeps}  \left(-  2 \varepsilon   \frac{\partial \vecee t^{\bl,\varepsilon} }{\partial x_1}     +  s^{\bl,\varepsilon}\vecee e_1 - \mathcal{H}(x_2) N_s^{\bl} \vecee e_1 \right) \frac{\partial}{\partial x_1} \frac{\partial v_1^\FF}{\partial x_2}\bigg|_\Sigma \vdot \vecee \varphi  \bigg|
    & \leq C \varepsilon^{3/2} \normgradphi \, ,
    \qquad 
    \bigg| \int_{ \Omepspm}  \vecee F_\varepsilon \vdot \vecee \varphi  \bigg|
    %=  \bigg| \int_{ \Omepspm}  \varepsilon^2 \sum_{i,j=1}^2  \Delta \vecee \gamma^{j,i,\varepsilon} \dpxij\vdot \vecee \varphi  \bigg|
    %=  \bigg| \int_{ \Omepspm}  \varepsilon^2 \sum_{i,j=1}^2 \varepsilon^{-2} \Delta_{\vecee y} \varepsilon \vecee \gamma^{j,i}(\vecee y) \dpxij\vdot \vecee \varphi  \bigg|
    % =  \bigg| \int_{ \Omepspm}  \varepsilon \sum_{i,j=1}^2  \Delta_{\vecee y}  \vecee \gamma^{j,i}(\vecee y) \dpxij\vdot \vecee \varphi  \bigg|  
     \leq C \varepsilon^2 \normgradphipm \, ,
    \\
    \bigg|\int_{\Omeps} \vecee G_\varepsilon \vdot \vecee \varphi \bigg|
    % = \bigg| \int_{\Omeps}      \varepsilon^2 \Delta \vecee \zeta^{\bl,\varepsilon}%}_{=\Delta_{\vecee y} \vecee \zeta^{\bl}}  
    % \dvdto  \vdot \vecee \varphi \bigg|
    % = \bigg| \int_{\Omeps}  \Delta_{\vecee y} \vecee \zeta^{\bl}  \dvdto  \vdot \vecee \varphi \bigg| 
    &\leq C \varepsilon^{1/2} \normgradphi \, .
    \label{eq:low-order-estimate-5}
\end{align}
In order to obtain estimates~\eqref{eq:low-order-estimate-1}--\eqref{eq:low-order-estimate-5}, we used the assumptions given in Remark~\ref{rem},  estimates~\eqref{eq:Poincare-estimate},~\eqref{eq:estimates-cell-problems},~\eqref{eq:estimates-gamma-cell},~\eqref{eq:estimates-t} and~\eqref{eq:estimates-zeta}.
% \begin{align*}
%      \| \vecee \gamma^{j,i,\varepsilon} \|_{L^2(\Omepspm)^2} &\leq C \varepsilon \, ,
% \end{align*}
% that is given in~\cite[(1.21)]{Jaeger_Mikelic_96}.

In addition, we derive the following sharp estimates for the terms appearing on the right-hand side of equation~\eqref{eq:weak-form-approx-6}:
\begin{align}
     \bigg| \int_{\Omepspm}  \vecee M_\varepsilon\vdot \vecee \varphi \bigg|
     \leq C\varepsilon^{5/2} \normphipm &\leq C \varepsilon^{7/2} \normgradphipm \, ,
     \label{eq:estiamtes-q-1}
     \\
    \bigg| \int_{\Omepspm}  \vecee N_\varepsilon\vdot \vecee \varphi \bigg|
    %=    \bigg| \int_{\Omepspm}  \varepsilon \sum_{j=1}^2 \left( \nabla  \dpxjH z^{j,\bl,\varepsilon}    \right) \vdot \vecee \varphi \bigg| 
    \leq C\varepsilon^{3/2} \normphipm
    &\leq C\varepsilon^{5/2} \normgradphipm \, ,
    \\
    \bigg|  \int_{\Omepspm}  \vecee L_\varepsilon  \colon \nabla \vecee \varphi \bigg|
    %=    \bigg| 2 \int_{\Omepspm} \varepsilon^2 \sum_{j=1}^2 \left( \vecee q^{j,\bl,\varepsilon} \otimes \nabla \dpxjH     \right)  \colon \nabla \vecee \varphi \bigg| 
    &\leq C\varepsilon^{5/2} \normgradphipm \, .
    \label{eq:estiamtes-q-3}
\end{align}
Estimates~\eqref{eq:estiamtes-q-1}--\eqref{eq:estiamtes-q-3} are obtained by applying~\eqref{eq:Poincare-estimate} and~\eqref{eq: estimates-q-z} provided in the appendix.
All remaining  terms on the right-hand side of equation~\eqref{eq:weak-form-approx-6} are estimated at least by $C \varepsilon^{5/2} \normgradphi$ that is proven in~\cite[section 3.2]{Eggenweiler_Rybak_MMS20}.
Thus, in total, we obtain the following result
% \begin{corollary}
% For the velocity and pressure error $\vecee U^{6,\varepsilon}$, $P^{6,\varepsilon}$ the following estimate holds
\begin{align}
    \bigg| \int_{\Omega^\varepsilon}&   \nabla \vecee U^{6,\varepsilon} \colon \nabla  \vecee  \varphi
    - \int_{\Omega^\varepsilon} P^{6,\varepsilon} \div \vecee \varphi
    - \int_{\Sigma} \sum_{i=1}^5 I^i   \varphi_2  
    %IR\notag
    %IR\\
    %IR& \quad 
    + \int_{\Omega_\PM^\varepsilon} \vecee C_\varepsilon  \vdot \vecee  \varphi 
    + \int_{ \Omeps}  \vecee D_\varepsilon \vdot \vecee \varphi  
    - \int_{\Omeps}   \vecee E_\varepsilon \vdot \vecee \varphi 
    - \int_{ \Omepspm} \vecee F_\varepsilon \vdot \vecee \varphi 
    - \int_{\Omeps}  \vecee G_\varepsilon \vdot \vecee \varphi  \bigg|
    \notag 
    \\[1ex]
    &\leq C \varepsilon^{5/2} \normgradphi \, ,\qquad  \forall \vecee \varphi \in \Vper \, .
 \label{eq:weak-form-approx-6-estimate}
\end{align}
% \end{corollary}
%
Now, our task is to eliminate the five integral terms of low order on the left-hand side of inequality~\eqref{eq:weak-form-approx-6-estimate}.
This is done in the next section.

\subsection{Elimination of low-order terms w.r.t. $\varepsilon$}
\label{sec:additional-BLP}
In this section, we present the cell and boundary layer problems that are needed to eliminate the five low-order terms w.r.t. $\varepsilon$ (that cannot be estimated at least by $C \varepsilon^{5/2}$) appearing on the left-hand side of~\eqref{eq:weak-form-approx-6-estimate}. % and also in the weak formulation~\eqref{eq:weak-form-approx-6} of the error functions. 
We get rid of the low-order integral terms by adding corrector functions, obtained from the cell and boundary layer problems, to the error functions $\vecee U^{6,\varepsilon}$ and $P^{6,\varepsilon}$.
Each term that needs to be eliminated requires its own correctors, thus, we treat each correction step separately.

\subsubsection{Elimination of integral term including $\vecee C_\varepsilon$}
\label{sec:problem-AUX-cell}
In order to eliminate the term in~\eqref{eq:weak-form-approx-6-estimate} containing $\vecee C_\varepsilon$, we introduce the following auxiliary cell problems %inspired but different from the one in~\citep[equation (127)]{Carraro_etal_15}. The auxiliary problem reads
\begin{equation}
    \begin{aligned}
        - \Delta_{\vecee y} \vecee w^{j,i} + \nabla_{\vecee y} \pi^{j,i} &= \left( -2\nabla_{\vecee y} \vecee w^{j} +\pi^j \ten I \right) \vecee e_i \quad &&\text{ in } Y_\text{f} \, ,
        \\
        \divy \vecee w^{j,i} &= 0 \quad &&\text{ in } Y_\text{f} \, ,
        \\
        \vecee w^{j,i} &= \vecee 0 \quad &&\text{ on } \partial Y_\text{s} \, , 
        \\
        \{ \vecee w^{j,i} , \pi^{j,i} \} \text{ are $\vecee y$-periodic} \, , & \int_{Y_\text{f}} \pi^{j,i} d \vecee y = 0 \, .
    \end{aligned}
    \label{eq:cell-w}
\end{equation}
We set $\vecee w^{j,i,\varepsilon} (\vecee x)= \vecee w^{j,i} \left(\frac{\vecee x}{\varepsilon} \right)$, $\pi^{j,i,\varepsilon} (\vecee x) = \pi^{j,i} \left(\frac{\vecee x}{\varepsilon} \right)$ for $\vecee x \in \Omepspm$ and $\vecee w^{j,i,\varepsilon} = \vecee 0$ in $\Omega_\PM \setminus \Omepspm$.

We introduce the boundary layer corrector 
\begin{align}\label{eq:corrector-w}
    \left\{ \varepsilon^3 \vecee w^{j,i,\varepsilon}{\dpxijdom}\, , \ {\varepsilon^2}\pi^{j,i,\varepsilon}{\dpxijdom} \right\} \, 
\end{align}
that becomes a part of the new error functions $\vecee U^{7,\varepsilon}$ and $P^{7,\varepsilon}$ defined in~\eqref{eq:error-7-U} and~\eqref{eq:error-7-P}.
Corrector~\eqref{eq:corrector-w} leads to the following terms in the weak formulation for the errors
\begingroup
\allowdisplaybreaks
\begin{align}
    -&\intpmvel{\varepsilon^3}{\sum_{i,j=1}^2 \vecee w^{j,i,\varepsilon}}{\dpxijdom}
    + \intpmp{\varepsilon^2}{\sum_{i,j=1}^2 \pi^{j,i,\varepsilon}}{\dpxijdom}
     -\int_{\Omepspm} \varepsilon \sum_{i,j=1}^2 \underbrace{\left( \Delta_{\vecee y} \vecee w^{j,i} - \nabla_{\vecee y} \pi^{j,i} \right)}_{=- \left( -2 \nabla_{\vecee y} \vecee w^{j} +\pi^j \ten I \right) \vecee e_i} \dpxijdom
    \vdot \vecee \varphi \notag
    \\
    % & =
    % -\int_{\partial \Omepspm} \varepsilon^3 \sum_{i,j=1}^2 \left( \nabla \vecee w^{i,j,\varepsilon} \dpxijdom + \vecee w^{i,j,\varepsilon} \otimes \nabla \dpxijdom  \right) \vecee n \vdot \vecee \varphi 
    % \notag 
    % \\
    % & +\int_{\Omepspm} \varepsilon^3 \sum_{i,j=1}^2 \left( \Delta \vecee w^{i,j,\varepsilon} \dpxijdom
    % + 2 \nabla \vecee w^{i,j,\varepsilon} \nabla \dpxijdom  + \vecee w^{i,j,\varepsilon} \Delta \dpxijdom \right) \vdot \vecee \varphi \notag 
    % \\
    % &
    % +  \int_{\partial \Omepspm} \varepsilon^2 \sum_{i,j=1}^2 \pi^{i,j,\varepsilon} \dpxijdom  \vecee n \vdot \vecee \varphi
    % \notag 
    % \\
    % & - \int_{\Omepspm} \varepsilon^2 \sum_{i,j=1}^2 \left( \nabla \pi^{i,j,\varepsilon}  \dpxijdom + \pi^{i,j,\varepsilon} \nabla  \dpxijdom \right) \vdot \vecee \varphi 
    % \notag 
    % \\
    =&
    %\todo{-\int_{\partial \Omepspm} \varepsilon^3 \sum_{i,j=1}^2 \left( \nabla \vecee w^{i,j,\varepsilon} \dpxijdom + \vecee w^{i,j,\varepsilon} \otimes \nabla \dpxijdom  \right) \vecee n \vdot \vecee \varphi \notag}
    %%%% name boundary integral explicitly
    -\int_{\Sigma} \varepsilon^3 \sum_{i,j=1}^2  \left( \nabla \vecee w^{j,i,\varepsilon} \dpxijdom + \vecee w^{j,i,\varepsilon} \otimes \nabla \dpxijdom  \right)  \vecee e_2 \vdot \vecee \varphi 
    +\int_{\Gamma_H} \varepsilon^3 \sum_{i,j=1}^2 \left( \nabla \vecee w^{j,i,\varepsilon} \dpxijdom + \vecee w^{j,i,\varepsilon} \otimes \nabla \dpxijdom  \right) \vecee e_2 \vdot \vecee \varphi 
    \notag
    \\
    & +\int_{\Omepspm} \varepsilon^3 \sum_{i,j=1}^2 \left(2 \nabla \vecee w^{j,i,\varepsilon} \nabla \dpxijdom  + \vecee w^{j,i,\varepsilon} \Delta \dpxijdom \right) \vdot \vecee \varphi 
    % \todo{+  \int_{\partial \Omepspm} \varepsilon^2 \sum_{i,j=1}^2 \pi^{i,j,\varepsilon} \dpxijdom  \vecee n \vdot \vecee \varphi}
    %%%% name boundary integral explicitly
    +  \int_{\Sigma} \varepsilon^2 \sum_{i,j=1}^2 \pi^{j,i,\varepsilon} \dpxijdom  \vecee e_2 \vdot \vecee \varphi
    - \int_{\Omepspm} \varepsilon^2 \sum_{i,j=1}^2 \pi^{j,i,\varepsilon} \nabla  \dpxijdom  \vdot \vecee \varphi 
    \, .
    \label{eq:rhs-w-kappa}
\end{align}
\endgroup
We note that for the third term on the left-hand side of~\eqref{eq:rhs-w-kappa} we have
\begin{align}
    \int_{\Omepspm} \varepsilon \sum_{i,j=1}^2 
     \left( -2\nabla_{\vecee y} \vecee w^{j} +\pi^j \ten I \right) \vecee e_i \dpxijdom
    \vdot \vecee \varphi 
    %IR\notag
    %IR\\
    %IR&
    =
    \int_{\Omepspm} \varepsilon \sum_{j=1}^2 \left( -2\nabla_{\vecee y} \vecee w^{j} +\pi^j \ten I \right) \nabla \pd{p^\PM}{x_j}
    \vdot \vecee \varphi = 
    -  \int_{\Omepspm} \vecee C_\varepsilon \vdot \vecee \varphi \, .
\end{align}
That is exactly the term that appeared on the left-hand side of~\eqref{eq:weak-form-approx-6} but with the opposite sign. Thus, when adding the corrector~\eqref{eq:corrector-w} to the error $\{\vecee U^{6,\varepsilon}, P^{6,\varepsilon} \}$, the integral term including $\vecee C_\varepsilon$ will cancel when considering the weak formulation of the new error functions.
Note that the cell problem solution $\{ \vecee w^{j,i} ,\pi^{j,i}\}$ is independent of $\varepsilon$ and it holds
\begin{align}
\label{eq:estimates-w}
    \| \vecee w^{j,i,\varepsilon} \|_{L^2(\Omepspm)^2}  &\leq C  \, ,
    \qquad 
    \| \nabla  \vecee w^{j,i,\varepsilon} \|_{L^2(\Omepspm)^{2\times 2}}  \leq C \varepsilon^{-1} \, ,
    \qquad 
    \| \pi^{j,i,\varepsilon} \|_{L^2(\Omepspm)}  \leq C \, .
\end{align}
Using estimates~\eqref{eq:estimates-w},~\eqref{eq:Poincare-estimate} and the assumptions from Remark~\ref{rem}, all  terms on the right-hand side of~\eqref{eq:rhs-w-kappa} can be estimated at least by $C \varepsilon^{5/2} \normgradphipm$.

% The new velocity and pressure error functions read
% \begin{align*}
%      \vecee U^{7,\varepsilon} = \vecee U^{6,\varepsilon} + {\varepsilon^3}{\sum_{i,k=1}^2 \vecee w^{i,k,\varepsilon}}{\dpxik} 
%      % \vecee v^\varepsilon - \left( \vecee v^{6,\varepsilon}_\app 
%      % - {\varepsilon^3}{\sum_{i,k=1}^2 \vecee w^{i,k,\varepsilon}}{\dpxik} \right) 
%      \, , 
%      \qquad 
%      P^{7,\varepsilon} = P^{6,\varepsilon} +{\varepsilon^2}{\sum_{i,k=1}^2 \pi^{i,k,\varepsilon}}{\dpxik} 
%      % p^\varepsilon - \left( p^{6,\varepsilon}_\app 
%      % - {\varepsilon^2}{\sum_{i,k=1}^2 \pi^{i,k,\varepsilon}}{\dpxik} \right)
%       \, . 
% \end{align*}

\subsubsection{Elimination of integral term including $\vecee D_\varepsilon$}
\label{sec:elimination-D}
For the elimination of $\int_{\Omega^\varepsilon} \vecee D_\varepsilon \vdot \vecee \varphi$ in equation~\eqref{eq:weak-form-approx-6} we construct a problem defined in the boundary layer stripe $Z^\bl = Z^+ \cup S \cup Z^-$, where $Z^+ = (0,1) \times (0, \infty)$ is the free-flow part, $S=(0,1) \times \{0\}$ the interface, and $Z^-= \cup_{k=1}^\infty \left( Y_\text{f} - (0,k)\right) $ the porous part of the stripe (figure~\ref{fig:setting}, right).
The boundary layer problem reads
\begin{equation}\label{eq:BLP-nu}
\begin{aligned}
    \Delta_{\vecee y} \vecee \nu^{\bl} - \nabla_{\vecee y} \sigma^{\bl} &= - (\vecee t^{\bl} -  \mathcal{H}(y_2) \vecee N^{\bl}) \quad \text{ in } Z^+ \cup Z^- \, , 
    \\
    \divy \vecee \nu^{\bl} &=  0 \quad \text{ in } Z^+ \cup Z^- \, , 
    \\
    \llbracket   \vecee \nu^{\bl} \rrbracket_S  &= \vecee 0  \quad \text{ on } S \, , 
    \\
    \llbracket  \nabla_{\vecee y} \vecee \nu^{\bl} -  \sigma^{\bl} \ten I\rrbracket_S \vecee e_2 &= \vecee 0  \quad \text{ on } S\, , 
    \\
    \vecee \nu^{\bl} = \vecee 0 \quad \text{ on } \cup_{k=1}^{\infty}(\partial Y_\text{s} - (0,k)), & \quad \{\vecee \nu^{\bl}, \sigma^{\bl} \} \text{ is 1-periodic in $y_1$}\, .
\end{aligned}
\end{equation}
Problem~\eqref{eq:BLP-nu} is a boundary layer problem in the sense of problem (AUX) introduced by~\cite{Jaeger_Mikelic_96} since for $\gamma>0$ we have $e^{\gamma |y_2|}(\vecee t^{\bl} -  \mathcal{H}(y_2) \vecee N^{\bl}) \in L^2(Z^\bl)^2$ which is proven in~\cite[Corollary 3.4]{Jaeger_etal_01}.
Thus, we know that the solution $\{\vecee \nu^{\bl}, \sigma^{\bl}\}$ to~\eqref{eq:BLP-nu} decays exponentially fast to zero for $y_2 \to - \infty$ and stabilizes to the following boundary layer constants for $y_2 \to  \infty$:
\begin{align*}
    \vecee R^{\bl} = \left( \int_S \nu_1^{\bl} d y_1, 0  \right)^\top \, ,\qquad 
    R_\sigma^{\bl} = \int_S \sigma^{\bl} d y_1 \,. 
\end{align*}

As usual, we set $\vecee \nu^{\bl,\varepsilon} (\vecee x)= \vecee \nu^{\bl} \left(\frac{\vecee x}{\varepsilon} \right)$, $\sigma^{\bl,\varepsilon} (\vecee x) = \sigma^{\bl} \left(\frac{\vecee x}{\varepsilon} \right)$ for $\vecee x \in \Omeps$ and extend the boundary layer velocity $\vecee \nu^{\bl,\varepsilon}$ to zero in $\Omega \setminus \Omeps$.
After~\citep[Corollary 3.11 and Proposition 3.12]{Jaeger_Mikelic_96} we have the following estimates for the boundary layer functions
\begin{equation}\label{eq:estimates-nu}
\begin{aligned}
     \| \vecee \nu^{\bl, \varepsilon} - \mathcal{H}(x_2)\vecee R^\bl \|_{L^2(\Omega)^2} &\leq C \varepsilon^{1/2}  \, ,
     \qquad 
     \| \nabla \vecee \nu^{\bl, \varepsilon} \|_{L^2(\Omega)^{2\times 2}} \leq C \varepsilon^{-1/2} \, ,
     \\
     \| \sigma^{\bl, \varepsilon} - \mathcal{H}(x_2)R_\sigma^\bl \|_{L^2(\Omega^\varepsilon)} &\leq C \varepsilon^{1/2}  \, .
\end{aligned}
\end{equation}

The part of the weak form corresponding to the boundary layer corrector 
\begin{align}\label{eq:corrector-nu}
\left\{ {\varepsilon^3 }{\left( \vecee \nu^{\bl,\varepsilon} - \mathcal{H} (x_2)\vecee R^{\bl} \right)}{\dvdtoo} , \ {\varepsilon^2}{\left( \sigma^{\bl,\varepsilon} - \mathcal{H}(x_2) R_\sigma^{\bl} \right)}{\dvdtoo}  \right\} \, ,
\end{align}
which is later added to the velocity and pressure error, reads
\begingroup
\begin{align}
    &\intvel{\varepsilon^3 }{\left( \vecee \nu^{\bl,\varepsilon} - \mathcal{H} (x_2)\vecee R^{\bl} \right)}{\dvdtoo} 
    - \intp{\varepsilon^2}{\left( \sigma^{\bl,\varepsilon} - \mathcal{H}(x_2) R_\sigma^{\bl} \right)}{\dvdtoo} 
    \notag 
    \\ 
    & \qquad 
    %+ \int_\Sigma \underbrace{{\varepsilon^2}{R_\sigma^{\bl} }{ \dvdtoo}}_{=\colon I^6} \varphi_2
     -
    \int_{\Omeps} \underbrace{\varepsilon \left(\vecee t^{\bl,\varepsilon} - \mathcal{H}(x_2) \vecee N^{\bl}\right)\dvdtoo}_{= \vecee D_\varepsilon} \vdot \vecee \varphi
    \notag
    \\
    &= \quad  
    -\int_\Sigma {\varepsilon^2}{R_\sigma^{\bl} }{ \dvdtoo} \varphi_2
    +\intdel{\varepsilon^2}{\vecee \nu^{\bl,\varepsilon} - \mathcal{H} (x_2)\vecee R^{\bl} }{\dvdtoo} \notag
    \\
    & \qquad 
    +\intotim{\varepsilon^3}{\vecee \nu^{\bl,\varepsilon} - \mathcal{H} (x_2)\vecee R^{\bl}}{\dvdtoo}  + \intpff{\varepsilon^2}{\sigma^{\bl,\varepsilon} -  R_\sigma^{\bl}}{\dvdtoo}
    \notag 
    \\
    & \qquad + \intppm{\varepsilon^2}{\sigma^{\bl,\varepsilon}}{\dvdtoo}  
     + \text{exponentially small terms} \, .
    \label{eq:rhs-nu}
\end{align}
\endgroup
Note that the term including $\vecee D_\varepsilon$ we wanted to eliminate in~\eqref{eq:weak-form-approx-6} appears now on the left-hand side of~\eqref{eq:rhs-nu} with the opposite sign, as desired.
Using the assumptions in Remark~\ref{rem} and estimates~\eqref{eq:estimates-nu}, all terms staying on the right-hand side (RHS) of~\eqref{eq:rhs-nu} are estimated by
\begin{align}
    | RHS\eqref{eq:rhs-nu} | \leq C \varepsilon^{5/2} \normgradphi \, .
\end{align}

% This gives us a new velocity and pressure error function
% \begin{align}
%      \vecee U^{8,\varepsilon} &= \vecee U^{7,\varepsilon}
%     %+  \varepsilon^2 \left( \vecee \xi^{\bl,\varepsilon} - \mathcal{H}(x_2) \vecee L^{\bl} \right) \frac{\partial}{\partial x_1}  \frac{\partial v_1^\FF}{\partial x_2}\bigg|_\Sigma  \notag
%      %\\
%      %&\quad 
%     +  \varepsilon^3 \left( \vecee \nu^{\bl,\varepsilon} - \mathcal{H}(x_2) \vecee R^{\bl} \right) \frac{\partial^2}{\partial x_1^2}  \frac{\partial v_1^\FF}{\partial x_2}\bigg|_\Sigma
%      \, , 
%      \\
%      P^{8,\varepsilon} &= P^{7,\varepsilon} 
%      %+  \varepsilon \left( \eta^{\bl,\varepsilon} - \mathcal{H}(x_2)  L_\eta^{\bl} \right) \frac{\partial}{\partial x_1}  \frac{\partial v_1^\FF}{\partial x_2}\bigg|_\Sigma \notag
%      %\\
%      %&\quad
%      +  \varepsilon^2 \left( \sigma^{\bl,\varepsilon} - \mathcal{H}(x_2)  R_\sigma^{\bl} \right) \frac{\partial^2}{\partial x_1^2}  \frac{\partial v_1^\FF}{\partial x_2}\bigg|_\Sigma
%       \, . 
% \end{align}

\subsubsection{Elimination of integral term including $\vecee E_\varepsilon$}
To get rid of the low-order term $\int_{\Omega^\varepsilon} \vecee E_\varepsilon \vdot \vecee \varphi$  in~\eqref{eq:weak-form-approx-6}, we define the following  boundary layer problem
\begin{equation}\label{eq:BLP-xi}
\begin{aligned}
    \Delta_{\vecee y} \vecee \xi^{\bl} - \nabla_{\vecee y} \eta^{\bl} &= - \lrp{ 2\frac{\partial \vecee t^{\bl} }{\partial y_1} - s^{\bl} \vecee e_1 { \color{black}+ \mathcal{H}(y_2) N_s^{\bl} \vecee e_1}} \quad \text{ in } Z^+ \cup Z^- \, ,
    \\
    \divy \vecee \xi^{\bl} &=  0%-t_1^{\bl} (\vecee y) + \mathcal{H}(y_2) N_1^\text{\bl}  
    \quad \text{ in } Z^+ \cup Z^- \, ,
    \\
    \llbracket   \vecee \xi^{\bl} \rrbracket_S  &= \vecee 0  \quad \text{ on } S \, , 
    \\
    \llbracket  \nabla_{\vecee y} \vecee \xi^{\bl} -  \eta^{\bl} \ten I\rrbracket_S \vecee e_2 &= \vecee 0  \quad \text{ on } S \, , 
    \\
    \vecee \xi^{\bl} = \vecee 0 \quad \text{ on }  \cup_{k=1}^{\infty}&(\partial Y_\text{s} - (0,k)),  \quad \{\vecee \xi^{\bl}, \eta^{\bl} \} \text{ is 1-periodic in $y_1$} \, . 
\end{aligned}
\end{equation}
Since for $\gamma>0$ it is $e^{\gamma|y_2|}\nabla_{\vecee y} \vecee t^{\bl} \in L^2(Z^{\bl})^2$ and  $e^{\gamma|y_2|}\left( s^{\bl} - \mathcal{H}(y_2)N_s^{\bl} \right) \in L^2(Z^{\bl})$, we know that problem~\eqref{eq:BLP-xi} fits in the form of the (AUX) problem from~\citep{Jaeger_Mikelic_96}. Thus, it is a boundary layer problem in the classical sense and there exists a solution $\{\vecee \xi^\bl, \eta^\bl\}$ that is unique (the pressure up to a constant).
We extend the boundary layer velocity by zero in $\Omega \setminus \Omega^\varepsilon$ and set  $ \vecee \xi^{\bl,\varepsilon} (\vecee x) = \vecee \xi^{\bl} \lrp{\frac{\vecee x}{\varepsilon}} $ and $\eta^{\bl,\varepsilon} (\vecee x) = \eta^{\bl} \lrp{\frac{\vecee x}{\varepsilon}} $ for $\vecee x \in \Omeps$.
After~\citep{Jaeger_Mikelic_96}  the boundary layer velocity and pressure stabilize to zero in the porous medium and to the following constants in the free-flow region 
\begin{align}\label{eq:BLP-constants-L}
    \vecee L^{\bl} = \left( L_1^\bl, 0\right)^\top= \left( \int_S \xi_1^{\bl}  d y_1, 0 \right)^\top \, ,
    \qquad 
    L_\eta^{\bl} =  \int_S \eta^{\bl}  d y_1\, .
\end{align}
Moreover, the following estimates for boundary layer velocity and pressure hold
\begin{equation}\label{eq:estimates-xi}
\begin{aligned}
     \| \vecee \xi^{\bl, \varepsilon} - \mathcal{H}(x_2)\vecee L^\bl \|_{L^2(\Omega)^2} &\leq C \varepsilon^{1/2}  \, ,
     \qquad 
     \| \nabla \vecee \xi^{\bl, \varepsilon} \|_{L^2(\Omega)^{2\times 2}} \leq C \varepsilon^{-1/2} \, ,
     \\
     \| \eta^{\bl, \varepsilon} - \mathcal{H}(x_2)L_\eta^\bl \|_{L^2(\Omega^\varepsilon)} &\leq C \varepsilon^{1/2}  \, .
\end{aligned}
\end{equation}

We introduce the following corrector that will appear in the new error functions
\begin{align}\label{eq:corrector-xi}
 \left\{ {\varepsilon^2 }{\left( \vecee \xi^{\bl,\varepsilon} - \mathcal{H} (x_2)\vecee L^{\bl} \right)}{\dvdto} , \
 {\varepsilon}{\left( \eta^{\bl,\varepsilon} - \mathcal{H}(x_2) L_\eta^{\bl} \right)}{\dvdto}  \right\} \, .
\end{align}
%The corresponding weak form is
The corrector generates the subsequent terms when considering the weak formulation for the velocity error~$\vecee U^{7,\varepsilon}$ and pressure error~$P^{7,\varepsilon}$:
\begingroup
\allowdisplaybreaks
\begin{align}
    & \intvel{\varepsilon^2 }{\left( \vecee \xi^{\bl,\varepsilon} - \mathcal{H} (x_2)\vecee L^{\bl} \right)}{\dvdto} 
    - \intp{\varepsilon}{\left( \eta^{\bl,\varepsilon} - \mathcal{H}(x_2) L_\eta^{\bl} \right)}{\dvdto} 
     \notag
    \\
    & \quad
    + \int_\Sigma \underbrace{\varepsilon L_\eta^{\bl}  \dvdto}_{=\colon I^6} \varphi_2
    +\int_{\Omeps}  \underbrace{ \left(-  2 \varepsilon   \frac{\partial}{\partial x_1} \vecee t^{\bl,\varepsilon}
    +  s^{\bl,\varepsilon}\vecee e_1 - \mathcal{H}(x_2) N_s^{\bl} \vecee e_1 \right) \frac{\partial}{\partial x_1} \frac{\partial v_1^\FF}{\partial x_2}\bigg|_\Sigma }_{= \vecee E_\varepsilon} \vdot \vecee \varphi 
    \notag
    \\
    =&\quad 
    \intdel{\varepsilon^2}{\vecee \xi^{\bl,\varepsilon} - \mathcal{H} (x_2)\vecee L^{\bl} }{\dvdto}
    + \intppm{\varepsilon}{\eta^{\bl,\varepsilon}}{\dvdto}
    \notag
    \\
    & \quad 
    +\intotim{\varepsilon^2}{\vecee \xi^{\bl,\varepsilon} - \mathcal{H} (x_2)\vecee L^{\bl}}{\dvdto} \notag
    \\
    & \quad
    + \intpff{\varepsilon}{\eta^{\bl,\varepsilon} -  L_\eta^{\bl}}{\dvdto}  
    + \text{exponentially small terms} \, .
    \label{eq:rhs-xi}
\end{align}
\endgroup
The integral including $I^6$ in the second line will later vanish due to the interface condition for the pressure.
Estimation of the last integral term over $\Omega_\FF$ on the right-hand side of~\eqref{eq:rhs-xi} is done in the same way as in~\citep[eqs. (3.22)--(3.24)]{Eggenweiler_Rybak_MMS20}.
Thus, under the assumptions given in Remark~\ref{rem} and using~\eqref{eq:Poincare-estimate} as well as \eqref{eq:estimates-xi}, we obtain the following estimate for all integral terms on the right-hand side of~\eqref{eq:rhs-xi}:
\begin{align}
    | RHS\eqref{eq:rhs-xi} | \leq C \varepsilon^{5/2} \normgradphi 
    \, .
\end{align}

% We introduce new velocity and pressure error functions as follows
% \begin{align}
%      \vecee U^{9,\varepsilon} &= \vecee U^{8,\varepsilon}
%     +  \varepsilon^2 \left( \vecee \xi^{\bl,\varepsilon} - \mathcal{H}(x_2) \vecee L^{\bl} \right) \frac{\partial}{\partial x_1}  \frac{\partial v_1^\FF}{\partial x_2}\bigg|_\Sigma  \, , 
%     \notag
%      \\
%      P^{9,\varepsilon} &= P^{8,\varepsilon} 
%      +  \varepsilon \left( \eta^{\bl,\varepsilon} - \mathcal{H}(x_2)  L_\eta^{\bl} \right) \frac{\partial}{\partial x_1}  \frac{\partial v_1^\FF}{\partial x_2}\bigg|_\Sigma \notag
%       \, . 
% \end{align}

\subsubsection{Elimination of integral term including $\vecee F_\varepsilon$}
In the same way as in section~\ref{sec:problem-AUX-cell}, we introduce other auxiliary cell problems in order to eliminate the integral in~\eqref{eq:weak-form-approx-6} that contains $\vecee F_\varepsilon$. For $i,j=1,2$ these problems read
\begin{equation}\label{eq:cell-gamma}
\begin{aligned}
    - \Delta_{\vecee y} \vecee W^{j,i} + \nabla_{\vecee y} \Lambda^{j,i} &= \Delta_{\vecee y } \vecee \gamma^{j,i} \quad &&\text{ in } Y_\text{f} \, ,
    \\
    \divy \vecee W^{j,i} &= 0 \quad &&\text{ in } Y_\text{f} \, ,
    \\
    \vecee W^{j,i} &= \vecee 0 \quad &&\text{ on } \partial Y_\text{s} \, , 
    \\
    \{ \vecee W^{j,i} , \Lambda^{j,i} \} \text{ are $\vecee y$-periodic} \, , & \int_{Y_\text{f}} \Lambda^{j,i} d \vecee y = 0 \, ,
\end{aligned}
\end{equation}
where $\vecee \gamma^{j,i}$ is the solution to problem~\eqref{eq:cell-problems-gamma}.
In a standard way, the solution to~\eqref{eq:cell-gamma} is extended to the pore space $\Omepspm$, i.e., $\vecee W^{j,i,\varepsilon} (\vecee x)= \vecee W^{j,i} \left(\frac{\vecee x}{\varepsilon} \right)$, $\Lambda^{j,i,\varepsilon} (\vecee x) = \Lambda^{j,i} \left(\frac{\vecee x}{\varepsilon} \right)$ for $\vecee x \in \Omepspm$. Moreover, we set $\vecee W^{j,i,\varepsilon} = \vecee 0$ in $\Omega_\PM \setminus \Omepspm$, and introduce the boundary layer corrector
\begin{align}\label{eq:corrector-W}
\left\{ {\varepsilon^3}{\sum_{i,j=1}^2 \vecee W^{j,i,\varepsilon}}{\dpxij} , \
 {\varepsilon^2}{\sum_{i,j=1}^2 \Lambda^{j,i,\varepsilon}}{\dpxij} \right\} \, .
\end{align}
Corresponding to this corrector the following terms appear in the weak formulation for the error functions
\begin{align}
    &-\intpmvel{\varepsilon^3}{\sum_{i,j=1}^2 \vecee W^{j,i,\varepsilon}}{\dpxij}
    + \intpmp{\varepsilon^2}{\sum_{i,j=1}^2 \Lambda^{j,i,\varepsilon}}{\dpxij}
    -\int_{\Omepspm} \varepsilon \sum_{i,j=1}^2 \underbrace{\left( \Delta_{\vecee y} \vecee W^{j,i} - \nabla_{\vecee y} \Lambda^{j,i} \right)}_{=- \Delta_{\vecee y} \vecee \gamma^{j,i}} \dpxij
    \vdot \vecee \varphi \notag
    \\
    & \  = 
    % -\int_{\partial \Omepspm} \varepsilon^3 \sum_{i,j=1}^2 \left( \nabla \vecee W^{i,j,\varepsilon} \dpxij + \vecee W^{i,j,\varepsilon} \otimes \nabla \dpxij  \right) \vecee n \vdot \vecee \varphi 
    % \notag 
    % \\
    -\int_{\Sigma} \varepsilon^3 \sum_{i,j=1}^2 \left( \nabla \vecee W^{j,i,\varepsilon} \dpxij + \vecee W^{j,i,\varepsilon} \otimes \nabla \dpxij  \right) \vecee e_2 \vdot \vecee \varphi 
    +\int_{\Gamma_H} \varepsilon^3 \sum_{i,j=1}^2 \left( \nabla \vecee W^{j,i,\varepsilon} \dpxij + \vecee W^{j,i,\varepsilon} \otimes \nabla \dpxij  \right) \vecee e_2 \vdot \vecee \varphi 
    \notag 
    \\
    & \quad +\int_{\Omepspm} \varepsilon^3 \sum_{i,j=1}^2 \left(2 \nabla \vecee W^{j,i,\varepsilon} \nabla \dpxij  + \vecee W^{j,i,\varepsilon} \Delta \dpxij\right) \vdot \vecee \varphi 
    % +  \int_{\partial \Omepspm} \varepsilon^2 \sum_{i,j=1}^2 \Lambda^{i,j,\varepsilon} \dpxij  \vecee n \vdot \vecee \varphi
    +  \int_{\Sigma} \varepsilon^2 \sum_{i,j=1}^2 \Lambda^{j,i,\varepsilon} \dpxij  \vecee e_2 \vdot \vecee \varphi
    \notag 
    \\
    & \quad
    - \int_{\Omepspm} \varepsilon^2 \sum_{i,j=1}^2 \Lambda^{j,i,\varepsilon} \nabla  \dpxij  \vdot \vecee \varphi \, .
    \label{eq:rhs-W-K}
\end{align}
The third term on the left-hand side of~\eqref{eq:rhs-W-K} is reformulated as follows
\[
    -\int_{\Omepspm} \varepsilon \sum_{i,j=1}^2 \underbrace{\left( \Delta_{\vecee y} \vecee W^{j,i} - \nabla_{\vecee y} \Lambda^{j,i} \right)}_{=- \Delta_{\vecee y} \vecee \gamma^{j,i}} \dpxij
    \vdot \vecee \varphi 
    = 
    \int_{ \Omepspm}  \underbrace{\varepsilon^2 \sum_{i,j=1}^2  \Delta \vecee \gamma^{j,i,\varepsilon} \dpxij}_{=\vecee F_\varepsilon} \vdot \vecee \varphi \, .
\]
%Note: This is correct since in the definition of gamma^\varepsilon there sits an additional \varepsilon!!!
Since we have
\begin{align*}
    \| \vecee W^{j,i,\varepsilon} \|_{L^2(\Omepspm)^2} \leq C  \, ,
    \quad 
    \| \nabla  \vecee W^{j,i,\varepsilon} \|_{L^2(\Omepspm)^{2\times 2}}  \leq C \varepsilon^{-1} \, ,
    \quad 
     \| \Lambda^{j,i,\varepsilon} \|_{L^2(\Omepspm)} &\leq C   \, ,
\end{align*}
and we assume the estimates presented in Remark~\ref{rem} hold, all terms on the right-hand side of~\eqref{eq:rhs-W-K} can be estimated at least by $C \varepsilon^{5/2} \normgradphipm$. Note that we also used the estimates~\eqref{eq:Poincare-estimate} here.

% We introduce new velocity and pressure error functions as follows
% \begin{align}
%      \vecee U^{10,\varepsilon} &= \vecee U^{9,\varepsilon}
%     - \varepsilon^3 {\sum_{i,j=1}^2 \vecee W^{j,i,\varepsilon}}{\dpxij} \, , 
%     \notag
%      \\
%      P^{10,\varepsilon} &= P^{9,\varepsilon} 
%      - \varepsilon^2 {\sum_{i,j=1}^2 \Lambda^{j,i,\varepsilon}}{\dpxij} \notag
%       \, . 
% \end{align}

\subsubsection{Elimination of integral term including $\vecee G_\varepsilon$}
In the weak formulation~\eqref{eq:weak-form-approx-6}, we observe that the term including $\vecee G_\varepsilon$ is of low estimation order w.r.t.~$\varepsilon$. In order to eliminate this term, we construct the following boundary layer problem
\begin{equation}\label{eq:BLP-c}
\begin{aligned}
    \Delta_{\vecee y} \vecee c^{\bl} - \nabla_{\vecee y} b^{\bl} &= \Delta_{\vecee y} \vecee \zeta^{\bl} \quad \text{ in } Z^+ \cup Z^- \, , %\label{eq:momentum-c}
    \\
    \divy \vecee c^{\bl} &=  0 \qquad \quad  \text{in } Z^+ \cup Z^- \, ,
    \\
    \llbracket   \vecee c^{\bl} \rrbracket_S  &= \vecee 0  \qquad \quad  \text{on } S \, , 
    \\
    \llbracket  \nabla_{\vecee y} \vecee c^{\bl} -  b^{\bl} \ten I\rrbracket_S \vecee e_2 &= \vecee 0  \qquad \quad  \text{on } S \, , 
    \\
    \vecee c^{\bl} = \vecee 0 \quad \text{ on } \cup_{k=1}^{\infty}(\partial Y_\text{s} - (0,k)), &\quad \{\vecee c^{\bl}, b^{\bl} \} \text{ is 1-periodic in $y_1$} \, ,%\label{eq:perio-c}
\end{aligned}
\end{equation}
where $\vecee \zeta^{{\bl}}$ is the solution to~\eqref{eq:BLP-zeta}.
Problem~\eqref{eq:BLP-c} is a boundary layer problem in the sense of~\cite[problem (AUX)]{Jaeger_Mikelic_96} due to the following argument. %We apply partial derivation to the first equation of~\eqref{eq:BLP-zeta} and obtain
We differentiate the first equation of~\eqref{eq:BLP-zeta} two times w.r.t. $y_1$ and obtain
\begin{align}\label{eq:zeta-fits-AUX}
    \frac{\partial^2}{\partial y_1^2}\left(  \divy  \vecee  \zeta^{\bl} \right)=  \divy  \left(  \frac{\partial^2 \vecee \zeta^{\bl} }{\partial y_1^2} \right)  = \frac{\partial^2 t_1^{\bl} }{\partial y_1^2} \, .
\end{align}
%For the term on the right-hand side of~\eqref{eq:zeta-fits-AUX}, we
We know from~\cite[Corollary~3.8]{Jaeger_etal_01} that for $\gamma^\text{lower} =\frac{1}{2}\ln\left(\frac{C^\text{lower}+1}{C^\text{lower}} \right)$ with constant $C^\text{lower}>0$ %satisfying inequality~(3.24) from~\cite{Jaeger_etal_01}
it holds
\[
|D^{\alpha} \vecee t^{\bl} (y_1, y_2) | \leq C(\bar a, \alpha) e^{- \gamma^\text{lower} |y_2|} \, , 
\]
for any $\bar a<0$, $\alpha \in \mathbb{N}$ and $y_2<\bar a$.
For this reason, we also know that $ e^{\gamma |y_2|} \Delta \vecee \zeta^{\bl} \in L^2(Z^+ \cup Z^-) $ and after~\cite[section~3]{Jaeger_Mikelic_96} problem~\eqref{eq:BLP-c} is a boundary layer problem.
Thus, the velocity and pressure stabilize exponentially to zero in the porous part and to the following boundary layer constants in the free-flow part of the boundary layer stripe
\begin{align}\label{eq:BLP-constants-E}
    \vecee E^{\bl} = \left(E_1^\bl,0 \right)^\top = \left( \int_S c_1^{\bl}  d y_1, 0 \right)^\top \, , 
    \qquad 
    E_b^{\bl} =  \int_S b^{\bl}  d y_1\, .
\end{align}
We extend the velocity $\vecee c^{\bl}$ to zero in the solid inclusions as usual and write $ \vecee c^{\bl,\varepsilon} (\vecee x) = \vecee c^{\bl} \lrp{\frac{\vecee x}{\varepsilon}}$ and $b^{\bl,\varepsilon} (\vecee x) = b^{\bl} \lrp{\frac{\vecee x}{\varepsilon}} $.
Then, following~\citep{Jaeger_Mikelic_96} it holds
\begin{equation}\label{eq:estimates-c}
\begin{aligned}
     \| \vecee c^{\bl, \varepsilon} - \mathcal{H}(x_2)\vecee E^\bl \|_{L^2(\Omega)^2} &\leq C \varepsilon^{1/2}  \, ,
     \qquad 
     \| \nabla \vecee c^{\bl, \varepsilon} \|_{L^2(\Omega)^{2\times 2}} \leq C \varepsilon^{-1/2} \, ,
     \\
     \| b^{\bl, \varepsilon} - \mathcal{H}(x_2) E_b^\bl \|_{L^2(\Omega^\varepsilon)} &\leq C \varepsilon^{1/2}  \, .
\end{aligned}
\end{equation}

% We add the boundary layer correctors and the corresponding constants to the velocity and pressure approximation and obtain the following error functions
% \begin{align*}
%      \vecee U^{11,\varepsilon} &= \vecee U^{10,\varepsilon} 
%      + \varepsilon^2 \left( \vecee c^{\bl,\varepsilon} - \mathcal{H}(x_2) \vecee E^{\bl} \right) \dvdto  \, , 
%      \\
%      P^{11,\varepsilon} &= P^{10,\varepsilon} 
%      +\varepsilon \left( b^{\bl,\varepsilon} - \mathcal{H}(x_2) E_b^{\bl} \right) \dvdto 
%       \, . 
% \end{align*}

In order to eliminate the integral term with~$\vecee G_\varepsilon$  in~\eqref{eq:weak-form-approx-6}, we define the following corrector
\begin{align}\label{eq:corrector-c}
 \left\{ 
 {\varepsilon^2 }{\left( \vecee c^{\bl,\varepsilon} - \mathcal{H} (x_2)\vecee E^{\bl} \right)}{\dvdto}   , \
 {\varepsilon}{\left( b^{\bl,\varepsilon} - \mathcal{H}(x_2) E_b^{\bl} \right)}{\dvdto} 
 \right\} \, .
\end{align}
The terms in the weak formulation corresponding to this corrector read
\begingroup
\allowdisplaybreaks
\begin{align}
    & \intvel{\varepsilon^2 }{\left( \vecee c^{\bl,\varepsilon} - \mathcal{H} (x_2)\vecee E^{\bl} \right)}{\dvdto} 
    - \intp{\varepsilon}{\left( b^{\bl,\varepsilon} - \mathcal{H}(x_2) E_b^{\bl} \right)}{\dvdto} 
    \notag 
    \\ 
    & \quad 
    + \int_\Sigma \underbrace{\varepsilon E_b^{\bl}  \dvdto}_{=\colon I^7} \varphi_2
    + \int_{\Omega^\varepsilon} \underbrace{ \varepsilon^2 \Delta \vecee \zeta^{\bl,\varepsilon}  \dvdto}_{=\vecee G_\varepsilon} \vdot \vecee \varphi 
    \notag
    \\
    =&\quad 
     \intdel{\varepsilon}{\vecee c^{\bl,\varepsilon} - \mathcal{H} (x_2)\vecee E_b^{\bl} }{\dvdto} 
     + \intppm{\varepsilon}{b^{\bl,\varepsilon}}{\dvdto} 
     \notag
    \\
    & \quad 
    +\intotim{\varepsilon^2}{\vecee c^{\bl,\varepsilon} - \mathcal{H} (x_2)\vecee E^{\bl}}{\dvdto} \notag
    \\
    & \quad
    + \intpff{\varepsilon}{b^{\bl,\varepsilon} -  E_b^{\bl}}{\dvdto}
    + \text{exponentially small terms} \, .
    \label{eq:rhs-c}
\end{align}
\endgroup
The integral $\int_{\Omega^\varepsilon} \vecee G_\varepsilon \vdot \vecee \varphi$ will cancel with the same integral appearing in the weak formulation for the new error functions~$\vecee U^{7,\varepsilon}$ and $P^{7,\varepsilon}$ defined in~\eqref{eq:error-7-U} and~\eqref{eq:error-7-P}. 
Using Remark~\ref{rem} and estimates~\eqref{eq:estimates-c}, the integrals on the right-hand side of~\eqref{eq:rhs-c} can be estimated at least by
\begin{align}
    |RHS\eqref{eq:rhs-c} | \leq C \varepsilon^{5/2} \normgradphi \, .
\end{align}

With the correctors~\eqref{eq:corrector-w},~\eqref{eq:corrector-nu},~\eqref{eq:corrector-xi},~\eqref{eq:corrector-W},~\eqref{eq:corrector-c} defined using the solutions to the boundary layer problems~\eqref{eq:BLP-nu},~\eqref{eq:BLP-xi},~\eqref{eq:BLP-c} and cell problems~\eqref{eq:cell-w},~\eqref{eq:cell-gamma}, we introduce the following new error functions
\begin{align}
     \vecee U^{7,\varepsilon} &= \vecee U^{6,\varepsilon}
     + \mathcal{H}(-x_2) {\varepsilon^3}{\sum_{i,j=1}^2 \vecee w^{j,i,\varepsilon}}{\dpxijdom} \notag
     % \\
     % & \quad 
     +   \varepsilon^3 \left( \vecee \nu^{\bl,\varepsilon} - \mathcal{H}(x_2) \vecee R^{\bl} \right) \frac{\partial^2}{\partial x_1^2}  \frac{\partial v_1^\FF}{\partial x_2}\bigg|_\Sigma 
     +  \varepsilon^2 \left( \vecee \xi^{\bl,\varepsilon} - \mathcal{H}(x_2) \vecee L^{\bl} \right) \frac{\partial}{\partial x_1}  \frac{\partial v_1^\FF}{\partial x_2}\bigg|_\Sigma  
     \notag
     \\
     & \quad 
     - \mathcal{H}(-x_2)\varepsilon^3 {\sum_{i,j=1}^2 \vecee W^{j,i,\varepsilon}}{\dpxij}
     + \varepsilon^2 \left( \vecee c^{\bl,\varepsilon} - \mathcal{H}(x_2) \vecee E^{\bl} \right) \dvdto  
     \, ,
     \label{eq:error-7-U}
     \\
     P^{7,\varepsilon} &= P^{6,\varepsilon} 
     + \mathcal{H}(-x_2) {\varepsilon^2}{\sum_{i,j=1}^2 \pi^{j,i,\varepsilon}}{\dpxijdom} \notag
     % \\
     % & \quad 
     +  \varepsilon^2 \left( \sigma^{\bl,\varepsilon} - \mathcal{H}(x_2)  R_\sigma^{\bl} \right) \frac{\partial^2}{\partial x_1^2}  \frac{\partial v_1^\FF}{\partial x_2}\bigg|_\Sigma 
     +  \varepsilon \left( \eta^{\bl,\varepsilon} - \mathcal{H}(x_2)  L_\eta^{\bl} \right) \frac{\partial}{\partial x_1}  \frac{\partial v_1^\FF}{\partial x_2}\bigg|_\Sigma 
     \notag
     \\
     & \quad 
     - \mathcal{H}(-x_2) \varepsilon^2 {\sum_{i,j=1}^2 \Lambda^{j,i,\varepsilon}}{\dpxij} +\varepsilon \left( b^{\bl,\varepsilon} - \mathcal{H}(x_2) E_b^{\bl} \right) \dvdto 
     \label{eq:error-7-P}
      \, . 
\end{align}

We obtain the following result for the new velocity and pressure errors $\vecee U^{7,\varepsilon} $ and $P^{7,\varepsilon}$.
\begin{corollary}\label{cor-1}
    For all $\vecee \varphi \in \Vper$ it holds
    \begin{align}
        \bigg| \int_{\Omega^\varepsilon} &\nabla \vecee U^{7,\varepsilon} \colon \nabla \vecee \varphi 
        - \int_{\Omega^\varepsilon} P^{7,\varepsilon} \nabla \vdot \vecee \varphi - \int_{\Sigma} \sum_{i=1}^5 I^i   \varphi_2  
        + \int_{\Sigma} \varepsilon \left( L_\eta^\bl + E_b^\bl \right) \dvdto  \varphi_2          \bigg|
        \notag 
        \\
        &
        \leq C \varepsilon^{5/2} \normgradphi + \text{exponentially small terms} \, .
        \label{eq:cor-1}
    \end{align}
\end{corollary}
\begin{proof}
    Combining~\eqref{eq:weak-form-approx-6},~\eqref{eq:rhs-w-kappa},~\eqref{eq:rhs-nu},~\eqref{eq:rhs-xi},~\eqref{eq:rhs-W-K},~\eqref{eq:rhs-c} and using the previously derived estimates % fact that 
    \begin{align*}
    &|RHS\eqref{eq:weak-form-approx-6}| + |RHS\eqref{eq:rhs-w-kappa}| + |RHS\eqref{eq:rhs-nu}| + |RHS\eqref{eq:rhs-xi}| + |RHS\eqref{eq:rhs-W-K}|+ |RHS\eqref{eq:rhs-c}|
    \\
    & \leq C \varepsilon^{5/2} \normgradphi  + \text{exponentially small terms} \, ,
    \end{align*}
    yields estimate~\eqref{eq:cor-1}.
\end{proof}

In order to eliminate the integrals over $\Sigma$ appearing on the left-hand side of~\eqref{eq:cor-1}, we set the following condition
\begin{align}
    0 &= \sum_{i=1}^5 I^i  + \varepsilon \left( L_\eta^\bl + E_b^\bl \right) \dvdto \notag 
    \\
    &=  p^\PM +  \frac{\partial v_2^\FF}{\partial x_2}\bigg|_\Sigma  - p^\FF 
    -N_s^\bl \frac{\partial v_1^\FF}{\partial x_2}\bigg|_\Sigma
    + \varepsilon \sum_{j=1}^2  M_\omega^{j,\bl} \frac{\partial p^\PM}{\partial x_j}\bigg|_\Sigma
    -\varepsilon \left( N_1^\bl + L_\eta^\bl + E_b^\bl \right) \dvdto \qquad \text{on } \Sigma  \, .
    \label{eq:set-pressure-IC}
\end{align}
Here, we used the definitions of $I^i$ for $i=1,\ldots,5$ introduced in~\eqref{eq:weak-form-approx-6} and the equality
\begin{align*}
    \big\llbracket \varepsilon^2 \nabla \vecee \zeta^{\bl,\varepsilon} \vecee e_2 \rrbracket_\Sigma 
    = 
    \big\llbracket \varepsilon \nabla_{\vecee y} \vecee \zeta^{\bl} (\vecee y) \vecee e_2 \rrbracket_S 
    \overset{\eqref{eq:zeta-gradient-jump}}{=} 
    - \varepsilon N_1^\bl \vecee e_2 \, .
\end{align*}
In this way, interface condition~\eqref{eq:IC-HO-p} is derived.

To obtain estimates for the error functions defined by~\eqref{eq:error-7-U} and~\eqref{eq:error-7-P}, we need to use the velocity error $\vecee U^{7,\varepsilon}$ as a test function in~\eqref{eq:cor-1}. The difficulty with $\vecee U^{7,\varepsilon}$ is that its trace on~$\Sigma$ is not correct yet and the boundary conditions at the lower boundary $\Gamma_H$ are not satisfied. Hence, we have to adjust the values of $\vecee U^{7,\varepsilon}$ at the interface and the lower boundary.

\subsection{Correction of trace and outer boundary effects}
\label{sec:trace-and-lower-boundary}
In this section, we correct the velocity trace as well as the velocity error function $\vecee U^{7,\varepsilon} $ on the lower boundary $\Gamma_H$.
These corrections are needed in order to get a velocity error that is an element of the test function space $\Vper$ defined in~\eqref{eq:Vperio}.

\subsubsection{Trace correction}
Since our goal is $\vecee U^{7,\varepsilon} \in \Vper$, we need the continuity of velocity trace across the interface, i.e., $ \llbracket \vecee U^{7,\varepsilon} \rrbracket_\Sigma = \vecee 0 $. At the moment, %this is not the case since 
we have
\begingroup
\allowdisplaybreaks
\begin{align}
     \llbracket \vecee U^{7,\varepsilon} \rrbracket_\Sigma 
    &= 
     -\vecee{v}^\FF 
    - \varepsilon \vecee N^{\bl}  \frac{\partial v_1^\FF}{\partial x_2}\bigg|_\Sigma 
    - \varepsilon^2 \sum_{j=1}^2  k_{2j}\vecee e_2  \frac{\partial p^\PM }{\partial x_j}\bigg|_\Sigma 
    + \varepsilon^2 \sum_{j=1}^2 \vecee M^{j,\bl} \frac{\partial p^\PM }{\partial x_j}\bigg|_\Sigma 
    - \varepsilon^2
    \underbrace{\llbracket \vecee \zeta^{\bl,\varepsilon} \rrbracket_\Sigma}_{=W^\bl \vecee e_2} \frac{\partial}{\partial x_1} \frac{\partial v_1^\FF}{\partial x_2}\bigg|_\Sigma 
     + \varepsilon^3  \sum_{i,j=1}^2   \vecee  C^{j,i,\bl} \frac{\partial^2 p^\PM}{\partial x_i x_j}\bigg|_\Sigma
     \notag
     \\
     & \quad 
     -  \varepsilon^3 \vecee R^{\bl} \frac{\partial^2}{\partial x_1^2}  \frac{\partial v_1^\FF}{\partial x_2}\bigg|_\Sigma 
     -  \varepsilon^2  \left( \vecee L^{\bl} + \vecee E^{\bl} \right)\frac{\partial}{\partial x_1}  \frac{\partial v_1^\FF}{\partial x_2}\bigg|_\Sigma 
     - {\varepsilon^3}{\sum_{i,j=1}^2 \vecee w^{j,i,\varepsilon}}{\dpxijdom}\bigg|_\Sigma
     + \varepsilon^3 {\sum_{i,j=1}^2 \vecee W^{j,i,\varepsilon}}{\dpxij}
      \, . 
      \label{eq:trace-U-7}
\end{align}
We observe that it is not possible to formulate interface conditions for the velocity components such that $ \llbracket \vecee U^{7,\varepsilon} \rrbracket_\Sigma = \vecee 0$ since the functions $\vecee w^{j,i,\varepsilon}$, $\vecee W^{j,i,\varepsilon}$ depend on $\vecee y$ and do not cancel each other.
Thus,
%In order to set $\llbracket \vecee U^{7,\varepsilon} \rrbracket_\Sigma = \vecee 0$, %which is necessary that we get $\vecee U^{7,\varepsilon}  \in \Vper$, 
we need to eliminate the last two terms appearing in equation~\eqref{eq:trace-U-7} %since they contain functions depending on the microscale variable $\vecee y$.
such that velocity trace continuity can be enforced (see eq.~\eqref{eq:set-velocity-IC}).
We correct the trace of $\vecee U^{7,\varepsilon}$ in the same way as in~\cite[section 3.2.3]{Eggenweiler_Rybak_MMS20}, i.e., we define two boundary layer problems similar to~\eqref{eq:BLP-beta}. 
These problems are defined for $i,j=1,2$ and read
\begin{equation}\label{eq:BLP-beta-w-W}
    \begin{aligned}
    -\Delta_{\vecee y} \vecee \beta_{m}^{j,i,\bl} + \nabla_{\vecee y} \omega_{m}^{j,i,\bl} &= \vecee 0  \quad \ \ \, \text{ in } Z^{+}  \cup Z^{-} \, , 
    \\ 
    \divy \vecee \beta_{m}^{j,i,\bl} &= 0 \quad \ \ \, \text{ in } Z^{+}  \cup Z^{-} \, , 
    \\
    \big\llbracket \vecee \beta_{m}^{j,i,\bl} \big\rrbracket_S &= \vecee j_{m}^{j,i}%- \vecee w^{j,i} 
    \quad \ \ \text{on } S \, , 
    \\
    \big\llbracket ( \nabla_{\vecee y} \vecee \beta_{m}^{j,i,\bl} - \omega_{m}^{j,i,\bl} \ten{I} ) \vecee{e}_2\big\rrbracket_S &= \vecee J_{m}^{j,i}% - \left( \nabla_{\vecee y} \vecee w^{j,i} - \pi^{j,i} \ten I \right) \vecee{e}_2 
    \quad \ \ \text{on } S \, , 
    \\
    \vecee \beta_{m}^{j,i,\bl} = \vecee 0 \quad \text{on } \cup_{k=1}^{\infty}(\partial Y_\text{s} - (0,k)) \, ,& \qquad \{\vecee \beta_{m}^{j,i,\bl}, \omega_{m}^{j,i,\bl}\} \text{ is 1-periodic in $y_1$} \, ,
    \end{aligned}
\end{equation}
where the velocity jump $\vecee j_{m}^{j,i}$ and the stress jump $\vecee J_{m}^{j,i}$ depend on $m \in \{w, W\}$.
To eliminate the second last term in~\eqref{eq:trace-U-7}, we solve problems~\eqref{eq:BLP-beta-w-W} with $m=w$ and the following jumps across the interface $S$:
\begin{align}\label{eq:BLP-beta-w-jump}
    \vecee j_{w}^{j,i} = - \vecee w^{j,i} 
    \, ,\quad 
    \vecee J_{w}^{j,i} = - \left( \nabla_{\vecee y} \vecee w^{j,i} - \pi^{j,i} \ten I \right) \vecee{e}_2 %\quad \ \text{ on } S
    \, .
\end{align}
In order to correct the last term in~\eqref{eq:trace-U-7}, we use the solutions to boundary layer problems~\eqref{eq:BLP-beta-w-W} with $m=W$ and
\begin{align}\label{eq:BLP-beta-W-jump}
    \vecee j_{W}^{j,i}  = - \vecee W^{j,i} 
    \, , \quad \
    \vecee J_{W}^{j,i}  = - \left( \nabla_{\vecee y} \vecee W^{j,i} - \Lambda^{j,i} \ten I \right) \vecee{e}_2 %\quad \ \text{ on } S
    \, .
\end{align}
Clearly, problems~\eqref{eq:BLP-beta-w-W} with~\eqref{eq:BLP-beta-w-jump} respective~\eqref{eq:BLP-beta-W-jump} are boundary layer problems in the sense of~\citep{Jaeger_Mikelic_96} and we have all the boundary layer problem properties, i.e., exponential stabilization for $|y_2| \to \infty $ to constants and estimates for the solution analogous to~\eqref{eq:estimates-t}.

We denote the boundary layer constants corresponding to problems~\eqref{eq:BLP-beta-w-W} for $i,j=1,2$ and $m \in \{w,W\}$ as follows
\begin{align}
    % \vecee{M}_{\vecee w}^{i,k,\bl} &= \bigg(\int_S {\beta_{w}}_1^{i,k, \bl}(y_1, +0) \ \text{d} y_1, 0\bigg) \, ,  
    % \quad \
    % M_{\omega,\vecee w}^{i,k,\bl} = \int_S \omega_{w}^{i,k,\bl}(y_1, +0) \ \text{d} y_1  \, ,
    % \label{eq:BLP-constants-w}
    % \\
    % \vecee{M}_{\vecee W}^{i,k,\bl} &= \bigg(\int_S {\beta_{W}}_1^{i,k, \bl}(y_1, +0) \ \text{d} y_1, 0\bigg) \, ,  
    % \quad \
    % M_{\omega,\vecee W}^{i,k,\bl} = \int_S \omega_{W}^{i,k,\bl}(y_1, +0) \ \text{d} y_1  \, .
    \vecee{M}_{m}^{j,i,\bl} &= \bigg(\int_S {\beta_{m,1}}^{j,i, \bl}(y_1, +0) \ \text{d} y_1, 0\bigg)^\top \, ,  
    \quad \
    M_{\omega,m}^{j,i,\bl} = \int_S \omega_{m}^{j,i,\bl}(y_1, +0) \ \text{d} y_1  \, ,
    \label{eq:BLP-constants-w-W}
\end{align}
where we used $\vecee \beta_{m}^{j,i, \bl} = \lrp{ {\beta_{m,1}}^{j,i, \bl} , {\beta_{m,2}}^{j,i, \bl} }^\top$.
As usual, we set $\vecee \beta_{m}^{j,i,\bl,\varepsilon} (\vecee x) = \vecee \beta_{m}^{j,i,\bl} \left( \frac{\vecee x}{\varepsilon} \right)$ and $\omega_{m}^{j,i,\bl,\varepsilon} (\vecee x) = \omega_{m}^{j,i,\bl} \left( \frac{\vecee x}{\varepsilon} \right)$ for $m \in \{w, W\}$, $\vecee x \in \Omega^\varepsilon$ and $i,j=1,2$, and the velocity is extended to zero in the solid inclusions.
Then, as it is well-known for boundary layer solutions~\citep{Jaeger_Mikelic_96}, the following estimates hold
\begin{equation}\label{eq:estimates-w-W}
\begin{aligned}
     \| \vecee \beta_{m}^{j,i, \bl,\varepsilon} - \mathcal{H}(x_2)\vecee{M}_{m}^{j,i,\bl} \|_{L^2(\Omega)^2} &\leq C \varepsilon^{1/2}  \, ,
     \qquad 
     \| \nabla \vecee \beta_{m}^{j,i, \bl,\varepsilon}  \|_{L^2(\Omega)^{2\times 2}} \leq C \varepsilon^{-1/2} \, ,
     \\
     \| \omega_{m}^{j,i,\bl,\varepsilon} - \mathcal{H}(x_2) M_{\omega,m}^{j,i,\bl} \|_{L^2(\Omega^\varepsilon)} &\leq C \varepsilon^{1/2}  \, .
\end{aligned}
\end{equation}

We define new velocity and pressure error functions
\begin{align}
     \vecee U^{8,\varepsilon} &= \vecee U^{7,\varepsilon}
     - \varepsilon^3 {\sum_{i,j=1}^2 \left( \vecee \beta_{w}^{j,i,\varepsilon} - \mathcal{H}(x_2) \vecee M_{w}^{j,i,\bl} \right)}{\dpxijdom}\bigg|_\Sigma 
     + \varepsilon^3 {\sum_{i,j=1}^2 \left( \vecee \beta_{W}^{j,i,\varepsilon} - \mathcal{H}(x_2) \vecee M_{W}^{j,i,\bl} \right)}{\dpxij} 
     \, ,
     \label{eq:error-8-U}
     \\
     P^{8,\varepsilon} &= P^{7,\varepsilon} 
     - \varepsilon^2 {\sum_{i,j=1}^2 \left( \omega^{j,i,\varepsilon} - \mathcal{H}(x_2) M_{\omega, W}^{j,i,\bl} \right)}{\dpxijdom}\bigg|_\Sigma
     + \varepsilon^2 {\sum_{i,j=1}^2 \left( \omega^{j,i,\varepsilon} - \mathcal{H}(x_2) M_{\omega, W}^{j,i,\bl} \right)}{\dpxij} 
     \label{eq:error-8-P}
      \, . 
\end{align}
Due to~\eqref{eq:estimates-w-W}, both corrections to $\vecee U^{7,\varepsilon}$ in~\eqref{eq:error-8-U} and~\eqref{eq:error-8-P}
are of order $\textit{O}(\varepsilon^{7/2})$ %in the $L^2$-norm 
for the velocity and of order $\textit{O}(\varepsilon^{5/2})$ %in $L^2$ 
for the pressure. Thus, the corrections do not change the result presented in Corollary~\ref{cor-1}.

\subsubsection{Correction on the lower boundary $\Gamma_H$}
\label{sec:lower-boundary}
Next, we need to correct the values of $\vecee U^{8,\varepsilon}$ and its gradient on the lower boundary $\Gamma_H$ to make the velocity error function an element of the test function space $\Vper$. At the moment, we have
\begin{align}
    U_2^{8,\varepsilon}\big|_{\Gamma_H} &=  
     {\varepsilon^3}{\sum_{i,j=1}^2  w_2^{j,i,\varepsilon}\big|_{\Gamma_H}}{\dpxijdom}\bigg|_{\Gamma_H} 
     - \varepsilon^3 {\sum_{i,j=1}^2  W_2^{j,i,\varepsilon}\big|_{\Gamma_H}}{\dpxij}
    + \text{exponentially small terms} \, ,
    \\
    \frac{\partial U_1^{8,\varepsilon}}{\partial x_2}\bigg|_{\Gamma_H}
    &=  {\varepsilon^2}{\sum_{i,j=1}^2  \frac{\partial w_1^{j,i}}{\partial y_2} (y_1,0)}{\dpxijdom}\bigg|_{\Gamma_H} 
    +{\varepsilon^3}{\sum_{i,j=1}^2  w_1^{j,i} (y_1,0)}\frac{\partial }{\partial x_2}{\dpxijdom}\bigg|_{\Gamma_H}
    - \varepsilon^2 \sum_{i,j=1}^2  \frac{\partial W_1^{j,i}}{\partial y_2}(y_1,0){\dpxij}
    \notag
    \\
    & \quad 
     - \varepsilon^3 \sum_{i,j=1}^2  W_1^{j,i}(y_1, 0) \frac{\partial}{\partial y_2}{\dpxij} 
    + \text{exponentially small terms} \, .
\end{align}
The corrections of the outer boundary effects follow exactly the same lines as in detail presented in~\cite[section 3.2.4]{Eggenweiler_Rybak_MMS20}. In total, we construct two boundary layer problems similar to~\eqref{eq:BLP-q}. Since the resulting boundary layer correctors neither contribute to the effective interface conditions~\eqref{eq:IC-HO-normal}--\eqref{eq:IC-HO-tangential} nor play a role for later calculations, we do not provide the corresponding boundary layer problems here. 
We highlight that both boundary layer correctors are at least of order $\textit{O}(\varepsilon^{5/2})$  for the velocity and of order $\textit{O}(\varepsilon^{3/2})$  for the pressure, and hence, do not deteriorate result~\eqref{eq:cor-1}.

The velocity and pressure errors including the two above-mentioned correctors are denoted by $\vecee U^{9,\varepsilon}$ and $P^{9,\varepsilon}$.
Then, by setting
\begin{align}\label{eq:set-velocity-IC}
   \llbracket \vecee U^{9,\varepsilon} \rrbracket_\Sigma = \vecee 0 \, ,
\end{align}
we obtain the interface conditions~\eqref{eq:IC-HO-normal} and~\eqref{eq:IC-HO-tangential}, and ensure that $\vecee U^{9,\varepsilon} \in \Vper$.
Using equations~\eqref{eq:set-pressure-IC},~\eqref{eq:set-velocity-IC} and Corollary~\ref{cor-1}, we obtain the following result.
\begin{corollary}\label{cor-2}
    We have $\vecee U^{9,\varepsilon} \in \Vper$ and for all  $\vecee \varphi \in V_\text{per}(\Omega^\varepsilon)$ the following estimate holds
    \begin{align}
        \bigg| \int_{\Omega^\varepsilon} &\nabla \vecee U^{9,\varepsilon} \colon \nabla \vecee \varphi 
        - \int_{\Omega^\varepsilon} P^{9,\varepsilon} \nabla \vdot \vecee \varphi 
        \bigg|
        \leq C \varepsilon^{5/2} \normgradphi 
        + \text{exp. small terms} \, .
        \label{eq:cor-2}
    \end{align}
\end{corollary}
\begin{proof}
    We constructed $\vecee U^{9,\varepsilon}$ in such a way that $\vecee U^{9,\varepsilon} \in H^1(\Omega^\varepsilon)^2$, $\vecee U^{9,\varepsilon}$ is $L$-periodic in $x_1$ and $\vecee U^{9,\varepsilon} = \vecee 0$ on $\partial \Omega^\varepsilon\setminus \partial \Omega$. Moreover, we have $\vecee U^{9,\varepsilon} = \vecee 0 + \emph{exponentially small terms}$ on the upper boundary $\Gamma_h$ due to the exponential decay of all boundary layer velocities.
    The condition $U_2^{9,\varepsilon} = 0 + \emph{exponentially small terms}$ on $\Gamma_H$ is guaranteed by the corrections described in section~\ref{sec:lower-boundary}.
    Hence, $\vecee U^{9,\varepsilon} \in \Vper$.
    Using the results from Corollary~\ref{cor-1} and the fact that all additional correctors to $\vecee U^{7,\varepsilon}, P^{7,\varepsilon}$ are of higher-order, i.e., the terms on the right-hand side in the corresponding weak form can be estimated at least by $C \varepsilon^{5/2}\normgradphi$, estimate~\eqref{eq:cor-2} is proven.
\end{proof}

It remains to estimate the pressure error $P^{9,\varepsilon}$ through the velocity error and then use the velocity error function as a test
function in~\eqref{eq:cor-2}. However, %at this stage, 
the difficulties are coming from the compressibility effects of $\vecee U^{9,\varepsilon}$ in the term $\int_{\Omega^\varepsilon} P^{9,\varepsilon} \nabla \vdot \vecee U^{9,\varepsilon}$ in~\eqref{eq:cor-2}. At this stage, we have
\begin{align}
    \div \vecee U^{9,\varepsilon}
    =& -\varepsilon^2 \sum_{j=1}^2 \left( {\beta}_1^{j,\bl, \varepsilon} - \mathcal{H}(x_2) M_1^{j,\bl} \right) \frac{\partial}{\partial x_1} \frac{\partial p^\PM }{\partial x_j}\bigg|_\Sigma 
    - \varepsilon^2 \zeta_1^{\bl,\varepsilon}  \frac{\partial}{\partial x_1} \frac{\partial v_1^\FF}{\partial x_2}\bigg|_\Sigma
    + \varepsilon^2 \left( c_1^{\bl,\varepsilon} - \mathcal{H}(x_2) E_1^{\bl} \right) \frac{\partial^2}{\partial x_1^2}  \frac{\partial v_1^\FF}{\partial x_2}\bigg|_\Sigma 
    \notag
    \\
    &
    +\varepsilon^2 \left( \xi_1^{\bl,\varepsilon} - \mathcal{H}(x_2) L_1^{\bl} \right) \frac{\partial^2}{\partial x_1^2}  \frac{\partial v_1^\FF}{\partial x_2}\bigg|_\Sigma 
    - \mathcal{H}(-x_2) \varepsilon^3\sum_{i,j=1}^2  W^{j,i,\varepsilon}_1 \pd{}{x_1} \dpxij 
    \notag 
    \\
    &
    - \mathcal{H}(-x_2) \varepsilon^2 \sum_{j=1}^2 \underbrace{\vecee \gamma^{j,i,\varepsilon}}_{=\varepsilon\vecee \gamma^{j,i}(\vecee y)} \vdot \nabla \frac{\partial^2 p^\PM}{\partial x_i x_j} %\notag 
    % \\
    % & 
    - \mathcal{H}(-x_2) \varepsilon^3\sum_{i,j=1}^2 \vecee w^{j,i,\varepsilon}  \vdot \nabla  \dpxijdom
    %+ \textit{O}_{L^2(\Omega^\varepsilon)^2}( \varepsilon^{7/2})
    + \textit{O}( \varepsilon^{7/2})\, 
    \label{eq:div-U9}
\end{align}
leading to $\|\div \vecee U^{9,\varepsilon}\|_{L^2(\Omeps)} \leq C \varepsilon^{5/2}$. To obtain a sufficient estimate for the pressure error function, we need to derive an estimate of order $\textit{O}(\varepsilon^{7/2})$ for $\div \vecee U^{9,\varepsilon} \text{ in } L^2(\Omeps)$.
To correct the divergence of the velocity error, we need to eliminate all the terms appearing on the right-hand side of~\eqref{eq:div-U9} which is done in the next section.

\subsection{Correction of compressibility effects}
\label{sec:compressibility}
We correct the compressibility effects of the velocity error using similar constructions as in~\cite[section 3.2.5]{Eggenweiler_Rybak_MMS20} and ~\cite[section 4.4]{Carraro_etal_15}. 
For this purpose, we need to construct ten auxiliary problems in total: seven boundary layer problems and three cell problems.
For the elimination of the first four terms on the right-hand side of~\eqref{eq:div-U9}, we construct the following boundary layer problems 
\begin{equation}\label{eq:BLP-theta}
\begin{aligned}
    \divy \vecee \theta_a^{\bl} &=  a  \quad \text{ in } Z^+ \cup Z^-,
    \\
    \llbracket \vecee \theta_a^{\bl} \rrbracket_S &= - \left( \int_{Z^{\bl}}  a  \ \text{d} \vecee y \right) \vecee e_2 \quad \text{ on } S, 
    \\
    \vecee \theta_a^{\bl} &= \vecee 0 \quad \text{ on } \cup_{k=1}^{\infty}(\partial Y_\text{s} - (0,k)), \quad \vecee \theta_a^{\bl} \text{ is 1-periodic in $y_1$}\, ,
\end{aligned}
\end{equation}
where the subscript $a$ is given by
\begin{align}
    a \in A:= \bigg\{&
    \left( {\beta}_1^{j,\bl} - \mathcal{H}(x_2) M_1^{j,\bl} \right), \ \,
    \zeta_1^{\bl}, 
     \ \, -\left( c_1^{\bl} - \mathcal{H}(x_2) E_1^{\bl} \right), \ \,
    -\left( \xi_1^{\bl} - \mathcal{H}(x_2) L_1^{\bl} \right) 
    \bigg\} \, , \quad  j=1,2 
    \, . \label{eq:a}
\end{align}
After~\cite[Proposition 3.20 and 3.21]{Jaeger_Mikelic_96}, problem~\eqref{eq:BLP-theta} has at least one solution $\vecee \theta_a^{\bl} \in H^1(Z^+ \cup Z^-)^2 \cap C_\text{loc}^\infty (Z^=\cup Z^-)^2$ such that $e^{\gamma|y_2|}\vecee \theta_a^\bl \in H^1(Z^+ \cup Z^-)^2$ for some $\gamma>0$ and all $a$ given in~\eqref{eq:a}.
We set $\vecee \theta^{\bl,\varepsilon}_a(\vecee x) = \vecee \theta^\bl_a (\frac{\vecee x}{\varepsilon})$ in $\Omeps$, and from the work of~\cite{Jaeger_Mikelic_96}, we get the following estimate
\begin{align}\label{eq:estimate-theta}
\|\vecee \theta^{\bl,\varepsilon}_a\|_{L^2(\Omeps)^2} \leq C \varepsilon^{1/2} \, .
\end{align}

We introduce the boundary layer correctors obtained from solving~\eqref{eq:BLP-theta} that are added to the velocity error function $\vecee U^{9,\varepsilon}$ as 
\begin{align}\label{eq:corrector-theta}
    \varepsilon^3 \vecee \theta^{\bl,\varepsilon}_a G_a(\vecee x) \, , 
    \quad G_a \in \bigg\{ \frac{\partial}{\partial x_1} \frac{\partial p^\PM}{\partial x_j}\bigg|_\Sigma, \ \,
    \frac{\partial}{\partial x_1}\frac{\partial v_1^\FF}{\partial x_2}\bigg|_\Sigma, \ \,
    \frac{\partial^2}{\partial x_1^2}\frac{\partial v_1^\FF}{\partial x_2}\bigg|_\Sigma 
    \bigg\} \, ,
\end{align}
where $G_a$ is the same function that is multiplied with $a$ in equation~\eqref{eq:div-U9}. 
For example, for $a=\left( {\beta_1}^{j,\bl} - \mathcal{H}(x_2) M_1^{j,\bl} \right)$ we have $G_a= \frac{\partial}{\partial x_1}\frac{\partial p^\PM}{\partial x_j}\big|_\Sigma$.
% For example, the corrector for the first term on the right-hand side in~\eqref{eq:div-U9} is
% \begin{align*}
%     \varepsilon^3 \vecee \theta^{\bl,\varepsilon}_{\left( {\beta_1}^{j,\bl} - \mathcal{H}(x_2) M_1^{j,\bl} \right)} \frac{\partial}{\partial x_1} \frac{\partial p^\PM}{\partial x_j}\bigg|_\Sigma \, .
% \end{align*}
Under the assumptions from Remark~\ref{rem} and using estimate~\eqref{eq:estimate-theta}, we obtain
\begin{align}
    \bigg| \int_{\Omeps} \nabla  \left(\varepsilon^3 \vecee \theta^{\bl,\varepsilon}_a G_a(\vecee x)\right) \colon \nabla \vecee \varphi \bigg|
    \leq C \varepsilon^{5/2} \normgradphi \, .
\end{align}
The divergence corresponding to correctors~\eqref{eq:corrector-theta} is 
\begin{align}
    \nabla \vdot \left( \varepsilon^3 \vecee \theta^{\bl,\varepsilon}_a G_a(\vecee x) \right)
    % &=
    % \varepsilon^2 \nabla_{\vecee y} \vdot \vecee \theta^{\bl,\varepsilon}_a G_a(\vecee x) 
    % +
    %  \varepsilon^3 \vecee \theta^{\bl,\varepsilon}_a \nabla \vdot G_a(\vecee x) 
    % \notag 
    % \\
    &=
    \varepsilon^2 a(\vecee y) G_a(\vecee x) 
    +
     \varepsilon^3 \vecee \theta^{\bl,\varepsilon}_a \nabla \vdot G_a(\vecee x) \qquad \text{in } \Omeps
     \, ,
     \label{eq:divergence-theta}
\end{align}
where the first term on the right-hand side of~\eqref{eq:divergence-theta} corresponds to one of the first four terms on the right-hand side of~\eqref{eq:div-U9} dependent on $a \in A$. Thus, when we add all correctors~\eqref{eq:corrector-theta} to the velocity error function, all the first four terms on the right-hand side of~\eqref{eq:div-U9} vanish.

In order to correct the last three terms in~\eqref{eq:div-U9}, we define the following cell problems
\begin{equation}\label{eq:cell-alpha}
\begin{aligned}
    \divy \vecee  \alpha_{b} &= b \quad \text{ in } Y_\text{f} \, ,
    \\
    \vecee \alpha_{b} &= \vecee 0 \quad \text{ on } \partial Y_\text{s} \, ,
    \quad \vecee \alpha_{b} \text{ is $1$-periodic in $\vecee y$}\, ,
\end{aligned}
\end{equation}
for $b \in B:=\big\{ W_1^{j,i}, \gamma_k^{j,i}, w_k^{j,i} \big\}$ and $i,j,k=1,2$.
Here, $\vecee W^{j,i}=( W_1^{j,i},  W_2^{j,i})^\top$ is the solution to~\eqref{eq:BLP-beta-w-W} with the velocity and stress jump conditions~\eqref{eq:BLP-beta-W-jump}, $\vecee \gamma^{j,i} = (\gamma_1^{j,i}, \gamma_2^{j,i})^\top$ is the solution to~\eqref{eq:cell-problems-gamma}, and $\vecee w^{j,i} = (w_1^{j,i}, w_2^{j,i})^\top$ is the solution to~\eqref{eq:BLP-beta-w-W} with~\eqref{eq:BLP-beta-w-jump}.
%The existence of at least one $\vecee \alpha_b \in H^1(Y_\text{f})^2$ satisfying~\eqref{eq:cell-alpha} (see~\cite[section 1.2.2]{Jaeger_Mikelic_96}).
%
Setting $\vecee  \alpha_b^\varepsilon (\vecee x) = \vecee  \alpha_b(\frac{\vecee x}{\varepsilon})$ for $\vecee x \in \Omepspm$, we construct the correctors which contribute to the velocity error function as 
\begin{align}
   \mathcal{H}(-x_2)\varepsilon^4 \vecee  \alpha_b^\varepsilon F_b(\vecee x)
\end{align}
with $F_b \in \big\{ \frac{\partial}{\partial x_1}\frac{\partial^2 p^\PM}{\partial x_i x_j}\big|_\Sigma, \frac{\partial^3 p^\PM}{\partial x_ix_jx_k} \big\}$, $i,j,k=1,2$, chosen as the same factor which is multiplied with $b$ in~\eqref{eq:div-U9}. 
Using Remark~\ref{rem} and the fact that $\|\vecee \alpha_b^{\varepsilon}\|_{L^2(\Omepspm)^2} \leq C$, we obtain the following estimate
\begin{align}\label{eq:estimate-alpha}
    \bigg| \intpmvel{\varepsilon^4}{ \vecee  \alpha_b^{\varepsilon}}{F_b(\vecee x)} \bigg| \leq C\varepsilon^3 \normgradphipm \, .
\end{align}

Moreover, we have
\begin{align}
    \div \left( \varepsilon^4\vecee  \alpha_b^{\varepsilon} F_b(\vecee x) \right) 
    &= \varepsilon^3 \underbrace{\varepsilon \div   \vecee  \alpha_b^{\varepsilon}F_b(\vecee x) }_{= b F_b(\vecee x) }
    +
    \varepsilon^4 \vecee \alpha_b^{\varepsilon} \vdot \nabla F_b(\vecee x)
    \, \qquad \text{in } \Omepspm \, .
    \label{eq:divergence-alpha}
\end{align}
Thus, the first term on the right-hand side of~\eqref{eq:divergence-alpha} cancels with the corresponding term on the right-hand side of~\eqref{eq:div-U9} when adding the corrector $\mathcal{H}(-x_2)\varepsilon^4\vecee  \alpha_b^{\varepsilon} F_b(\vecee x) $ to the velocity error $\vecee U^{9,\varepsilon}$.

Since the problems given by~\eqref{eq:cell-alpha} are defined only in the porous-medium domain $\Omepspm$, auxiliary boundary layer velocities and pressures correcting their values on the interface~$\Sigma$ are needed.
These problems are of the same form as~\eqref{eq:BLP-gamma} but with different jump conditions across the interface. The construction of these problems is analogous to those in~\citep{Eggenweiler_Rybak_MMS20,Carraro_etal_15} and we refer the reader therein.
However, we note that the corresponding boundary layer correctors are of higher order and do not contribute to the interface conditions nor to the result from Corollary~\ref{cor-2}.
% Note that our corrections are of higher order and do not affect the new interface conditions.
In the following, we denote these additional velocity trace correctors by $\{ \vecee v^{\bl,\text{trace}}, p^{\bl,\text{trace}}\}$
%'\emph{tr. cor. $\vecee  \alpha_b^\varepsilon$}' (trace correctors w.r.t $\vecee  \alpha_b^\varepsilon$) 
in the velocity and pressure error functions~\eqref{eq:error-10-U} and~\eqref{eq:error-10-P}.

We define new error functions
\begin{align}
    \vecee U^{10,\varepsilon} &= \vecee U^{9,\varepsilon} 
    + \varepsilon^3 \sum_{a \in A} \vecee \theta^{\bl,\varepsilon}_a G_a(\vecee x)
    + \mathcal{H}(-x_2) \varepsilon^4 \sum_{b \in B} \vecee  \alpha_b^\varepsilon F(\vecee x)
    %+ \emph{tr. cor. $\vecee  \alpha_b^\varepsilon$} 
    + \vecee v^{\bl,\text{trace}}
    \, ,
    \label{eq:error-10-U}
    \\
    P^{10,\varepsilon} &= P^{9,\varepsilon} 
    %+ \emph{tr. cor. $\vecee  \alpha_b^\varepsilon$} 
    + p^{\bl,\text{trace}}
    \, ,
    \label{eq:error-10-P}
\end{align}
and we prove the following result.
\begin{corollary}
    It is $\vecee U^{10,\varepsilon} \in \Vper$, and for all  $\vecee \varphi \in V_\text{per}(\Omega^\varepsilon)$ it holds
\begin{align}
    \bigg| \int_{\Omega^\varepsilon} \nabla \vecee U^{10,\varepsilon} \colon \nabla \vecee \varphi 
    - \int_{\Omega^\varepsilon} P^{10,\varepsilon} \nabla \vdot \vecee \varphi 
    \bigg|
    &\leq C \varepsilon^{5/2} \normgradphi  + \text{exp. small terms}\, ,
    \label{eq:cor-3-1}
    \\
    \| \div \vecee U^{10,\varepsilon} \|_{L^2(\Omeps)} &\leq C \varepsilon^{7/2} \, .
    \label{eq:cor-3-2}
\end{align}
%Inserting $\vecee U^{10,\varepsilon}$ as a test function in~\eqref{eq:cor-3-1} and using~\eqref{eq:cor-3-2}, we get
Moreover, we get
\begin{align}
    \int_{\Omega^\varepsilon} \! |\nabla \vecee U^{10,\varepsilon}|^2
    &\leq C \varepsilon^{5/2} \! \left( \varepsilon \| P^{10,\varepsilon} \|_{L^2(\Omeps)}  +  \norm{\nabla \vecee  U^{10,\varepsilon} }_{L^2(\Omega^\varepsilon)^{{2\times2}}} \right)
    \! + \text{exp. small terms} \, .
    \label{eq:cor-3-3}
\end{align}
\end{corollary}
\begin{proof}
    By construction, we have $\vecee U^{10,\varepsilon} \in \Vper$ since all correctors added to $\vecee U^{9,\varepsilon}$ in~\eqref{eq:error-10-U} are either of boundary layer type (i.e. $\vecee \theta^{\bl,\varepsilon}_a$), vanishing away from the interface and having a constant value at the interface, or are derived from cell problems (i.e. $\vecee a^\varepsilon_b$) for which the values at the interface are corrected.
    Then, estimate~\eqref{eq:cor-3-1} is a direct consequence of~\eqref{eq:cor-2},~\eqref{eq:estimate-theta} and~\eqref{eq:estimate-alpha}.
    Combining equations~\eqref{eq:div-U9},~\eqref{eq:divergence-theta} and~\eqref{eq:divergence-alpha} yields estimate~\eqref{eq:cor-3-2}.
    Furthermore, by applying the Cauchy-Schwarz inequality to~\eqref{eq:cor-3-1}, inserting $\vecee U^{10,\varepsilon}$ as a test function in~\eqref{eq:cor-3-1} and using the result given in~\eqref{eq:cor-3-2}, we obtain~\eqref{eq:cor-3-3}.
\end{proof}

As the next step, we estimate the pressure error function~$P^{10,\varepsilon}$ using the velocity error function $\vecee U^{10,\varepsilon}$.

\subsection{Pressure estimates}
\label{sec:pressure-estimate}
In this section, we derive an estimate for the pressure error~$P^{10,\varepsilon}$. For this purpose, we follow the ideas presented in~\cite[section 4.5]{Carraro_etal_15} and use some of the proven results therein.
This is possible since the pore-scale problem~\eqref{eq:pore-scale} is of the same form as the one in~\citep{Carraro_etal_15}, only with different boundary conditions on the upper and lower boundaries~$\Gamma_h$ and $\Gamma_H$. 
We extend the pressure error function $P^{10,\varepsilon}$ defined in $\Omeps$ to $\widetilde P^{10,\varepsilon}$ defined in $\Omega$ in a standard way~\citep[eq. (1.23)]{Hornung_97}: 
\begin{align}
    \widetilde P^{10,\varepsilon} = \begin{cases}
      P^{10,\varepsilon} & \text{in } \Omeps \, ,
      \\
      \frac{1}{\varepsilon^2|Y_\text{f}|} \displaystyle
      \int_{\varepsilon (Y_\text{f} -(k_1,k_2) )} P^{10,\varepsilon} \, \mathrm{d}\vecee y &  \text{in } \Omega \setminus\Omeps \, ,
    \end{cases}    
    \label{eq:pressure-extension}
\end{align}  
where $k_1, k_2 \in \mathbb{N}$.
Then, Proposition~14 from~\citep{Carraro_etal_15} yields
\begin{align}\label{eq:pressure-estimate-in-Omega}
    \| \widetilde P^{10,\varepsilon} \|_{L^2(\Omega)} \leq \frac{C}{\varepsilon} \left( \norm{\nabla \vecee  U^{10,\varepsilon} }_{L^2(\Omega^\varepsilon)^{{2\times2}}} + C \varepsilon^{5/2} \right) \, .
    % \bigg\{  &
    % \norm{\nabla \vecee  U^{10,\varepsilon} }_{L^2(\Omega^\varepsilon)^{{2\times2}}} 
    % +  \norm{ }_{L^2(\Omega^\varepsilon)^2}
    % +  \varepsilon \norm{ }_{L^2(\Omega^\varepsilon)^2}
    % \notag
    % \\
    % & +  \norm{ }_{L^2(\Omega^\varepsilon)^{{2\times2}}}
    % + \sqrt{\varepsilon} \left( ... \right)
    % \bigg\}
\end{align}
In the next section, we insert~\eqref{eq:pressure-estimate-in-Omega} in~\eqref{eq:cor-3-3} in order to prove Theorem~\ref{theo}, i.e., to obtain rigorous estimates for the velocity and pressure error functions.

\subsection{Global energy estimates (proof of Theorem~\ref{theo})}
\label{sec:global-energy-estimates}
In this section, we prove the results presented in Theorem~\ref{theo}. We derive error estimates for the velocity error function, its gradient, and the pressure error.
Using estimate~\eqref{eq:pressure-estimate-in-Omega} in~\eqref{eq:cor-3-3}, we get
\begin{align}\label{eq:estimate-grad-velocity}
    \norm{\nabla \vecee  U^{10,\varepsilon} }^2_{L^2(\Omega^\varepsilon)^{{2\times2}}} 
    &\leq 
    %\frac{C}{\varepsilon} \varepsilon^{7/2}  \left( \norm{\nabla \vecee  U^{10,\varepsilon} }_{L^2(\Omega^\varepsilon)^{{2\times2}}} + C \varepsilon^{5/2} \right)  + C \varepsilon^{5/2} \norm{\nabla \vecee  U^{10,\varepsilon} }_{L^2(\Omega^\varepsilon)^{{2\times2}}} =
    C \varepsilon^{5/2}   \norm{\nabla \vecee  U^{10,\varepsilon} }_{L^2(\Omega^\varepsilon)^{{2\times2}}} + C \varepsilon^{5/2}   \, ,
\end{align}
which leads to
\begin{align}
    \norm{\nabla \vecee  U^{10,\varepsilon} }_{L^2(\Omega^\varepsilon)^{{2\times2}}} \leq C \varepsilon^{5/2} \, .
\end{align}
Thus, the estimate for the pressure is 
\begin{align}
    \| \widetilde P^{10,\varepsilon} \|_{L^2(\Omega)} \leq C \varepsilon^{3/2} \, .
\end{align}
Using the classical Poincaré inequality, estimate~\eqref{eq:estimate-grad-velocity} and the inequalities given in~\eqref{eq:Poincare-estimate}, we obtain
\begin{align}\label{eq:estimate-velocity}
    \norm{\vecee  U^{10,\varepsilon} }_{L^2(\Omega_\FF)^2} \leq C \varepsilon^{5/2} \, , \quad \
    \norm{\vecee  U^{10,\varepsilon} }_{L^2(\Omepspm)^2} \leq C \varepsilon^{7/2} \, , \quad \ 
    \norm{\vecee  U^{10,\varepsilon} }_{L^2(\Sigma)^2} \leq C \varepsilon^{3}  \, .
\end{align}
Note that the $L^2$-estimate for the velocity $\vecee  U^{10,\varepsilon}$ in the free-flow region $\Omega_\FF$ could also be improved by the order of $\varepsilon^{1/2}$ using the theory of very weak solutions, similar as it is done in~\citep{Carraro_etal_15}. 
However, since estimates~\eqref{eq:estimate-velocity} are already accurate enough, we do not provide the steps for the improvement of estimating $\norm{\vecee  U^{10,\varepsilon} }_{L^2(\Omega_\FF)^2}$ in this work.

\section{Numerical results}\label{sec:numerics}

In this section we provide a test case to validate the proposed coupling conditions. We  
%~\eqref{eq:IC-HO-normal}--\eqref{eq:IC-HO-tangential} 
compare numerical simulation results for the macroscale problem using the new interface conditions and the averaged pore-scale resolved results. We impose boundary conditions such that we obtain the flow, which approaches the interface with at arbitrary angles.
For numerical results, we reformulate the new conditions and neglect terms of order $\varepsilon^3$ for the velocity and $\varepsilon^2$ for the pressure (see~\ref{appendix:reformulated-IC}).
% We consider a coupled problem, where the fluid flow meets the interface with different angles.
% We consider a coupled problem, where the fluid flow approaches the interface with arbitrary flow directions.
We also demonstrate advantages of the new conditions~\eqref{eq:IC-HO-normal-reform}--\eqref{eq:IC-HO-tangential-reform} in comparison to the generalized interface conditions~\eqref{eq:ER-mass}--\eqref{eq:ER-tangential} and the classical conditions~\eqref{eq:IC-classical-mass}--\eqref{eq:IC-classical-BJJ}. In sections that follow we first briefly provide the discretization scheme and the numerical solution algorithm. Then we describe the test case, and finally illustrate the velocity and pressure profiles. 

\subsection{Discretization scheme and numerical solution method}
The pore-scale and macroscale numerical simulations are conducted using the open source PDE solver~\textsc{FreeFEM++} \cite{Hecht_2012}. Here, we use the Taylor--Hood elements (P2/P1).
%
% The pore-scale problem~\eqref{eq:pore-scale} is solved using the software package \textsc{FreeFEM++}~\cite{Hecht_2012} with the Taylor--Hood (P2/P1) finite elements. 
%
%The pore-scale problem~\eqref{eq:pore-scale} is solved using the Taylor--Hood (P2/P1) finite elements. 
% To compare pore-scale and macroscale simulation results fairly, apply ensemble averaging~\citep{Sudhakar_etal_2021,Eggenweiler_Rybak_MMS20} to average out the microscopic effects in the pore-scale results.
Because of the solid inclusions in the porous-medium domain the pore-scale simulation results contain physical oscillations.
Therefore, it is necessary to perform averaging of the pore-scale simulation results in order to facilitate their comparison with macroscale simulations. 
In this work, we apply ensemble averaging~\citep{Eggenweiler_Rybak_MMS20,Sudhakar_etal_2021} to mitigate the microscopic effects present in the pore-scale results.
The idea of ensemble averaging is as follows: 
First we generate 50 samples of pore-scale simulation results by shifting solid inclusions horizontally at $\varepsilon/50$. Secondly, we compute the averaged pore-scale velocity and pressure as the arithmetic mean of the velocities and pressures obtained from these $50$ samples.
%
%The macroscale problem~\eqref{eq:macro-FF}--\eqref{eq:macro-PM} with the different sets of interface conditions is discretized using the finite volume method on staggered grids and solved using our in-house C++ code. Hereby, the free-flow and porous-medium domains are partitioned into squares with length $h=??$  and the meshes are conforming at the interface $\Sigma$.

Effective model parameters (permeability of the porous medium, boundary layer coefficients appearing in the coupling conditions) are computed from the solutions to cell and boundary layer problems. 
% Since we cannot handle infinitely long domains numerically, instead of the boundary layer stripe $Z^\bl$ we consider a cut-off stripe $Z^k := Z^\bl \cap \{(0,1) \times (-k,k)\}$ (figure~\ref{fig:BL-stripe-flow-field}, left) with $k=4$ as it is done in~\citep{Jaeger_Mikelic_00,Carraro_etal_15}.
% Then, boundary layer problems~\eqref{eq:BLP-t},~\eqref{eq:BLP-beta},~\eqref{eq:BLP-zeta},~\eqref{eq:BLP-xi} and~\eqref{eq:BLP-c} are solved in the cut-off stripe~$Z^4$ and the resulting solutions are used to compute the boundary layer constants given in~\eqref{eq:BLP-constants-N},~\eqref{eq:BLP-constants-M},~\eqref{eq:BLP-constant-W},~\eqref{eq:BLP-constants-E} and~\eqref{eq:BLP-constants-L}. 
Since we cannot handle infinitely long domains numerically, we consider a cut-off stripe $Z^m := Z^\bl \cap \{(0,1) \times (-m,m)\}$, $m \in \mathbb{N}$ (figure~\ref{fig:BL-stripe-flow-field}, left) as it is done in~\cite{Carraro_etal_15,Jaeger_Mikelic_00}.
We solve the boundary layer problems~\eqref{eq:BLP-t},~\eqref{eq:BLP-beta},~\eqref{eq:BLP-zeta}~\eqref{eq:BLP-xi} and~\eqref{eq:BLP-c} in the cut-off stripe~$Z^4$ and use the resulting solutions to compute the boundary layer constants given in~\eqref{eq:BLP-constants-N},~\eqref{eq:BLP-constants-M},~\eqref{eq:BLP-constant-W},~\eqref{eq:BLP-constants-L} and~\eqref{eq:BLP-constants-E}. 
All these computations are also performed with the solver \textsc{FreeFEM++}.
%The cell problems~\eqref{eq:cell-problems} defined within the fluid part $Y_\text{f}$ of the unit cell are solved using \textsc{FreeFEM++}.
%For the computation of boundary layer constants we follow the ideas from~\citep{Jaeger_Mikelic_00,Carraro_etal_15}. 
% Since we cannot handle infinitely long domains numerically, instead of the boundary layer stripe $Z^\bl$ we consider a cut-off stripe $Z^k := Z^\bl \cap \{(0,1) \times (-k,k)\}$ (figure~\ref{fig:BL-stripe-flow-field}, left) with $k=4$ as it is done in~\citep{Jaeger_Mikelic_00,Carraro_etal_15}. 

\subsection{Test case setting}
In this work we examine the lid-driven cavity problem above the porous layer that has been extensively studied and serves as a benchmark for assessing fluid--porous interface conditions~\cite{Lacis_Bagheri_17,Rybak_etal_21,Sudhakar_etal_2021}. The more detailed validation study in beyond the scope of this manuscript and is subject to additional paper. 

We consider the flow domains $\Omega_\FF=(0,1) \times (0,0.5)$ and $\Omega_\PM=(0,1) \times (-0.5,0)$ separated by the fluid--porous  interface $\Sigma=(0,1) \times \{0\}$. At the pore scale, the consider $20 \times 10$ in-line arranged circular solid obstacles with diameter $d^\varepsilon=d \varepsilon$, where $\varepsilon=1/20$ and $d=0.5$ is diameter of inclusion within the unit cell and the cut-off stripe~$Z^4$ (figure~\ref{fig:BL-stripe-flow-field}, left). 
% By solving the cell problems~\eqref{eq:cell-problems}in the corresponding unit cell
% , we obtain 
% \[
% \ten K = \begin{pmatrix} 1.99 \mathrm{e-}2 & 0 \\ 0 & 1.99 \mathrm{e-}2 
% \end{pmatrix} \, .
% \]
% The boundary layer constants appearing in the higher-order interface conditions~\eqref{eq:IC-HO-normal}--\eqref{eq:IC-HO-tangential} for the considered geometry are provided in table~\ref{tab:BL-constants}.
The computed values of permeability $\ten K = k \ten I$ and the boundary layer coefficients in the higher-order interface conditions~\eqref{eq:IC-HO-normal-reform}--\eqref{eq:IC-HO-tangential-reform} are provided in table~\ref{tab:BL-constants} for the considered geometry.

\begin{table}[h]
    \centering
    \begin{tabular}{|c | c | c | c | c | c | c | c | c | c | c | c | c |}
     %\multicolumn{11}{c}{Boundary layer constants for circular solid inclusions ($d=0.5$)} \\ 
     \hline 
        %& 
        $k$ & $N_1^\bl$ & $N_s^\bl$ &  
        $M_1^{1,\bl}$ & $M_\omega^{1,\bl}$ & 
        $M_1^{2,\bl}$ & $M_\omega^{2,\bl}$ &  
        $E_1^{\bl}$ & $E_b^{\bl}$ &  
        $L_1^{\bl}$ & $L_\eta^{\bl}$ &  
        $W^\bl$ 
        \\%[1ex]
        \hline \hline
        %$d=0.5$ & 
        $1.99 \mathrm{e-}2$ & $-3.04 \mathrm{e-}1$ & $0$ &  
        $-4.76\mathrm{e-}2$ & $0$ & 
        $0$ & $2.58\mathrm{e-}2$ &  
        $0$ & $3.04\mathrm{e-}1$ &  
        $0$ & $4.70\mathrm{e-}3$ & 
        $4.76\mathrm{e-}2$ 
         \\ \hline 
         %\hline
    \end{tabular}
    \caption{Effective coefficients appearing in conditions~\eqref{eq:IC-HO-normal}--\eqref{eq:IC-HO-tangential} and~\eqref{eq:IC-HO-normal-reform}--\eqref{eq:IC-HO-tangential-reform} in case of circular solid inclusions with diameter $d=0.5$.}
    \label{tab:BL-constants}
\end{table}

\begin{figure}
\centering
    \includegraphics[width=0.23\textwidth]{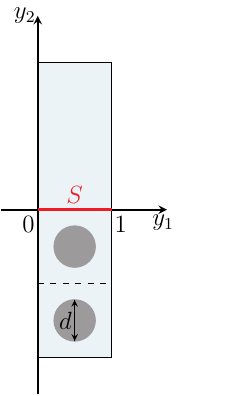}
    \qquad 
    \includegraphics[width=0.5\textwidth]{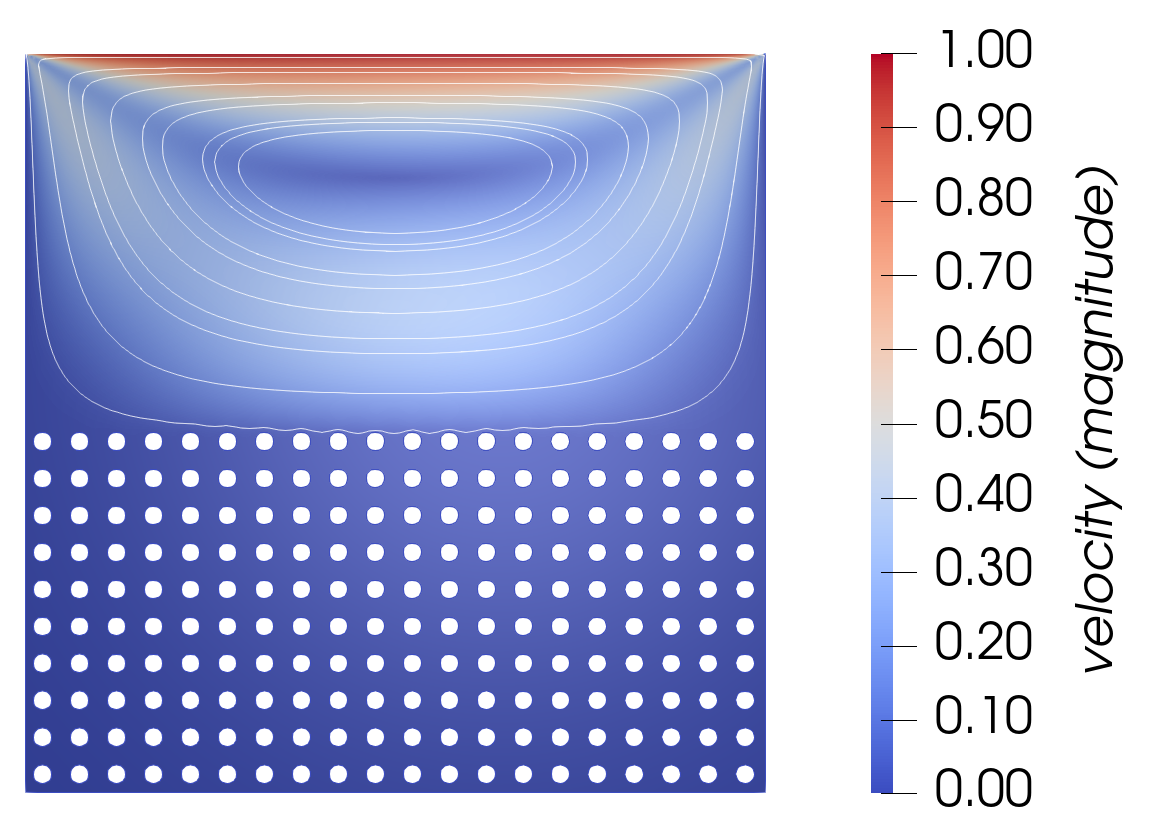}
    \caption{Cut-off stripe $Z^m$ for $m=2$ (left) and magnitude of pore-scale velocity (right).}
    \label{fig:BL-stripe-flow-field}
\end{figure}

% We consider the lid-driven cavity flow over a porous bed since this problem is well-studied and  a classical benchmark problem for effective interface conditions~\citep{Lacis_Bagheri_17,Rybak_etal_21,Sudhakar_etal_2021}. 
%We examine the lid-driven cavity flow over a porous bed, as this problem has been extensively studied and serves as a benchmark for assessing fluid--porous interface conditions~\citep{Lacis_Bagheri_17,Rybak_etal_21,Sudhakar_etal_2021}.
%
For the microscopic flow model~\eqref{eq:pore-scale}, we impose the following external boundary conditions 
\begin{align}
    \vecee v^\varepsilon = (1, 0)^\top \quad \text{ on } \Gamma_\text{top} \, ,
    \quad \ \ 
    \vecee v^\varepsilon = \vecee 0 \quad \text{ on } \Gamma_\text{wall} \, ,
    \quad \ \
    \left( \nabla \vecee v^\varepsilon - p^\varepsilon\ten I \right)\vecee n = \vecee 0 \quad \text{ on } \Gamma_\text{bot} \, ,
    \label{eq:LDC-BC-pore-scale}
\end{align}
where $\Gamma_\text{top} = (0,1) \times \{0.5\}$, $\Gamma_\text{bot} = (0,1) \times \{-0.5\}$ and $\Gamma_\text{wall} = \partial \Omega \setminus (\Gamma_\text{top} \cup \Gamma_\text{bot})$.
%The magnitude of the corresponding pore-scale velocity is presented in figure~\ref{fig:BL-stripe-flow-field} (right). 
%We see that the flow is not parallel to the fluid--porous interface. 
Note that in this way we obtain the flow, which is non-parallel to the fluid--porous interface (figure~\ref{fig:BL-stripe-flow-field}, right) and both velocity components are nonzero away from $x_1=0.5$.

For the coupled macroscale Stokes--Darcy problem~\eqref{eq:macro-FF},~\eqref{eq:macro-PM} the boundary conditions read
\begin{align}
    \vecee v^\FF &= (1, 0)^\top \quad \text{ on } \Gamma_\text{top} \, ,
    \qquad \ \
    \vecee v^\FF = \vecee 0 \quad \text{ on } \partial \Omega_\FF \setminus (\Gamma_\text{top} \cup \Sigma) \, ,
    \label{eq:LDC-BC-FF}
    \\
    p^\PM &= 0   \quad \qquad \ \text{ on } \Gamma_\text{bot} \, , 
    \quad   \,
    \vecee v^\PM \vdot \vecee n = 0 \quad \text{ on } \partial \Omega_\PM \setminus (\Sigma \cup \Gamma_\text{bot}) 
     \, .
     \label{eq:LDC-BC-PM}
\end{align}
To complete the macroscale model formulation, we apply different sets of interface conditions. These include the classical conditions~\eqref{eq:IC-classical-mass}--\eqref{eq:IC-classical-BJJ}, the generalized conditions~\eqref{eq:ER-mass}--\eqref{eq:ER-tangential} and the newly derived higher-order coupling conditions in their reformulated version~\eqref{eq:IC-HO-normal-reform}--\eqref{eq:IC-HO-tangential-reform} as explained in~\ref{appendix:reformulated-IC}.

%Note that the lid-driven cavity flow problem does not fit into the setting from the theoretical derivation due to nonperiodic boundary conditions on the lateral boundaries. However, the assumption of periodicity, making the theory less complicated, can be relaxed for numerical simulations. 
%In the next section, we show that the Stokes--Darcy model~\eqref{eq:macro-FF},~\eqref{eq:macro-PM},~\eqref{eq:LDC-BC-FF} and~\eqref{eq:LDC-BC-PM} with the new interface conditions~\eqref{eq:IC-HO-normal-reform}--\eqref{eq:IC-HO-tangential-reform} is best suited to describe %the lid-driven cavity flow over a porous medium.

\subsection{Numerical simulation results}
In figure~\ref{fig:cross-section}, we provide velocity %and pressure 
profiles along fixed cross-sections corresponding to the pore-scale model~\eqref{eq:pore-scale},~\eqref{eq:LDC-BC-pore-scale} and the Stokes--Darcy model~\eqref{eq:macro-FF},~\eqref{eq:macro-PM},~\eqref{eq:LDC-BC-FF} and~\eqref{eq:LDC-BC-PM} with three different sets of interface conditions (classical, generalized, higher-order).
Horizontal velocity profiles along the interface at $x_2=0$ are presented in figure~\ref{fig:cross-section} (left). We observe that the profiles corresponding to the macroscale model with both the generalized and the higher-order conditions fit very well with the averaged pore-scale velocity profile. However, this is not the case for the classical interface conditions. 
%The good fit for the two sets of interface conditions~\eqref{eq:ER-mass}--\eqref{eq:ER-tangential} and~\eqref{eq:IC-HO-normal}--\eqref{eq:IC-HO-tangential} is what we also expect since the additional contribution in condition~\eqref{eq:IC-HO-tangential} is rather small.
Profiles for the vertical velocity component along the interface ($x_2=0$) are shown in figure~\ref{fig:cross-section} (right). %Here, the differences between the macroscale model with the classical or the generalized interface conditions compared to the newly developed higher-order conditions are dramatic.
Here, the disparities between the macroscale model using either the classical or the generalized interface conditions when compared to the newly developed higher-order conditions are significant.
%We see that the only macroscale model that is able to capture the behavior of the normal velocity along the fluid--porous interface is the Stokes--Darcy model with the higher-order coupling conditions~\eqref{eq:IC-HO-normal}--\eqref{eq:IC-HO-tangential}. 
It becomes evident that the Stokes--Darcy model with the coupling conditions~\eqref{eq:IC-HO-normal-reform}--\eqref{eq:IC-HO-tangential-reform} is the most accurate one. % in terms of the vertical velocity component along the fluid--porous interface. 
%At the bottom of figure~\ref{fig:cross-section}, we present normal velocity profiles and pressure profiles at the cross-section at $x_1=0.7$. We see that for these cross-sections both the ....

\begin{figure}[ht]
    \centering
    \includegraphics[scale=0.75]{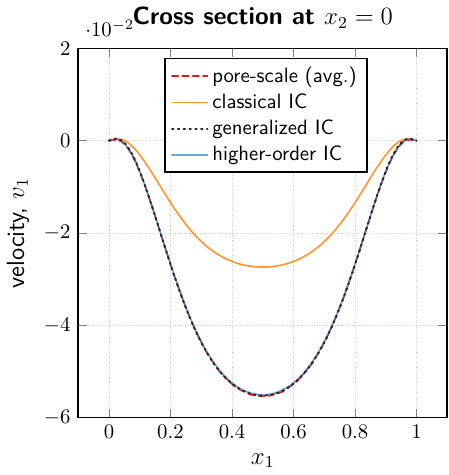} \quad 
    \includegraphics[scale=0.75]{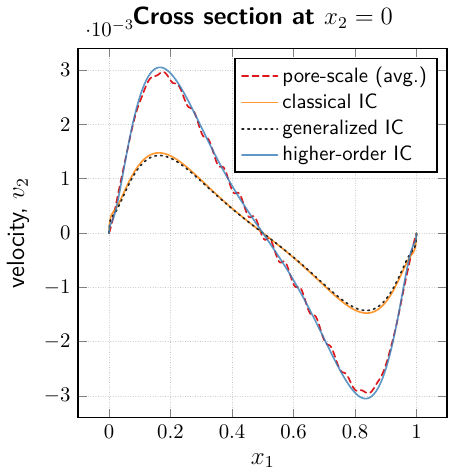}
    % \\[1ex]
    % \includegraphics[scale=0.75]{v_x07_FreeFEM.pdf} \quad
    % \includegraphics[scale=0.75]{p_x07_FreeFEM.pdf}
    \caption{Horizontal (left) and vertical (right) velocity profiles at the cross-section $x_2=0$.}
    \label{fig:cross-section}
\end{figure}

For vertical cross-sections at $x_1=0.7$, where the flow is non-parallel to the interface (figure~\ref{fig:cross-section-x1}), we observe almost no differences between the macroscale velocity and pressure profiles in the case of the generalized and higher-order interface conditions.
%Both profiles are in very good agreement with the pore-scale result as already observed in~\cite{Eggenweiler_Rybak_MMS20}. 
These profiles closely align with the pore-scale results. %, as previously shown for the generalized conditions in~\citep{Eggenweiler_Rybak_MMS20}.
The classical interface conditions are not very accurate %in describing the fluid flow 
for such flow problems, especially for the pressure profile we observe significant deviations from the pore-scale simulation results (figure~\ref{fig:cross-section-x1}, right).

\begin{figure}[t]
    \centering
    \includegraphics[scale=0.75]{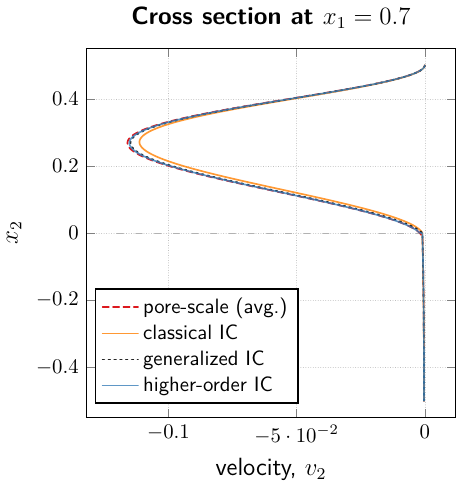} \quad
    \includegraphics[scale=0.75]{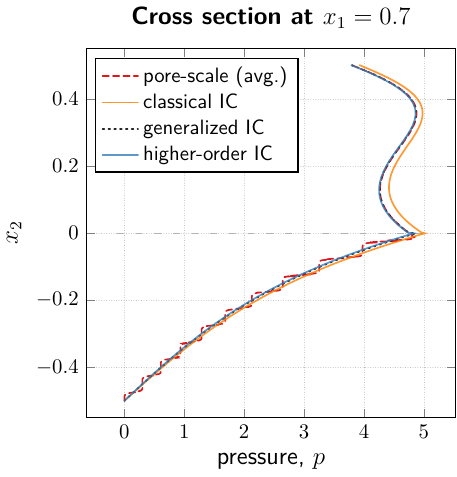}
    \caption{Vertical velocity (left) and pressure (right) profiles at the cross-section $x_1=0.7$.}
    \label{fig:cross-section-x1}
\end{figure}

To summarize, the Stokes--Darcy model with the higher-order interface conditions~\eqref{eq:IC-HO-normal-reform}--\eqref{eq:IC-HO-tangential-reform} is the most accurate macroscale model to describe fluid flows in coupled systems. %We studied the lid-driven cavity flow over a porous bed and observed that the macroscale model with the higher-order interface conditions~\eqref{eq:IC-HO-normal-reform}--\eqref{eq:IC-HO-tangential-reform} is the most accurate macroscale model to describe the fluid flow.
The classical conditions~\eqref{eq:IC-classical-mass}--\eqref{eq:IC-classical-BJJ}  %are not able to capture the flow behavior. 
%have certain limitations in capturing the flow behavior.
%
%have certain constraints 
are not well-suited for the accurate representation of the flow behavior.
The generalized conditions~\eqref{eq:ER-mass}--\eqref{eq:ER-tangential} perform significantly better than the classical ones but not as accurate as the higher-order conditions.
The importance of the additional terms appearing in the newly developed conditions~\eqref{eq:IC-HO-normal-reform}--\eqref{eq:IC-HO-tangential-reform} in comparison to the generalized interface conditions~\eqref{eq:ER-mass}--\eqref{eq:ER-tangential} becomes obvious when considering the normal velocity component along the interface.

%In this section, we showed that the Stokes--Darcy model~\eqref{eq:macro-FF},~\eqref{eq:macro-PM},~\eqref{eq:LDC-BC-FF} and~\eqref{eq:LDC-BC-PM} with the new interface conditions~\eqref{eq:IC-HO-normal-reform}--\eqref{eq:IC-HO-tangential-reform} is best suited to describe 
%the fluid flow in coupled free-flow and porous-medium flow systems.

\section{Conclusions}
\label{sec:conclusion}
In this paper, we have rigorously derived higher-order interface conditions to couple the Stokes and Darcy flow equations.
These coupling conditions are developed for arbitrary flows to the fluid--porous interface applying periodic homogenization and boundary layer theory. 
Rigorous error estimates that justify the obtained homogenization result are proven under an assumption on regularity and boundedness of the free-flow velocity and the porous-medium pressure.
All effective coefficients appearing in the developed coupling conditions are computed numerically based on the pore geometry.
We provide numerical simulation results that show the validity of the new interface conditions for arbitrary flows to the interface as well as their advantage in comparison to existing coupling concepts.

\appendix

% \bmsection{Estimates\label{app1}}
% \vspace*{12pt}

\renewcommand{\thesection}{\Alph{section}}
\renewcommand{\thesection}{\appendixname\ \Alph{section}}

\section{Estimates\label{app1}}
For the estimation of integral terms appearing in weak formulations, we use the Poincaré inequality and the results from \cite[Lemma 4.10]{Espedal_98}: \\
For $\varphi \in H^1(\Omega^\varepsilon_\PM)$ with $\varphi= 0$ on $\partial \Omega^\varepsilon_\PM \setminus \partial \Omega_\PM$, we have
\begin{align}\label{eq:Poincare-estimate}
    \norm{ \varphi}_{L^2(\Omega_\PM^\varepsilon)} \leq & \ C \varepsilon \norm{\nabla \varphi}_{L^2(\Omega_\PM^\varepsilon)^2} \, , 
    \quad
    \norm{ \varphi}_{L^2(\Sigma)} \leq \ C \varepsilon^{1/2} \norm{\nabla \varphi}_{L^2(\Omega_\PM^\varepsilon)^2} \, .
\end{align}

% \bmsection{Cell problems and boundary layer problems}
% \label{appendix:cell-and-boundary-layer-problems}
% \vspace*{12pt}

\section{Cell problems and boundary layer problems\label{appendix:cell-and-boundary-layer-problems}}

In this section, we provide the cell and boundary layer problems that were used in~\citep{Eggenweiler_Rybak_MMS20} to construct the pore-scale approximation~$\{\vecee v^{6,\varepsilon}_\app, p^{6,\varepsilon}_\app\}$ given by~\eqref{eq:approx-6-MMS-v},~\eqref{eq:approx-6-MMS-p}.

The cell problems are defined in the fluid part $Y_\text{f}$ of the unit cell $Y=(0,1)\times (0,1)$ presented in figure~\ref{fig:setting} (middle).
The boundary layer problems are defined on the boundary layer stripe $Z^\bl = Z^+ \cup S \cup Z^-$ (figure~\ref{fig:setting}, right) that is introduced in section~\ref{sec:elimination-D}.
The most important property of boundary layer solutions is their exponentially fast stabilization to some boundary layer constants for $|y_2| \to \infty$. For details, we refer the reader to~\cite[section 3]{Jaeger_Mikelic_96} where this property was proven. % for boundary layer type problems.
%Thus, we do not present the problems for $\vecee \zeta^{\bl}$ and $\vecee q^{j,\bl}$ in this manuscript.

%\renewcommand{\thesubsection}{\appendixname\ \Alph{subsection}}
\subsection{Cell problems}
As a part of the approximation of the pore-scale solution $\{\vecee v^\varepsilon, p^\varepsilon\}$ in section~\ref{sec:derivation}, we consider the following cell problems for $j=1,2$: 
\begin{equation}\label{eq:cell-problems}
    \begin{aligned}
    - \Delta_ {{\vecee y}} \vecee{w}^{j}  &+ \nabla_{{\vecee y}} \pi^j = \vecee 0 \quad \text{in 
    $Y_\text{f}$} \, , 
    \quad 
    \divy \vecee{w}^{j}  =0 \quad \text{in 
    $Y_\text{f}$} \, , 
    \quad \int_{Y_\text{f}} \pi^j \ \text{d} {\vecee y} =0 \, ,
    \\
    \vecee{w}^{j} &=  \vecee{0} \quad \text{on $\partial Y_\text{f} \setminus \partial Y$} \, ,  \quad
    \{ \vecee{w}^{j}, \pi^j \} \text{ is 1-periodic} \text{ in } {\vecee y} \, .
    \end{aligned}
\end{equation}
Problems~\eqref{eq:cell-problems} are well-known in homogenization theory~\citep{Hornung_97,Jaeger_Mikelic_96,Allaire_89} and were also applied in~\cite[eq. (3.3)]{Eggenweiler_Rybak_MMS20} for the pore-scale approximation~$\{\vecee v^{6,\varepsilon}_\app, p^{6,\varepsilon}_\app\}$.
As usual, the velocity is extended to zero in the solid inclusions, i.e., $\vecee w^{j,\varepsilon} = \vecee 0$ on $\Omega_\PM \setminus \Omega_\PM^\varepsilon$. 
The following bounds have been proven in~\cite[eqs. (1.11), (1.12)]{Jaeger_Mikelic_96}:
\begin{align}\label{eq:estimates-cell-problems}
    \| \pi^{j,\varepsilon}\|_{L^2(\Omepspm)} &\leq C \, ,
     \qquad 
     \| \vecee w^{j,\varepsilon}\|_{L^2(\Omepspm)^2} \leq C \, .
\end{align}

\subsection{Cell problems~\cite[eq. (3.42)]{Eggenweiler_Rybak_MMS20}\label{app2.a}}

In order to correct compressibility effects of the velocity error function in~\cite[eq. (3.41)]{Eggenweiler_Rybak_MMS20} the following cell problems are used 
\begin{equation}
\begin{aligned}
\divy \vecee \gamma^{j,i}  &= w_i^j - \frac{k_{ij}}{|Y_\text{f}|} \quad \text{in } Y_\text{f},
\\
 \vecee \gamma^{j,i} = \vecee 0 \quad \text{on } \partial Y_\text{f} \setminus \partial Y,& \quad \vecee \gamma^{j,i} \text{ is 1-periodic in $\vecee y$}.
\end{aligned}\label{eq:cell-problems-gamma}
\end{equation}
Moreover,  $\vecee \gamma^{j,i, \varepsilon}(\vecee x) = \varepsilon \vecee \gamma^{j,i} (\vecee y)$ is set for $\vecee x \in \Omega_\PM^\varepsilon$ and $\vecee \gamma^{j,i, \varepsilon}$ is extended by zero in $\Omega_\PM \setminus \Omega_\PM^\varepsilon$. After~\cite[eqs. (1.21),~(1.22)]{Jaeger_Mikelic_96} we know that
\begin{align}
    \| \vecee \gamma^{j,i, \varepsilon} \|_{L^2(\Omepspm)^2} \leq C \varepsilon \, , 
    \qquad 
    \| \nabla \vecee \gamma^{j,i, \varepsilon} \|_{L^2(\Omepspm)^{2 \times 2}} \leq C  \, .
    \label{eq:estimates-gamma-cell}
\end{align}

\subsection{Boundary layer problem~\cite[eq. (3.14)]{Eggenweiler_Rybak_MMS20}}% correcting approximation in the free-flow region}
The first boundary layer problem introduced in~\cite[eq. (3.14)]{Eggenweiler_Rybak_MMS20} corrects the pore-scale approximation in the free-flow region and is given by
\begin{equation}\label{eq:BLP-t}
    \begin{aligned}
    -\Delta_{\vecee y} \vecee t^{\bl} + \nabla_{\vecee y} s^{\bl} &= \vecee0 \quad \text{in } Z^{+}  \cup Z^{-} \, , 
    \\
    \divy \vecee t^{\bl} &= 0 \quad \text{in } Z^{+}  \cup Z^{-} \, ,  
    \\ 
    \big\llbracket \vecee t^{\bl} \big\rrbracket_S &= \vecee 0 \quad \text{on } S \, , 
    \\
    \big\llbracket ( \nabla_{\vecee y} \vecee t^{\bl} - s^{\bl} \ten{I} ) \vecee{e}_2\big\rrbracket_S &=  \vecee{e}_1 \hspace{1.5ex} \text{on } S \, , 
    \\
    \vecee t^{\bl} = \vecee 0 \quad \text{on } \cup_{k=1}^{\infty}(\partial Y_\text{s} - (0,k)) \, ,& \quad \{\vecee t^{\bl}, s^{\bl}\} \text{ is 1-periodic in $y_1$} \, . 
    \end{aligned} 
\end{equation}
The boundary layer constants to which the velocity $\vecee t^{\bl}$ and pressure $s^{\bl}$ stabilize exponentially fast for $y_2 \to \infty$ are defined as follows
\begin{align}\label{eq:BLP-constants-N}
    \vecee N^{\bl} &= \left(N_1^{\bl}, 0\right)^\top = \bigg(\int_S  t_1^{\bl}(y_1, +0) \ \text{d} y_1, 0\bigg)^\top \, , 
    \qquad
    N_{s}^{\bl} = \int_S s^{\bl}(y_1, +0) \ \text{d} y_1 \, .
\end{align}
In \cite[eqs. (1.67)--(1.69)]{Jaeger_Mikelic_96} it is shown that the following estimates hold
\begin{equation}\label{eq:estimates-t}
\begin{aligned}
     \| \vecee t^{\bl, \varepsilon} - \mathcal{H}(x_2)\vecee N^\bl \|_{L^2(\Omega)^2} &\leq C \varepsilon^{1/2}  \, ,
     \qquad 
     \| \nabla \vecee t^{\bl, \varepsilon} \|_{L^2(\Omega)^{2\times 2}} \leq C \varepsilon^{-1/2} \, ,
     \\
     \| s^{\bl, \varepsilon} - \mathcal{H}(x_2)N_s^\bl \|_{L^2(\Omega^\varepsilon)} &\leq C \varepsilon^{1/2}  \, .
\end{aligned}
\end{equation}
For further details, regularity results, and estimates, we refer to~\cite[section 3.2.2]{Eggenweiler_Rybak_MMS20} and~\cite[section 1.2.7]{Jaeger_Mikelic_96}.

\subsection{Boundary layer problem~\cite[eq. (3.25)]{Eggenweiler_Rybak_MMS20}}% establishing velocity trace continuity}
The second boundary layer correctors that are used in the pore-scale approximation~$\{\vecee v^{6,\varepsilon}_\app, p^{6,\varepsilon}_\app\}$ are obtained from the solutions to the following problem
\begin{equation}\label{eq:BLP-beta}
    \begin{aligned}
    -\Delta_{\vecee y} \vecee \beta^{j,\bl} + \nabla_{\vecee y} \omega^{j,\bl} &= \vecee0  \quad \text{in } Z^{+}  \cup Z^{-} \, , 
    \\ 
    \divy \vecee \beta^{j,\bl} &= 0 \quad \text{in } Z^{+}  \cup Z^{-} \, , 
    \\
    \big\llbracket \vecee \beta^{j,\bl} \big\rrbracket_S &= k_{2j}\vecee{e}_2 - \vecee w^j \quad \text{on } S \, , 
    \\
    \big\llbracket ( \nabla_{\vecee y} \vecee \beta^{j,\bl} - \omega^{j,\bl} \ten{I} ) \vecee{e}_2\big\rrbracket_S &= - \left( \nabla_{\vecee y} \vecee w^{j} - \pi^j \ten I \right) \vecee{e}_2 \quad \text{on } S \, , 
    \\
    \vecee \beta^{j,\bl} = \vecee 0 \quad \text{on } \cup_{k=1}^{\infty}(\partial Y_\text{s} - (0,k)) \, ,& \qquad \{\vecee \beta^{j,\bl}, \omega^{j,\bl}\} \text{ is 1-periodic in $y_1$} \, . 
    \end{aligned}
\end{equation}
The corresponding boundary layer constants are given by
\begin{align}\label{eq:BLP-constants-M}
    \vecee{M}^{j,\bl}= \left(M_1^{j,\bl}, 0\right)^\top &= \bigg(\int_S \beta_1^{j, \bl}(y_1, +0) \ \text{d} y_1, 0\bigg)^\top \, ,  
    \quad \
    M_{\omega}^{j,\bl} = \int_S \omega^{j,\bl}(y_1, +0) \ \text{d} y_1  \, . 
\end{align}
For further details on the construction and properties of the solution $\{\vecee \beta^{j,\bl}, \omega^{j,\bl}\}$, we refer to~\cite[section 3.2.3]{Eggenweiler_Rybak_MMS20} and~\cite[section 1.2.3]{Jaeger_Mikelic_96}.

\subsection{Boundary layer problem~\cite[eq. (3.32)]{Eggenweiler_Rybak_MMS20}}% correcting the velocity on the lower boundary}
The boundary layer problem correcting the values of the velocity error function and its gradient on the lower boundary reads
\begin{equation}\label{eq:BLP-q}
    \begin{aligned}
    -\Delta_{\vecee y} \vecee q^{j,\bl} + \nabla_{\vecee y} z^{j,\bl} &= \vecee 0  \quad \text{in } Z^{-} \, , 
    \\ 
    \divy \vecee q^{j,\bl} &= 0 \quad \text{in }  Z^{-} \, , 
    \\
    q_2^{j,\bl}  &= k_{2j} - w_2^j \quad \text{on } S \, , 
    \\
    \frac{\partial q_1^{j,\bl}}{\partial y_2} &= - \frac{\partial w_1^{j}}{\partial y_2} \quad \hspace*{2.5ex} \text{on } S \, , 
    \\
    \vecee q^{j,\bl} = \vecee 0 \quad \text{on } \cup_{k=1}^{\infty}(\partial Y_\text{s} - (0,k)) \, ,& \qquad \{\vecee q^{j,\bl}, z^{j,\bl}\} \text{ is 1-periodic in $y_1$} \, . 
    \end{aligned}
\end{equation}
Following the theory from~\cite{Jaeger_Mikelic_96} and the results from~\cite{Carraro_etal_15}, problem~\eqref{eq:BLP-q} is of boundary layer type. 
% %Existence and uniqueness of the solution $\{ \vecee q^{j,\bl}, z^{j,\bl}\}$ can be proven as in~\cite{Jaeger_Mikelic_96}:
% It admits a unique solution $\vecee q^{j,\bl} \in H^1(Z^-)^2 $, smooth in $Z^-$. Furthermore, there exists $\gamma > 0$ such that $e^{\gamma |y_2 |} \vecee q^{j,\bl} \in L^2(Z^-)^2$ and, after adjusting a constant, $e^{\gamma |y_2 |} z^{j,\bl} \in L^2(Z^-)$. 
%Moreover, the following estimates hold
Thus, we have
\begin{align}
    \| \vecee q^{j,\bl,\varepsilon} \|_{L^2(\Omeps)^2} \leq C \varepsilon^{1/2} \, , 
    \quad 
    \| z^{j,\bl,\varepsilon} \|_{L^2(\Omeps)} \leq C \varepsilon^{1/2} \, .
    \label{eq: estimates-q-z}
\end{align}

\subsection{Boundary layer problem~\cite[eq. (3.43)]{Eggenweiler_Rybak_MMS20}}
To correct the divergence of the velocity error function, the following problem is used
\begin{equation}
\begin{aligned}
    -\Delta_{\vecee y} \vecee \gamma^{j,i,\bl} + \nabla_{\vecee y} \pi^{j,i,\bl} &= \vecee0 \quad \text{in } Z^{+}  \cup Z^{-} \, , 
    \\
    \divy \vecee \gamma^{j,i,\bl } &= 0 \quad \text{in } Z^{+}  \cup Z^{-}, 
    \\ 
    \big\llbracket \vecee \gamma^{j,i,\bl} \big\rrbracket_S &= \vecee \gamma^{j,i} \hspace{6.75ex} \text{on } S \, ,
    \\
    \big\llbracket ( \nabla_{\vecee y} \vecee \gamma^{j,i,\bl} - \pi^{j,i,\bl} \ten{I} ) \vecee{e}_2\big\rrbracket_S &=  \nabla_{\vecee y } \vecee \gamma^{j,i} \vecee e_2  \quad \text{on } S \, ,
    \\
    \vecee \gamma^{j,i,\bl} = \vecee 0 \quad \text{on } \cup_{k=1}^{\infty}(\partial Y_\text{s} - (0,k)),& \quad \{\vecee \gamma^{j,i,\bl}, \pi^{j,i,\bl}\} \text{ is 1-periodic in $y_1$} \, .
\end{aligned}  \label{eq:BLP-gamma}
\end{equation}
We set $\vecee \gamma^{j,i,\bl,\varepsilon}(\vecee x) = \varepsilon \vecee \gamma^{j,i,\bl}(\vecee y), \ \pi^{j,i,\bl,\varepsilon}(\vecee x) = \pi^{j,i,\varepsilon}(\vecee y)\text{ for } \vecee x \in \Omega^\varepsilon
$
and extend $\vecee \gamma^{j,i,\bl,\varepsilon}$ by zero for $\vecee x \in \Omega \setminus \Omega^\varepsilon$. 
The corresponding boundary layer constants
are given by
\begin{align}\label{eq:BLP-constants-gamma}
    \vecee{C}^{j,i,\bl} &= \bigg(\int_S \gamma_1^{j,i, \bl}(y_1, +0) \ \text{d} y_1, 0\bigg)^\top \, ,  
    \quad \
    C_{\pi}^{j,i,\bl} = \int_S \pi^{j,i,\bl}(y_1, +0) \ \text{d} y_1  \, . 
\end{align}
Estimates and further properties of the solution to problem~\eqref{eq:BLP-gamma} are presented in~\citep{Eggenweiler_Rybak_MMS20,Jaeger_Mikelic_96}.

\subsection{Boundary layer problem~\cite[eq. (3.44)]{Eggenweiler_Rybak_MMS20}}% correcting the compressibility effects}
Boundary layer problem~(3.44) from~\citep{Eggenweiler_Rybak_MMS20} also corrects compressibility effects of the velocity error. %coming from $t_1^{\bl,\varepsilon}$ 
This problem reads
\begin{equation}
    \begin{aligned} \label{eq:BLP-zeta}
    \divy \vecee \zeta^{\bl} &=  t_1^{\bl} (\vecee y) - \mathcal{H}(x_2)N_1^{\bl} \quad \text{ in } Z^+ \cup Z^- \, ,
    \\
    \llbracket \vecee \zeta^{\bl} \rrbracket_S &= - \left( \int_{Z^{\bl}}  (t_1^{\bl} (\vecee y) - \mathcal{H}(x_2)N_1^{\bl}) \ \text{d} \vecee y \right) \vecee e_2 \quad \text{ on } S \, , 
    \\
    \vecee \zeta^{\bl} &= \vecee 0 \quad \text{ on } \cup_{k=1}^{\infty}(\partial Y_\text{s} - (0,k)), \quad \vecee \zeta^{\bl} \text{ is 1-periodic in $y_1$} \, .
    \end{aligned}
\end{equation}
After~\cite[Proposition 3.20]{Jaeger_Mikelic_96} problem~\eqref{eq:BLP-zeta} has at least one solution $\vecee \zeta^{\bl} \in H^1(Z^+ \cup Z^-)^2 \cap C_\text{loc}(Z^+ \cup Z^-)^2$. 
%Note on uniqueness here!!
%Note that uniqueness of a solution cannot be guaranteed. 
%
We introduce the boundary layer constant $W^\bl$ by
\begin{align}\label{eq:BLP-constant-W}
\llbracket \vecee \zeta^{\bl} \rrbracket_S = - \left(\int_{Z^{\bl}} (t_1^{\bl} - \mathcal{H}(x_2) N_1^{\bl}) \, \text{d} \vecee y  \right) \vecee e_2
= - \left( \int_{Z^-} t_1^{\bl} \, \text{d} \vecee y \right) \vecee e_2 =\colon W^\bl \vecee e_2 
%\quad \text{ on } S 
\, .
\end{align}
%Using the first equation in~\eqref{eq:BLP-zeta}, we get
Furthermore, we have
\begin{align}
    \llbracket \nabla_{\vecee y} \vecee \zeta^{\bl} \vecee e_2 \rrbracket_S
    = \bigg\llbracket\frac{\partial}{\partial y_2} \vecee \zeta^{\bl}   \bigg\rrbracket_S  
    % &= \bigg\llbracket  \begin{pmatrix}
    %     \frac{\partial}{\partial y_2}\zeta_1^{\bl} \\
    %     t_1^{\bl} - \mathcal{H}(x_2)N_1^{\bl} -\frac{\partial}{\partial y_1}\zeta_1^{\bl}
    % \end{pmatrix} \bigg\rrbracket_S 
    % \notag
    % \\
    % &= - N_1^{\bl}\vecee e_2 + \bigg\llbracket  \begin{pmatrix}
    %     \frac{\partial}{\partial y_2}\zeta_1^{\bl} \\
    %       -\frac{\partial}{\partial y_1}\zeta_1^{\bl}
    % \end{pmatrix} \bigg\rrbracket_S
    % \notag
    % \\
    &= - N_1^{\bl}\vecee e_2 + \big\llbracket \operatorname{curl} \zeta_1^{\bl}
    \big\rrbracket_S \, ,
\end{align}
where we applied~\eqref{eq:BLP-zeta} %for the second equality
and used continuity of boundary layer velocity from~\eqref{eq:BLP-t} across the interface. % $\llbracket t_1^\bl \rrbracket_S = 0$ (see~\eqref{eq:BLP-t}).
Now, we make use of the construction of such boundary layer correctors (see~\cite[Proposition 3.20]{Jaeger_Mikelic_96}).
We know that $\vecee \zeta^{\bl} = \nabla \nu + \operatorname{curl} \theta$, where $\nu$ is the unique (up to a constant) solution to a problem in analogy to (3.69) from~\cite{Jaeger_Mikelic_96} and $\theta$ is the solution to a problem analogous to~(3.71) in the same reference~\citep{Jaeger_Mikelic_96}.
%For $\theta$ equation (3.71) from~\citep{Jaeger_Mikelic_96} holds. 
These facts lead to
\begin{align}
    \llbracket \nabla_{\vecee y} \vecee \zeta^{\bl} \vecee e_2 \rrbracket_S
    %= \bigg\llbracket\frac{\partial}{\partial y_2} \vecee \zeta^{\bl}   \bigg\rrbracket_S  
    &= - N_1^{\bl}\vecee e_2 
    \quad \text{ on } S \, .
    \label{eq:zeta-gradient-jump}
\end{align}
Moreover, setting $\vecee \zeta^{\bl,\varepsilon}(\vecee x) = \vecee \zeta^{\bl} (\vecee y)$ for $\vecee x \in \Omeps$ and extending the velocity by zero in the solid inclusions, we know from~\citep{Jaeger_Mikelic_96}  that
\begin{align}\label{eq:estimates-zeta}
     \| \vecee \zeta^{\bl,\varepsilon} \|_{L^2(\Omega)^2} \leq C \varepsilon^{1/2} \, .
\end{align}

\subsection{Boundary layer problem~\cite[eq. (3.49)]{Eggenweiler_Rybak_MMS20}}
The last boundary layer problem for the correction of compressibility effects from~\citep{Eggenweiler_Rybak_MMS20} has the following form
\begin{equation}\label{eq:BLP-Z}
\begin{aligned}
     \divy \vecee Z^{j,\bl} &=  q_1^{j,\bl}\quad \text{in }  Z^- ,
     \\
     \llbracket 
     \vecee Z^{j,\bl}  \rrbracket_S 
     &= - \left( \int_{Z^-} q_1^{j,\bl} \ \text{d} \vecee y\right) \vecee e_2 \quad  \text{on } S ,\\
     \vecee Z^{j,\bl} &= \vecee 0 \quad \text{ on } \cup_{k=1}^{\infty}(\partial Y_\text{s} - (0,k)), \quad \vecee Z^{j,\bl} \text{ is 1-periodic in $y_1$}.
\end{aligned}
\end{equation}
The existence of a solution to~\eqref{eq:BLP-Z}, its exponential decay, %can be proven as in~\citep{Jaeger_Mikelic_96}, 
and the usual $L^2$-estimates for boundary layer velocity %can be obtained.
follow from the theory developed by~\cite{Jaeger_Mikelic_96}.

\section{Reformulation of interface conditions for numerical simulations}
\label{appendix:reformulated-IC}
In section~\ref{sec:numerics}, where numerical simulation results are presented, we consider an isotropic porous medium. Recall that in this case the constants $N_s^\bl=M_\omega^{1,\bl}=M_1^{2,\bl}=E_1^\bl=L_1^\bl=0$. 
%Thus, neglecting terms of order $\varepsilon^3$ for the velocity and $\varepsilon^2$ for the pressure, we
We modify the higher-order interface conditions~\eqref{eq:IC-HO-normal}--\eqref{eq:IC-HO-tangential} as follows
\begin{eqnarray}
    v_2^\FF     &=&  v_2^\PM  
    +\varepsilon \frac{W^\bl}{N_1^\bl} \frac{\partial v_1^\FF}{\partial x_1}\bigg|_\Sigma 
    \, ,
    \label{eq:IC-HO-normal-reform}
    \\
    p^\PM &=& p^\FF - \frac{\partial v_2^\FF}{\partial x_2}\bigg|_\Sigma 
    - \varepsilon M_\omega^{2,\bl} \frac{\partial p^\PM }{\partial x_2}\bigg|_\Sigma
    - \frac{1}{N_1^\bl} \left( L_\eta^\bl
     + E_b^\bl +  N_1^\bl \right)  \frac{\partial v_1^\FF}{\partial x_1}\bigg|_\Sigma 
     \, ,
     \label{eq:IC-HO-p-reform}
     \\
     v_1^\FF &=&  - \varepsilon N_1^\bl  \frac{\partial v_1^\FF}{\partial x_2}\bigg|_\Sigma 
    +\varepsilon^2 M_1^{1,\bl} \frac{\partial p^\PM }{\partial x_1}\bigg|_\Sigma 
    \label{eq:IC-HO-tangential-reform}
    \, .
\end{eqnarray}
Condition~\eqref{eq:IC-HO-tangential-reform} is condition~\eqref{eq:IC-HO-tangential} for $M_1^{2,\bl}=E_1^\bl=L_1^\bl=0$. 
We use~\eqref{eq:IC-HO-tangential-reform} to replace $\partial v_1^\FF / \partial x_2|_\Sigma$  in conditions~\eqref{eq:IC-HO-normal} and~\eqref{eq:IC-HO-p}. Then, we neglect terms of order $\varepsilon^3$ for the velocity and $\varepsilon^2$ for the pressure. As a result, we obtain interface conditions~\eqref{eq:IC-HO-normal-reform} and~\eqref{eq:IC-HO-p-reform} which we implemented in our \textsc{FreeFEM++} code.
%In our \textsc{FreeFEM++} code, we use the higher-order interface conditions for the Stokes--Darcy coupling in the form~\eqref{eq:IC-HO-normal-reform}--\eqref{eq:IC-HO-tangential-reform}.
%These conditions have the advantage in comparison to~\eqref{eq:IC-HO-normal}--\eqref{eq:IC-HO-tangential} that no second-order derivatives w.r.t. the free-flow velocity $v_1^\FF$ appear. This simplifies the treatment of the resulting coupled Stokes--Darcy system from the numerical perspective.
Although we eliminated some of the higher-order terms from~\eqref{eq:IC-HO-normal}--\eqref{eq:IC-HO-tangential} to obtain~\eqref{eq:IC-HO-normal-reform}--\eqref{eq:IC-HO-tangential-reform}, we note that the modified conditions~\eqref{eq:IC-HO-normal-reform}--\eqref{eq:IC-HO-tangential-reform} still incorporate higher-order contributions.
We also demonstrated that these modified interface conditions provide more accurate simulation results for Stokes--Darcy problems than existing coupling concepts.

\renewcommand{\thesection}{\appendixname\ \Alph{section}}

\section*{Acknowledgements}
\vspace{-1ex}
\noindent The authors thank Joscha Nickl for valuable discussions related to boundary layer theory. Moreover, the authors express their gratitude to U\v{g}is L\={a}cis for his helpful tips on implementing the macroscale coupled Stokes--Darcy models in \textsc{FreeFEM++}.
\section*{Funding}
\vspace{-1ex}
\noindent The work is funded by the Deutsche Forschungsgemeinschaft (DFG, German Research Foundation) -- Project Number 490872182 and Project Number 327154368~-- SFB 1313.
\vspace{-2ex}
\section*{Data availability}
\vspace{-1ex}
\noindent The datasets generated and analysed during the current study are 
available from the corresponding author on reasonable request.
\vspace{-2ex}
\section*{Declaration of interest}
\vspace{-1ex}
\noindent The authors have no competing interests to declare that are relevant to the content of this paper.
\vspace{-2ex}
\section*{Declaration of generative AI and AI-assisted technologies in the writing process}
\vspace{-1ex}
\noindent During the preparation of this work, the authors did not use any generative AI and AI-assisted technologies.

%\bibliographystyle{elsarticle-num} 
%\bibliography{references}

\end{document}